\pdfoutput=1
\documentclass{scrartcl}
\usepackage[intlimits]{amsmath}

\usepackage{amsthm,amssymb,bm,ifpdf,authblk,extpfeil}
\usepackage{accents}
\usepackage{subcaption}
\usepackage{url}
\usepackage{longtable}
\usepackage[abbrev,nobysame,lite]{amsrefs}
\usepackage[T1]{fontenc}
\usepackage{lmodern}
\usepackage[inline]{enumitem}
\usepackage{microtype}
\usepackage[dvipsnames]{xcolor}
\usepackage{tikz}
\usetikzlibrary{shapes}
\usepackage[all]{xy}
\usepackage{tocstyle}
\usepackage{hyperref}

\definecolor{ourdark}{rgb}{0.0, 0.1, 0.2}
\definecolor{ouryellow}{rgb}{0.85, 0.85, 0}
\definecolor{ourgreen}{rgb}{0.0, 0.5, 0.9}

\definecolor{our0}{rgb}{0.0, 0.0, 0.0}
\definecolor{our1}{rgb}{0.2, 0.2, 0.7}
\definecolor{our2}{rgb}{0.4, 0.4, 0.4}
\definecolor{our3}{rgb}{0.4, 0.6, 0.4}
\definecolor{our4}{rgb}{0.85, 0.85, 0.4}
\definecolor{our5}{rgb}{1, 1, 1}

\allowdisplaybreaks[1]

\newtheorem{thm}{Theorem}[section]
\newtheorem{lemma}[thm]{Lemma}
\newtheorem{prop}[thm]{Proposition}

\newtheorem{qu}[thm]{Question}
\newtheorem{cor}[thm]{Corollary}
\newtheorem{thmabc}{Theorem}
\newtheorem{corabc}[thmabc]{Corollary}

\theoremstyle{definition}
\newtheorem{ex}[thm]{Example}
\newtheorem*{ex*}{Example}

\newtheorem{defn}[thm]{Definition}
\newtheorem*{defn*}{Definition}

\newtheorem{rem}[thm]{Remark}
\newtheorem*{rem*}{Remark}

\numberwithin{equation}{section}
\allowdisplaybreaks[1]

\makeatletter
\DeclareOldFontCommand{\rm}{\normalfont\rmfamily}{\mathrm}
\DeclareOldFontCommand{\sf}{\normalfont\sffamily}{\mathsf}
\DeclareOldFontCommand{\bf}{\normalfont\bfseries}{\mathbf}
\DeclareOldFontCommand{\it}{\normalfont\itshape}{\mathit}
\makeatother

\renewcommand{\le}{\leqslant}
\renewcommand{\ge}{\geqslant}

\def\emptyset{\varnothing}

\def\emph{}
\DeclareTextFontCommand{\bfemph}{\bf}
\DeclareTextFontCommand{\itemph}{\it}
\def\emph{\bfemph}

\makeatletter
\def\blankfootnote{\xdef\@thefnmark{}\@footnotetext}

\newcommand*{\textlabel}[2]{%
  \edef\@currentlabel{#1}
  \phantomsection
  #1\label{#2}
}
\makeatother

\DeclareMathOperator{\infl}{inf}
\DeclareMathOperator{\resn}{res}

\DeclareMathOperator{\CG}{K}
\DeclareMathOperator{\DG}{\Delta}

\newcommand{\blank}{\ensuremath{\square}}
\newcommand{\Colours}{\ensuremath{\mathcal{C}}}
\newcommand{\Coloursb}{\ensuremath{\Colours \cup\{\blank\}}}

\DeclareMathOperator{\Orbit}{Orbit}
\DeclareMathOperator{\orbit}{\mathsf{orbit}}

\newcommand{\Rho}{\ensuremath{\mathsf{P}}}
\newcommand{\bRho}{\ensuremath{\bm{\Rho}}}
\newcommand{\bGamma}{\ensuremath{\bm{\mathsf{\Gamma}}}}
\newcommand{\bTheta}{\ensuremath{\bm{\mathsf{\Theta}}}}
\newcommand{\bSigma}{\ensuremath{\bm{\mathsf{\Sigma}}}}

\DeclareMathOperator{\inc}{\mathsf{inc}}
\DeclareMathOperator{\ret}{\mathsf{ret}}

\newcommand{\onto}{\twoheadrightarrow}
\newcommand{\into}{\rightarrowtail}
\newcommand{\incl}{\hookrightarrow}
\newcommand{\xonto}{\xtwoheadrightarrow}

\newcommand{\xto}{\xrightarrow}

\newcommand{\cc}{\ensuremath{\mathsf{cc}}}

\newcommand{\ak}{\ensuremath{\mathsf{ask}}}

\newcommand{\join}{\ensuremath{\vee}}

\DeclareMathOperator{\Powf}{\mathcal{P}_{\mathrm{f}}}

\let\AA\undefined
\newcommand{\RF}{\mathfrak{K}}
\newcommand{\AA}{\mathbf{A}}
\newcommand{\QQ}{\mathbf{Q}}
\newcommand{\FF}{\mathbf{F}}
\newcommand{\NN}{\mathbf{N}}
\newcommand{\ZZ}{\mathbf{Z}}
\newcommand{\CC}{\mathbf{C}}

\newcommand{\UU}{\mathbf{U}}
\newcommand{\fa}{\ensuremath{\mathfrak a}}
\newcommand{\fg}{\ensuremath{\mathfrak g}}
\newcommand{\Zeta}{\ensuremath{\mathsf{Z}}}
\newcommand{\bone}{\ensuremath{\mathbf{1}}}
\newcommand{\PID}{\textsf{\textup{PID}}}
\newcommand{\DVR}{\textsf{\textup{DVR}}}
\newcommand{\fO}{\mathfrak{O}}
\newcommand{\fP}{\mathfrak{P}}
\newcommand{\cA}{\mathcal{A}}
\newcommand{\cB}{\mathcal{B}}
\newcommand{\cD}{\mathcal{D}}
\newcommand{\cF}{\mathcal{F}}
\newcommand{\cG}{\mathcal{G}}
\newcommand{\cJ}{\mathcal{J}}
\newcommand{\cN}{\mathcal{N}}
\newcommand{\sA}{\mathsf{A}}
\newcommand{\sB}{\mathsf{B}}
\newcommand{\sC}{\mathsf{C}}
\newcommand{\sD}{\mathsf{D}}
\newcommand{\sE}{\mathsf{E}}
\newcommand{\sF}{\mathsf{F}}
\newcommand{\sG}{\mathsf{G}}
\newcommand{\sL}{\mathsf{L}}
\newcommand{\sS}{\mathsf{S}}
\newcommand{\cZ}{\mathcal{Z}}

\DeclareMathOperator{\Board}{Rel}
\DeclareMathOperator{\AltBoard}{ARel}
\DeclareMathOperator{\SymBoard}{SRel}

\newlength{\dhatheight}

\newcommand{\sslash}{\mathbin{/\mkern-6mu/}}

\DeclareMathOperator{\concnt}{k}
\DeclareMathOperator{\Alt}{Alt} 
\DeclareMathOperator{\Sym}{Sym}

\DeclareMathOperator{\GL}{GL}

\newcommand{\XX}{\bm{X}}

\newcommand{\Sl}{\ensuremath{\mathfrak{sl}}}
\newcommand{\tr}{\ensuremath{\mathfrak{tr}}}

\newcommand{\stacks}[1]{{\cite[\href{http://stacks.math.columbia.edu/tag/#1}{Tag #1}]{stacks-project}}}

\DeclareMathOperator{\cent}{c}

\DeclareMathOperator{\Ker}{Ker}
\DeclareMathOperator{\Coker}{Coker}
\DeclareMathOperator{\Zent}{Z}
\DeclareMathOperator{\Fit}{Fit}

\DeclareMathOperator{\Hom}{Hom}
\DeclareMathOperator{\Mat}{M}
\DeclareMathOperator{\ad}{ad}
\DeclareMathOperator{\dd}{d\!}
\newcommand{\normal}{\triangleleft}
\newcommand{\dtimes}{\ensuremath{\cdot}}
\newcommand{\card}[1]{\lvert#1\rvert}

\DeclarePairedDelimiter{\abs}{\lvert}{\rvert}

\DeclareMathOperator{\rank}{rk}

\DeclareMathOperator{\Real}{Re}

\newcommand{\llb}{\ensuremath{[\![ }}
\newcommand{\rrb}{\ensuremath{]\!] }}

\newcommand{\std}{\ensuremath{\mathsf e}}
\newcommand{\ask}[1]{\operatorname{ask}({#1})}

\ifpdf
  \hypersetup{pdftitle={Linear relations with disjoint supports and average sizes of kernels},
    pdfauthor={Angela Carnevale and Tobias Rossmann}
  }
\fi
\title{Linear relations with disjoint supports and average sizes of kernels}
\author{Angela Carnevale and Tobias Rossmann}
\date{}

\begin{document}

\maketitle

\vspace*{-5em}
\begin{center}
  \begin{tikzpicture}
    [box/.style={rectangle,draw=black,thick, minimum size=1cm},
    scale=.2, every node/.style={scale=.2}
    ]
    \foreach \x in {0,1,2,3}{
      \foreach \y in {0,1,2,3}
      \node[box,fill=white] at (\x,\y){};
    }
    \node[box,fill=BlueViolet] at (0,3){};  
    \node[box,fill=BlueViolet] at (1,2){};  
    \node[box,fill=BlueViolet] at (2,1){};  
    \node[box,fill=BlueViolet] at (3,0){};
    
    \node[box,fill=Green] at (1,3){};  
    \node[box,fill=Green] at (2,2){};  
    \node[box,fill=Green] at (3,1){};  
    \node[box,fill=Green] at (0,0){};
    
    \node[box,fill=YellowOrange] at (2,3){};  
    \node[box,fill=YellowOrange] at (3,2){};  
    \node[box,fill=YellowOrange] at (0,1){};  
    \node[box,fill=YellowOrange] at (1,0){};      

    \node[box,fill=Red] at (3,3){};  
    \node[box,fill=Red] at (0,2){};  
    \node[box,fill=Red] at (1,1){};  
  \end{tikzpicture}
\end{center}

\begin{abstract}
  \small
  We study the effects of imposing linear relations within modules of
  matrices on average sizes of kernels.
  The relations that we consider can be described combinatorially in terms of
  partial colourings of grids.
  The cells of these grids correspond to positions in matrices and each
  defining relation involves all cells of a given colour.
  We prove that imposing such relations arising from ``admissible'' partial colourings 
  has no effect on average sizes of kernels over finite quotients of discrete
  valuation rings.
  This vastly generalises the known fact that average sizes of kernels
  of general square and traceless matrices of the same size coincide over such rings.
  As a group-theoretic application, we explicitly determine zeta functions enumerating
  conjugacy classes of finite $p$-groups derived from free class-$3$-nilpotent
  groups for $p\ge 5$.
\end{abstract}

\blankfootnote{%
\phantom{.}
}

\tableofcontents

\section{Introduction}
\label{s:intro}

\subsection{Motivation: shapes, ranks, and average sizes of kernels of
  matrices}
\label{ss:motivation}

The starting point of the research described in this article is
the observation that four families of modules of matrices
share a number of curious features.
For a (commutative) ring $R$, let $\Mat_{d\times e}(R)$ denote  the module of all $d\times e$
matrices over $R$.
Further let $\Alt_d(R)$, $\Sym_d(R)$, and $\Sl_d(R)$ denote the modules of
alternating (i.e.\ antisymmetric with zeros along the diagonal), symmetric,
and traceless $d\times d$ matrices over $R$, respectively.
To describe the aforementioned common features,
let $T_d(R)\subset \Mat_d(R) := \Mat_{d\times d}(R)$ be one of the preceding
four types of modules of $d\times d$ matrices.
Then:
\begin{enumerate}[label={(\Alph*)}]
\item
  \label{motivation_simple}
  The module $T_d(R)$ is defined by ``simple'' 
  linear relations among matrix entries within the ambient module $\Mat_{d}(R)$.
  More precisely, we can choose defining linear relations
  (e.g.\ $x_{ij} - x_{ji} = 0$ or $x_{11}+\dotsb+x_{dd} = 0$) such that
  non-zero coefficients are units and
  different relations are supported on disjoint sets
  of matrix positions.
\item
  \label{motivation_rank}
  Taking $R$ to be a finite field $\FF_q$,
  the number of matrices in $T_d(\FF_q)$ of fixed rank $r$ is given by a polynomial
  in $q$ (which depends on $d$ and $r$).
  Moreover, this polynomial admits an explicit description in
  terms of permutation statistics on the Coxeter group $\mathrm B_d = \{\pm
  1\} \wr \mathrm S_d$ of signed permutations of $d$ letters.
  More generally, such permutation statistics can be used to express the
  number of matrices of given elementary divisor type in $T_d(R)$ when $R$ is
  a finite quotient of a discrete valuation ring (\DVR).
  Such results were first recorded by Stasinski and Voll~\cite{SV14}.
  For further related results, see \cite{BC17,CSV18} and references therein.
\item
  \label{motivation_ask}
  For a module of matrices $A$ over a finite ring, let $\ask A := \frac
  1{\card A}\sum_{a\in A}\card{\Ker(a)} \in \QQ$ be the
  \textbf average \textbf size of the \textbf kernel of
  an element of $A$.
  Then $\ask{T_d(\FF_q)}$ is given by a rather simple rational function in $q$.
  For example, Linial and Weitz~\cite{LW00} and, independently, Fulman and
  Goldstein~\cite{FG15} showed that
  $\ask{\Mat_{d}(\FF_q)}  =  2 - q^{-d}$.
  More generally, the average size of the kernel of an element of $T_d(R)$
  turns out to be well-behaved if $R$ is a finite quotient of a \DVR{};
  see \cite{ask,ask2}.
\end{enumerate}

\paragraph{Rank counts and average sizes of kernels.}
Since $\ask{T_d(\FF_q)}$ is expressible in terms of the numbers of
matrices of given rank in $T_d(\FF_q)$ (see \cite[\S 2.1]{ask}), one might
suspect that the simple shapes of $\ask{T_d(\FF_q)}$ in~\ref{motivation_ask}
could be explained combinatorially via~\ref{motivation_rank}.
Such an explanation has so far remained elusive; see \cite[\S 2.3]{ask}.
Instead, at present, \ref{motivation_ask} is perhaps best understood
using the formalism of ``ask zeta functions'' from \cite{ask,ask2},
sketched in \S\ref{ss:mini_ask} and discussed further in \S\ref{s:ask}.
By using a duality operation,
this point of view explains the simple shape of $\ask{T_d(\FF_q)}$ by relating
the spaces $T_d(\FF_q)$ to a classical topic: spaces of matrices of constant
rank. (See \cite[\S 5.3]{ask2}.) 

\paragraph{Rank counts: tame and wild.}
Beyond the sources cited above, various authors have studied the numbers
of matrices of given rank within combinatorially defined spaces
of matrices; see e.g.~\cite{LLMPSZ11}.
On the other hand,
while polynomiality results have been obtained in some cases,
Belkale and Brosnan~\cite{BB03} showed that the enumeration
of matrices of given rank over $\FF_q$ is ``arithmetically wild'' even for
seemingly simple spaces of symmetric matrices.
Indeed, they showed that, in a precise technical sense, the enumeration of
such matrices is as difficult as counting $\FF_q$-rational points of arbitrary
$\ZZ$-defined varieties.

\paragraph{Average sizes of kernels and support constraints.}
Average sizes of kernels of matrices within spaces and modules of generic,
alternating, or symmetric matrices defined by combinatorial support
constraints have been studied in \cite{cico}.
It turns out that while the aforementioned arithmetically wild behaviour
which is visible on the level of rank counts disappears entirely upon taking
the average,
a rich and intricate combinatorial structure governs the behaviour
of average sizes of kernels.
In particular, in the setting of \cite{cico}, average sizes of kernels are
usually far removed from the simplicity in~\ref{motivation_ask}.

\paragraph{}
In light of the above, the present authors regard it as remarkable that
\ref{motivation_simple}--\ref{motivation_ask} are simultaneously satisfied for
the modules $T_d(R)$ from above.
In fact, we are aware of only few such examples and of no systematic method
for constructing them.
This turns out to be due to the delicate nature
of rank distributions in spaces of matrices beyond~\ref{motivation_rank}.

In the present article, we construct large families
of modules of matrices defined via linear relations as in
\ref{motivation_simple}.
These modules are defined in terms of partial colourings of the cells
of suitable grids, as defined later in the paper.
Subject to admissibility conditions, we will show that average sizes
of kernels within these modules are as tame as for the $T_d(R)$
in \ref{motivation_ask}.
On the other hand, simple examples will show that there can be no analogue of
\ref{motivation_rank} for these modules.
Indeed, it is known (see \cite[Thm~4.11]{ask}) that ask zeta functions
associated with $\ZZ$-defined modules of matrices in general depend on
arithmetic properties of the \DVR{} in question---Example~\ref{ex:guide2} will
provide an explicit illustration of this in the present setting.

\subsection{Background: ask zeta functions}
\label{ss:mini_ask}

Before stating our main results,
we briefly recall basic
facts on ask zeta functions; we will give a more comprehensive account in
\S\ref{s:ask}.
Let $\fO$ be a compact \DVR{} with maximal ideal~$\fP$ and residue field
$\fO/\fP$ of cardinality $q$.
For example, $\fO$ might be the ring of $p$-adic integers~$\ZZ_p$ or the ring
of formal power series $\FF_q\llb t\rrb$.
Given a module $M\subset \Mat_{d\times e}(\fO)$, let~$M_n$ denote
its image in $\Mat_{d\times e}(\fO/\fP^n)$.
The \emph{ask zeta function} of $M$ is the generating function
\[
  \Zeta^\ak_M(T) := \sum_{n=0}^\infty \ask{M_n}T^n.
\]

These generating functions are closely related to the enumeration of orbits
and conjugacy classes of unipotent groups;
see \cite[\S 8]{ask}, \cite[\S\S 6--7]{ask2}, and \cite[\S 2.4]{cico}.
This connection forms the basis of the group-theoretic results in the
present paper, to be described in \S\ref{ss:adjoint}.

If $\fO$ has characteristic zero, then $\Zeta^\ak_M(T)\in \QQ(T)$
by \cite[Thm~1.4]{ask}.
As examples in \cite{ask,ask2,cico} illustrate, these rational functions can
be quite complicated, even for seemingly harmless and natural examples of
modules of matrices. 
On the other hand, $\Zeta^\ak_M(T)$ is occasionally of the simple
shape $\frac{1-q^aT}{(1-q^bT)(1-q^cT)} = 1 + (q^b + q^c - q^a)T + \mathcal O(T^2)$.
This is closely related to the informal simplicity of $\ask{T_d(\FF_q)}$
mentioned in \ref{motivation_ask} from \S\ref{ss:motivation} via the following.

\begin{prop}[{\cite[\S 5]{ask}}]
  \label{prop:classical_ask}
  Let $\fO$ be a compact \DVR{} with residue cardinality $q$.
  Then:
  \begin{enumerate}[label=(\roman{*})]
  \item
    \label{prop:classical_ask1}
    $\Zeta^{\ak}_{\Mat_{d\times e}(\fO)} =
    \frac{1-q^{-e}T}{(1-T)(1-q^{d-e}T)}$.
  \item
    \label{prop:classical_ask2}
    $\Zeta^{\ak}_{\Alt_d(\fO)} = \Zeta^{\ak}_{\Mat_{d\times(d-1)}(\fO)}(T) = \frac{1-q^{1-d}T}{(1-T)(1-qT)}$.
  \item
    \label{prop:classical_ask3}
    $\Zeta^{\ak}_{\Sym_d(\fO)} = \Zeta^{\ak}_{\Mat_d(\fO)}(T) =
    \frac{1-q^{-d}T}{(1-T)^2}$.
  \item
    \label{prop:classical_sl}
    If $d > 1$, then
    $\Zeta^{\ak}_{\Sl_d(\fO)}(T) = \Zeta^{\ak}_{\Mat_d(\fO)}(T) = \frac{1-q^{-d}T}{(1-T)^2}$.
  \end{enumerate}
\end{prop}

Note that we may recover $\ask{T_d(\FF_q)}$ 
in part \ref{motivation_ask} of \S\ref{ss:motivation} 
from Proposition~\ref{prop:classical_ask}.
For example, $\ask{\Alt_d(\FF_q)} = 1 + q - q^{1-d}$;
see \cite[Lemma 5.3]{FG15}.

\subsection{Linear relations from partial colourings}
\label{ss:intro_relations}

We now describe the three main types of modules of matrices that we
will consider.

\paragraph{Partial colourings of grids.}
Let $d,e \ge 1$.
We write $[d] = \{ 1,2,\dotsc,d\}$.
A \emph{partial colouring} of the \emph{grid} $[d]\times [e]$
is a family $\mathcal A = (A_c)_{c\in \Colours}$ of pairwise disjoint (possibly empty) subsets
of $[d] \times [e]$ indexed by a given set $\Colours$ of \emph{colours}.
Equivalently, $\mathcal A$ is the family of fibres of the elements of
$\Colours$ under a function $[d]\times [e]\to \Coloursb$.
Here and throughout this paper, $\blank$ is a fixed symbol (which we call \emph{blank}) that
does not belong to $\Colours$.
We refer to the elements of $[d]\times[e]$ as \emph{cells}.
A cell $(i,j)\in [d]\times [e]$ is \emph{blank} if it does not belong to any $A_c$.

\paragraph{Three types of modules.}
Let $R$ be a ring.
Given a partial colouring $\cA = (A_c)_{c\in \Colours}$ of $[d]\times [e]$ and
a $d\times e$ matrix $u$ whose entries are units of $R$, we define three
modules of matrices over~$R$.
First,
\[
  \Board_{d\times e}(\cA,u,R) := \left\{ [x_{ij}] \in \Mat_{d\times
      e}(R) :  \forall c\in \Colours, \sum\limits_{(i,j)\in A_c} u_{ij} x_{ij} = 0 \right\}.
\]
In other words, $\Board_{d\times e}(\cA,u,R)$ is obtained from $\Mat_{d\times e}(R)$
by imposing, for each colour, 
a linear relation among all entries whose positions are of that colour, and with
unit coefficients coming from the matrix $u$.

\begin{ex}
  \label{ex:sl}
  Fix a colour $\mathsf{blue} \in \Colours$ and let $\cA = (A_c)_{c\in
    \Colours}$ be the partial colouring of $[d]\times [d]$ with
  $A_{\mathsf{blue}} = \{ (1,1),(2,2),\dotsc,(d,d)\}$ and $A_c = \emptyset$
  for all other colours $c\in \Colours$.
  Let $u = \bone$ be the all-ones $d\times d$ matrix.
  Then $\Board_{d\times d}(\cA,\bone,R) = \Sl_d(R)$.
\end{ex}

Given a partial colouring $\cA$ of $[d]\times [e]$ and a matrix $u$ as above,
we may impose relations among the entries in the top right $d\times e$ block
of $\Alt_{d+e}(R)$ and $\Sym_{d+e}(R)$.
Formally, define
\begin{align*}
  \AltBoard_{d\times e}(\cA,u,R) &:=
  \left\{
    \begin{bmatrix}
      a & x \\ -x^\top & b
    \end{bmatrix}
    : a\in \Alt_d(R),  b\in \Alt_e(R),
    x\in \Board_{d\times e}(\cA,u,R)
  \right\} \text{ and} \\
  \SymBoard_{d\times e}(\cA,u,R) & :=
  \left\{
    \begin{bmatrix}
      a & x \\ x^\top & b
    \end{bmatrix}
    : a\in \Sym_d(R),  b\in \Sym_e(R),
    x\in \Board_{d\times e}(\cA,u,R)
  \right\}.
\end{align*}

We refer to $\Board_{d\times e}(\cA,u,R)$, $\AltBoard_{d \times e}(\cA,u,R)$, and
$\SymBoard_{d\times e}(\cA,u,R)$ as (rectangular, alternating, or symmetric)
\emph{relation modules} associated with the partial colouring~$\cA$.
If $u$ is the all-ones matrix $\bone$, then we often simply write $\Board_{d \times
  e}(\cA,R)$ instead of $\Board_{d\times e}(\cA,\bone,R)$ and analogously for
$\AltBoard_{d \times e}$ and $\SymBoard_{d\times e}$.

In general, neither
$\AltBoard_{d\times e}(\cA,u,R)$ nor
$\SymBoard_{d\times e}(\cA,u,R)$ is an instance of a module
$\Board_{(d+e)\times (d+e)}(\cA',u',R)$ for a partial colouring $\cA'$ of the
grid $[d+e]\times [d+e]$.
Note that in the definitions of alternating and symmetric relation modules,
we only impose relations among entries in the off-diagonal blocks within the
ambient modules.
Later on in this paper, we will also consider more general relations among entries
within $\Alt_d(R)$ or $\Sym_d(R)$; cf.\ Remark~\ref{rem:middle}.

\paragraph{Relation modules and ask zeta functions.}
Let $\fO$ be a compact \DVR{}.
In general, passing from an ambient module of matrices to a relation module
changes the associated ask zeta functions. 
For example, suppose that $\cA=(A_c)_{c\in\Colours}$ is
a partial colouring of $[d]\times [e]$ such
that there is at most one cell of any given colour.
(That is, $\card{A_c} \le 1$ for all $c\in \Colours$.)
Then $\Board_{d\times e}(\cA,u,\fO)$ consists of those matrices $[x_{ij}] \in
\Mat_{d\times e}(\fO)$ such that $x_{ij} = 0$ whenever~$(i,j)$ is a coloured
cell.
Ask zeta functions associated with modules of general rectangular, symmetric, or
alternating matrices satisfying such support constraints are precisely the
subject of \cite{cico}.
In particular, the results there show that these ask zeta functions are
in general vastly more complicated than the tame formulae for
the ask zeta functions of the ambient modules in Proposition~\ref{prop:classical_ask}.
The present article is devoted to a question which is orthogonal to
the setting of~\cite{cico}:

\begin{qu}
  Which simple linear relations among matrix entries have no effects on
  associated ask zeta functions?
\end{qu}

The complicated formulae in \cite{cico} reflect an intricate combinatorial
structure found in the rank loci of certain types of matrices of linear
forms.
Our approach in this article instead seeks to identify operations which have
no effect on these rank loci, or at least none that would be visible on the
level of ask zeta functions.\\

\noindent
As illustrated by $\Sl_d(\fO)$ in Proposition~\ref{prop:classical_ask}\ref{prop:classical_sl} and 
Example~\ref{ex:sl}, there are examples of partial colourings $\cA$
such that $\Mat_{d\times e}(\fO)$ and its proper submodule
$\Board_{d\times e}(\cA,u,\fO)$ have the same ask zeta function.
As we will now explain, there are many more such examples.

\vspace*{-0.7em}

\paragraph{Admissible partial colourings.}
Let $\cA = (A_c)_{c\in \Colours}$ be a fixed partial colouring of
$[d]\times [e]$.
By a \emph{subgrid} $G$ of $[d] \times [e]$, we mean a set of the form
$G = I\times J$ for $I\subset [d]$ and $J\subset [e]$.
We say that a subgrid $G$ of $[d]\times [e]$ is \emph{colour-closed} (w.r.t.\
$\cA$) if $A_c\subset G$ for each colour $c\in \Colours$ that appears within
$G$ (i.e.\ for which $A_c\cap G\not= \emptyset$).

\begin{defn}
  A partial colouring $\cA$ of $[d]\times [e]$ is \emph{admissible} if
  every non-empty colour-closed subgrid of $[d]\times [e]$ contains a blank cell. 
\end{defn}

\begin{ex}
  \label{ex:guide1}
  For $\ell = \text a,\text b,\text c, \text d$, let
  $\cA(\ell)$ be the partial colouring of $[3]\times [3]$ in
  Figure~\ref{fig:guide1}.
  Here and throughout, white cells indicate blanks and cells are indexed in
  the same way as matrix entries. (For example, the top left cell is $(1,1)$.)

  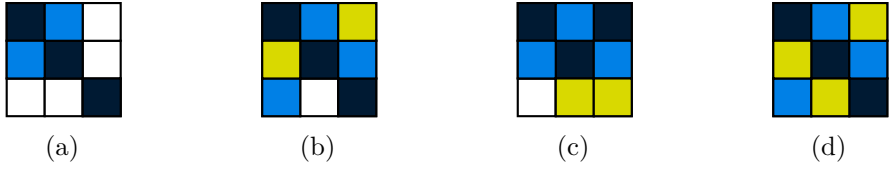
\begin{figure}[h]
    \centering
    \begin{subfigure}[t]{0.22\textwidth}
      \centering
      \begin{tikzpicture}
        [box/.style={rectangle,draw=black,thick, minimum size=1cm},
        scale=0.5, every node/.style={scale=0.5}
        ]
        \foreach \x in {0,1,2}{
          \foreach \y in {0,1,2}
          \node[box,fill=white] at (\x,\y){};
        }
        \node[box,fill=ourdark] at (0,2){};  
        \node[box,fill=ourdark] at (1,1){};  
        \node[box,fill=ourdark] at (2,0){};

        \node[box,fill=ourgreen] at (0,1){};
        \node[box,fill=ourgreen] at (1,2){};
      \end{tikzpicture}
      \caption{}
    \end{subfigure}
    \begin{subfigure}[t]{0.22\textwidth}
      \centering
      \begin{tikzpicture}
        [box/.style={rectangle,draw=black,thick, minimum size=1cm},
        scale=0.5, every node/.style={scale=0.5}
        ]
        \foreach \x in {0,1,2}{
          \foreach \y in {0,1,2}
          \node[box,fill=white] at (\x,\y){};
        }
        \node[box,fill=ourdark] at (0,2){};  
        \node[box,fill=ourdark] at (1,1){};  
        \node[box,fill=ourdark] at (2,0){};
        \node[box,fill=ourgreen] at (1,2){};  
        \node[box,fill=ourgreen] at (2,1){};  
        \node[box,fill=ourgreen] at (0,0){};  
        \node[box,fill=ouryellow] at (2,2){};  
        \node[box,fill=ouryellow] at (0,1){};  
      \end{tikzpicture}
      \caption{}
      \label{fig:PW3}
    \end{subfigure}
    \begin{subfigure}[t]{0.22\textwidth}
      \centering
      \begin{tikzpicture}
        [box/.style={rectangle,draw=black,thick, minimum size=1cm},
        scale=0.5, every node/.style={scale=0.5}
        ]
        \foreach \x in {0,1,2}{
          \foreach \y in {0,1,2}
          \node[box,fill=white] at (\x,\y){};
        }
        \node[box,fill=ourdark] at (0,2){};  
        \node[box,fill=ourdark] at (1,1){};  
        \node[box,fill=ourdark] at (2,2){};  
        \node[box,fill=ourgreen] at (0,1){};  
        \node[box,fill=ourgreen] at (1,2){};  
        \node[box,fill=ourgreen] at (2,1){};  
        \node[box,fill=ouryellow] at (1,0){};  
        \node[box,fill=ouryellow] at (2,0){};  
      \end{tikzpicture}
      \caption{}
    \end{subfigure}
    \begin{subfigure}[t]{0.22\textwidth}
      \centering
      \begin{tikzpicture}
        [box/.style={rectangle,draw=black,thick, minimum size=1cm},
        scale=0.5, every node/.style={scale=0.5}
        ]
        \foreach \x in {0,1,2}{
          \foreach \y in {0,1,2}
          \node[box,fill=white] at (\x,\y){};
        }
        \node[box,fill=ourdark] at (0,2){};  
        \node[box,fill=ourdark] at (1,1){};  
        \node[box,fill=ourdark] at (2,0){};
        \node[box,fill=ourgreen] at (1,2){};  
        \node[box,fill=ourgreen] at (2,1){};  
        \node[box,fill=ourgreen] at (0,0){};  
        \node[box,fill=ouryellow] at (2,2){};  
        \node[box,fill=ouryellow] at (0,1){};  
        \node[box,fill=ouryellow] at (1,0){};  
      \end{tikzpicture}
      \caption{}
    \end{subfigure}
    \caption{Four partial colourings of $[3]\times [3]$}
    \label{fig:guide1}
  \end{figure}

  For an alternative description of $\cA(\ell)$, let $c_1$, $c_2$, and $c_3$
  be distinct colours.
  By using matrix entries to specify colours (or blanks),
  each of the following matrices $C(\ell)$ encodes the partial
  colouring $\cA(\ell)$:

  {\small
  \[
    C(\mathrm a) =
    \begin{bmatrix}
      c_1 & c_2 & \blank \\
      c_2 & c_1 & \blank \\
      \blank & \blank & c_1
    \end{bmatrix},
    C(\mathrm b) =
    \begin{bmatrix}
      c_1 & c_2 & c_3 \\
      c_3 & c_1 & c_2 \\
      c_2 & \blank & c_1
    \end{bmatrix},
    C(\mathrm c) =
    \begin{bmatrix}
      c_1 & c_2 & c_1 \\
      c_2 & c_1 & c_2 \\
      \blank & c_3 & c_3
    \end{bmatrix},
    C(\mathrm c) =
    \begin{bmatrix}
      c_1 & c_2 & c_3 \\
      c_3 & c_1 & c_2 \\
      c_2 & c_3 & c_1
    \end{bmatrix}.
  \]}

  Note that the partial colourings $\cA(\mathrm a)$ and $\cA(\mathrm b)$ are
  admissible, while $\mathcal A(\mathrm c)$ and $\mathcal A(\mathrm d)$ are not.
\end{ex}

\begin{ex}
  \label{ex:sl_reprise}
  The partial colouring in Example~\ref{ex:sl} is admissible
  if and only if $d > 1$.
\end{ex}

\subsection{Results I: preservation of ask zeta functions}
\label{ss:preservation}

The following theorem is the main result of the present paper.
It states that relation modules arising from admissible partial colourings
have the same ask zeta functions as their ambient modules.
In fact, this remains true if the ambient module in question is suitably
embedded into a larger module of matrices.
Let $\std_{ij}$ be the elementary matrix with an entry $1$ in position $(i,j)$.
Let $\fO$ be a compact \DVR{} with residue cardinality $q$.

\begin{thmabc}
  \label{thm:embedded}
  Let $\cA$ be an admissible partial colouring of $[d]\times [e]$.
  Let $u \in \Mat_{d\times e}(\fO)$ have unit entries.
  Let $m$, $n$, $M$, and $M'\subset M$
  be given by one of the rows of the following table:
  \begin{center}
    \begin{tabular}{cc|cc}
      $m$ & $n$ & $M$ & $M'$ \\
      \hline
      $d$ & $e$& $\Mat_{d\times e}(\fO)$ & $\Board_{d\times e}(\cA,u,\fO)$ \\
      $d+e$ & $d+e$& $\Alt_{d+e}(\fO)$ & $\AltBoard_{d\times e}(\cA,u,\fO)$ \\
      $d+e$ & $d+e$& $\Sym_{d+e}(\fO)$ & $\SymBoard_{d\times e}(\cA,u,\fO)$. \\
    \end{tabular}
  \end{center}
  Let $1 \le r_1 < \dotsb < r_m\le \tilde m$ and
  $1 \le c_1< \dotsb  < c_n\le \tilde n$.
  Let $\widetilde\dtimes\colon \Mat_{m\times n}(\fO) \into \Mat_{\tilde m\times \tilde n}(\fO)$
  be the embedding with $\tilde\std_{ij} = \std_{r_ic_j}$.
  Let $N\subset \Mat_{\tilde m\times \tilde n}(\fO)$ be an arbitrary submodule.
  Then
  \[
    \Zeta^{\ak}_{\widetilde{M\phantom'}+ N}(T) = \Zeta^{\ak}_{\widetilde{M'}+ N}(T).
  \]
\end{thmabc}

\begin{ex}
  \label{ex:trd}
  Let $\tr_d(\fO)$ be the module of upper triangular $d\times d$ matrices over
  $\fO$.
  Let $d > 1$.
  Then Theorem~\ref{thm:embedded},
  Example~\ref{ex:sl}, and Example~\ref{ex:sl_reprise}
  show that $\tr_{2d}(\fO)$
  and
  \[
    L_d(\fO) := \begin{bmatrix}
      \tr_d(\fO) & \Sl_d(\fO) \\
      0 & \tr_d(\fO) 
    \end{bmatrix}
    \subset \tr_{2d}(\fO)
  \]
  have the same ask zeta function.
  In detail, we may take $M = \Mat_d(\fO)$, $M' = \Sl_d(\fO)$, 
  $N = \left[\begin{smallmatrix} \tr_d(\fO) & 0 \\ 0 &
      \tr_d(\fO)\end{smallmatrix}\right]$,
  and let $\widetilde\dtimes\colon \Mat_d(\fO) \into \Mat_{2d}(\fO)$
  be the embedding $a \mapsto \left[\begin{smallmatrix} 0 & a \\ 0 &
      0 \end{smallmatrix}\right]$
  in Theorem~\ref{thm:embedded}.
  Using \cite[Prop.\ 5.15(ii)]{ask}, we then conclude that
  $\Zeta^{\ak}_{L_d(\fO)}(T) = \Zeta^{\ak}_{\tr_{2d}(\fO)}(T) = \frac{(1-q^{-1}T)^{2d}}{(1-T)^{2d+1}}$.
\end{ex}

Theorem~\ref{thm:embedded} asserts that imposing suitable linear relations
within modules of matrices preserves associated ask zeta functions.
Example~\ref{ex:trd} illustrates that the former theorem
sometimes allows us to reduce computations of ask zeta functions to
previous results in the literature.
In the same spirit, by taking $\tilde m = m$, $\tilde n = n$, and $N = 0$
in Theorem~\ref{thm:embedded} and using Proposition~\ref{prop:classical_ask},
we obtain the following.

\begin{corabc}
  \label{cor:rec}
  If $\cA$ is admissible, then
  $\Zeta^{\ak}_{\Board_{d \times e}(\cA,u,\fO)}(T) =
  \frac{1-q^{-e}T}{(1-T)(1-q^{d-e}T)}$.
\end{corabc}

Using Example~\ref{ex:sl} and Example~\ref{ex:sl_reprise},
we thus recognise Corollary~\ref{cor:rec} as a vast generalisation of
Proposition~\ref{prop:classical_ask}\ref{prop:classical_sl}.

\begin{corabc}
  \label{cor:asym}
  If $\cA$ is admissible, then $\Zeta^{\ak}_{\AltBoard_{d\times e}(\cA,u,\fO)}(T) =
  \frac{1 - q^{1-d-e}T}{(1-T)(1-qT)}$.
\end{corabc}

Corollary~\ref{cor:asym} has group-theoretic consequences; see \S\ref{ss:adjoint}.

\begin{corabc}
  \label{cor:sym}
  If $\cA$ is admissible, then $\Zeta^{\ak}_{\SymBoard_{d\times e}(\cA,u,\fO)}(T) =
  \frac{1 - q^{-d-e}T}{(1-T)^2}$.
\end{corabc}

While Corollaries~\ref{cor:rec}--\ref{cor:sym} follow from
Theorem~\ref{thm:embedded}, we will reverse the order and first establish the
former corollaries and only then prove Theorem~\ref{thm:embedded}.

\begin{rem}
  \label{rem:middle}
  Theorem~\ref{thm:template}, our most general version of
  Theorem~\ref{thm:embedded}, will establish the conclusions of the latter
  for more general admissible partial colourings (suitably defined) of ``alternating''
  or ``symmetric'' grids.
  The latter include partial colourings that do not arise from a partial
  colouring of a rectangular grid (i.e.\ by imposing relations in a top right block)
  as in Theorem~\ref{thm:embedded}.
  This will, for instance, explain why the ask zeta function of
  $\{ x\in \Alt_4(\fO) : x_{12} + x_{23} + x_{34} = 0\}$
  coincides with that of $\Alt_4(\fO)$;
  see Example~\ref{ex:snakes}.
\end{rem}

A natural follow-up problem that we shall not pursue here is to find
analogues of Theorem~\ref{thm:embedded} and
Corollaries~\ref{cor:rec}--\ref{cor:sym} for submodules of more general
classes of ambient modules. 
Inspired by \cite{cico}, submodules of $\Mat_{d\times e}(\fO)$, $\Alt_d(\fO)$, or
$\Sym_d(\fO)$ defined via support constraints would provide an interesting class
of such ambient modules.
We will see that Theorem~\ref{thm:jacobi} below is a first step in this direction.

It is also natural to ask to what extent our admissibility assumptions are
necessary for the validity of our results above.
It is possible to construct examples of non-admissible partial
colourings $\cA$ and matrices $u$ with unit entries such that the
conclusion of Corollary~\ref{cor:rec} holds.
However, the authors do not know of an example of a non-admissible partial
colouring $\cA$ such that the conclusions of Corollary~\ref{cor:rec} hold for
\itemph{all} $u$.

\subsection{Examples, non-examples, and rank distributions}
\label{ss:intro_exs}

We now describe several examples and non-examples of Corollary~\ref{cor:rec}
that illustrate a number of features, in particular pertaining to
rank distributions within spaces of matrices over finite fields as in
\ref{motivation_rank} from \S\ref{ss:motivation}.
(For group-theoretic interpretations, see \S\ref{ss:intro_reprise}.)

\begin{ex}
  \label{ex:guide2}
  Let $M(\ell) = \Board_{3\times 3}(\cA(\ell),\fO)$
  for $\cA(\ell)$ as in Example~\ref{ex:guide1}.
  By Corollary~\ref{cor:rec},
  $\Zeta^\ak_{M(\text a)}(T) = \Zeta^\ak_{M(\text b)}(T) =
  \frac{1-q^{-3}T}{(1-T)^2} = \Zeta^{\ak}_{\Mat_3(\fO)}(T)$.
  On the other hand,
  the conclusion of Corollary~\ref{cor:rec} does not hold for either $M(\text 
  c)$ or $M(\text d)$.
  Indeed, using the package \textsf{Zeta}~\cite{Zeta,1489} for
  SageMath~\cite{SageMath}, we find that if $\fO$ is a compact \DVR{} with sufficiently large 
  residue characteristic and residue cardinality $q$,
  then $\Zeta^{\ak}_{M(\text c)}(T)$ and $\Zeta^{\ak}_{M(\text d)}(T)$ are
  both of the form

  {\footnotesize
  \begin{align*}
     \frac{
      1
      + N q^{-1} T
      - 2 (N + 1) q^{-2}  T
      + N q^{-3}  T
      + q^{-4}T^2}
    {(1- q^{-1}T)(1 - T)^2}
    & =
    1 + \left(2  + \frac{N+1}q - 2\frac{N+1}{q^{2}} + \frac{N}{q^{3}}\right)T + \mathcal O(T^2),
  \end{align*}}
  
  \noindent
  where $N = N(\ell,q)$.
  In detail, $N(\text c,q) = 2$ and $N(\text d,q)$ is the number of
  roots of $X^2 + X + 1$ in $\FF_q$.
  Thus, $N(\text d,q) = 2$ if $q\equiv 1 \pmod 3$ and $N(\text d,q) = 0$ otherwise.
  Hence, the average size of a kernel within the image, $\bar M(\text d)$ say, of $M(\text d)$
  over the residue field of $\fO$ is not rational in~$q$
  and the number of matrices of rank $1$ in $\bar M(\text d)$ is not polynomial in $q$.
\end{ex}

\begin{ex}
  \label{ex:adm_porc}
  Consider the admissible partial colouring $\cA$ of $[3]\times [3]$ given in
  Figure~\ref{fig:adm3x3}.
  Hence, for each commutative ring $R$,
  \[
    \begin{aligned}
      \Board_{3\times 3}(\cA,R) =\bigl\{
      [x_{ij}]\in \Mat_3(R) & : x_{11} + x_{22}\!\!\!\! &= x_{12} + x_{23} &=  x_{13}+x_{31} \\& & = x_{21} + x_{32} &= 0 
      \bigr\}.
    \end{aligned}
  \]
  Using the methods for symbolically counting rational points
  on varieties from \cite[\S 5]{padzeta}
  and implemented in \textsf{Zeta}, we find that if $q$ is a power of a sufficiently large prime,
  then the number, $r_1(q)$ say, of matrices of rank $1$ in $\Board_{3\times 3}(\cA,\FF_q)$ is
  $(N(q)+1)(q-1)$, where $N(q)$ is the number of roots of $X^4 + 1$ in $\FF_q$.
  In particular, $r_1(q)$ is not given by a polynomial in~$q$ (although it is
  a quasi-polynomial).
  This shows that while the modules $\Board_{d\times e}(\cA,u,R)$ associated
  with admissible partial colourings $\cA$ do exhibit all the features
  described in \ref{motivation_simple} and \ref{motivation_ask} from
  \S\ref{ss:motivation},
  we are forced to abandon \ref{motivation_rank}.
\end{ex}

\begin{figure}
  \centering
  \begin{subfigure}[t]{0.45\textwidth}
    \centering
    \begin{tikzpicture}
      [box/.style={rectangle,draw=black,thick, minimum size=1cm},
      scale=0.5, every node/.style={scale=0.5}
      ]
      \foreach \x in {0,1,2}{
        \foreach \y in {0,1,2}
        \node[box,fill=white] at (\x,\y){};
      }
      \node[box,fill=our0] at (0,2){};  
      \node[box,fill=our0] at (1,1){};

      \node[box,fill=our4] at (1,2){};  
      \node[box,fill=our4] at (2,1){};

      \node[box,fill=our1] at (2,2){};  
      \node[box,fill=our1] at (0,0){};

      \node[box,fill=our3] at (0,1){};  
      \node[box,fill=our3] at (1,0){};
    \end{tikzpicture}
    \caption{admissible}
    \label{fig:adm3x3}
  \end{subfigure}
  \begin{subfigure}[t]{0.45\textwidth}
    \centering
    \begin{tikzpicture}
      [box/.style={rectangle,draw=black,thick, minimum size=1cm},
      scale=0.5, every node/.style={scale=0.5}
      ]
      \foreach \x in {0,1,2}{
        \foreach \y in {0,1,2}
        \node[box,fill=white] at (\x,\y){};
      }
      \node[box,fill=our0] at (0,2){};  
      \node[box,fill=our0] at (2,0){};

      \node[box,fill=our4] at (1,2){};  
      \node[box,fill=our4] at (0,1){};

      \node[box,fill=our1] at (2,2){};  
      \node[box,fill=our1] at (1,1){};
      \node[box,fill=our1] at (1,0){};

      \node[box,fill=our3] at (0,0){};  
      \node[box,fill=our3] at (2,1){};
    \end{tikzpicture}
    \caption{non-admissible}
    \label{fig:nonporc}
  \end{subfigure}
  \caption{Two partial colourings of $[3]\times [3]$}
\end{figure}
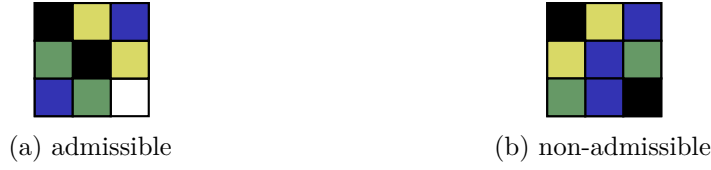

One of the key themes of \cite{cico} is the cancellation of arithmetically
wild behaviour of certain counting problems upon taking an average.
In this spirit, relation modules can be used to produce explicit examples of
mildly wild instances of such cancellations.

\begin{ex}
  \label{ex:nonporc}
  Let $\cN$ be the non-admissible partial colouring
  given in
  Figure~\ref{fig:nonporc}.
  Hence,
  \[
    \begin{aligned}
      \Board_{3\times 3}(\cN,R) =\bigl\{
      [x_{ij}]\in \Mat_3(R) & : x_{11} + x_{33}\!\!\!\! &= x_{12} + x_{21} &=  x_{13}+x_{22}+x_{32} \\& & = x_{23} + x_{31} &= 0 
      \bigr\}.
    \end{aligned}
  \]
  Using \textsf{Zeta}, we find that if $\fO$ is a compact \DVR{} with sufficiently large residue
  characteristic, then
  $\Zeta^{\ak}_{\Board_{3\times 3}(\cN,\fO)}(T) = \frac{(1-q^{-2} T)^2}{(1-q^{-1}T)(1-T)^2}$.
  Curiously, this formula coincides with the ask zeta function of the
  ``staircase module'' 
  $\left[\begin{smallmatrix}
      0 & * & * \\
      * & * & * \\
      * & * & *
    \end{smallmatrix}\right]
    \subset \Mat_{3}(\fO)
    $;
  see \cite[Prop.\ 5.10]{cico}.
  We cannot at present explain this coincidence.

  The rank loci of $\Board_{3\times 3}(\cN,\FF_q)$ are arithmetically richer
  than the tame formula for $\Zeta^{\ak}_{\Board_{3\times
      3}(\cN,\fO)}(T)$ might indicate.
  Indeed, using \textsf{Zeta} again (similarly to Example~\ref{ex:adm_porc}), for
  almost all primes $p$ and all powers $q$ of $p$, we find that the number of matrices of
  rank $1$ in $\Board_{3\times 3}(\cN,\FF_q)$ is $(N(q)+1)(q-1)$, where $N(q)$ is the number of roots of $f(X) := X^5 + X - 1 =
  (X^{2} - X + 1) (X^{3} + X^{2} - 1)$ in $\FF_q$.
  (The method for enumerating rational points implemented in \textsf{Zeta}
  does not keep track of the finitely many primes that need to be excluded.
  Although we will not need it, we note that experimental evidence suggests
  that the above formulae for the numbers of matrices of rank $1$ in
  $\Board_{3\times 3}(\cN,\FF_q)$ here and in
  $\Board_{3\times 3}(\cA,\FF_q)$ from Example~\ref{ex:adm_porc} might in fact
  both be correct without any restrictions on $q$.)

  As $p$ ranges over rational primes, $N(p)$ is \textit{not} constant on
  residue classes modulo any number $m\ge 1$.
  This follows from class field theory.
  Namely, a number field $K$ is abelian if and only if the following
  condition is satisfied:
  there exist  $m\ge 1$ and $H\subset [m]$ such that each sufficiently large rational prime
  $p$ splits completely in $K$ if and only if $p$ is congruent to an element
  of $H$ mod $m$.
  (See \cite[\S 5]{ConHCFT} or \cite[\S 7]{Gar81}.)
  To apply this here, let $a,b\in \CC$ with $a^2 - a + 1 = 0 = b^3 + b^2 - 1$.
  Using e.g.~SageMath, we find that the Galois group of the normal closure $E$
  of $\QQ(a,b)/\QQ$ is non-abelian.
  By basic facts on factorisation in number fields, a rational prime
  splits completely in $E$ if and only if it splits completely
  in $\QQ(a)$ and in $\QQ(b)$.
  It then follows that a rational prime $p\gg 0$ satisfies $N(p) = 5$ if and only
  if $p$ splits completely in $E$, a property that is not
  characterised by congruence conditions by the aforementioned result.
\end{ex}

\subsection{Results II: class counting zeta functions of free class-$3$-nilpotent groups}
\label{ss:adjoint}

Our final main result is a group-theoretic application of ideas
underpinning Theorem~\ref{thm:embedded}.

\paragraph{Class counting zeta functions.}
Grunewald, Segal, and Smith~\cite{GSS88} pioneered the study of zeta functions
in group theory.
Over the following decades, a rich theory encompassing numerous types of
algebraically motivated zeta functions has emerged;
see \cite{Vol11} for a recent survey.
The study of the following class of group-theoretic zeta
functions goes back to du~Sautoy~\cite{dS05}.
Let $\concnt(H)$ denote the number of conjugacy classes of a group $H$.
Let $\fO$ be a compact \DVR{} with maximal ideal $\fP$.
Let $\sG$ be a group scheme of finite type over $\fO$.
The \emph{class counting zeta function} of $\sG$ is the generating
function
\[
  \Zeta^{\cc}_{\sG}(T) := \sum_{n=0}^\infty \concnt(\sG(\fO/\fP^n)) T^n.
\]
In the literature,
these and related functions are also called
``conjugacy class'' and ``class number''  zeta functions.
For recent work in the area, see \cite{BDOP13,Lins1/19,Lins2/20,ask,ask2,cico}.

\paragraph{Ask zeta functions as class counting zeta functions.}
Let $M\subset \Alt_d(\ZZ)$ be a submodule.
As explained in \cite[\S\S 1.2--1.3, 2.4]{cico}, there exists
a unipotent group scheme $\sG_M$
such that $\Zeta^{\cc}_{\sG\otimes \fO}(T) = \Zeta^{\ak}_{M_\fO}(q^\ell T)$
for each compact \DVR{} $\fO$,
where $\ell$ is the rank of $M$ as a $\ZZ$-module, $q$ is the residue
cardinality of $\fO$, and $M_\fO$ is the $\fO$-submodule of $\Alt_d(\fO)$
generated by~$M$.
The group scheme $\sG_M$ is unipotent of class at most $2$ with underlying
scheme $\AA^{d+\ell}_{\ZZ}$.
For odd prime powers $q$, the finite group $\sG_M(\FF_q)$ can be easily
described in terms of the Baer correspondence~\cite{Bae38}; see \cite[\S 2.4]{cico}.

Let $\cA = (A_c)_{c\in\Colours}$
be an admissible partial colouring of of $[d]\times [e]$.
Let $M(\cA) := \AltBoard_{d\times e}(\cA,\ZZ) \subset \Alt_{d+e}(\ZZ)$.
If $b$ denotes the number of colours $c\in \Colours$ with $A_c\not=
\emptyset$, then $M(\cA)$ has rank $\binom{d+e}2 - b$.
Using Corollary~\ref{cor:asym}, we thus conclude that for each compact
\DVR{} $\fO$ as above, $\Zeta^\cc_{\sG_{M(\cA)}\otimes \fO}(T) = (1-q^{\binom{d+e}2 - d - e - b + 1}T)/((1-T)(1-qT))$.
This can e.g.~be used to construct examples of non-isomorphic group schemes
with identical associated class counting zeta functions.
While this constitutes an immediate group-theoretic application of our
results, our final main result (Theorem~\ref{thm:jacobi})
follows a different path.

\paragraph{Unipotent group schemes from Lie algebras.}
Let $R$ be a (commutative) ring.
For further details on the following, see \S\ref{ss:cc_ask} below.
Let $\fg$ be a nilpotent Lie $R$-algebra of class at most $c$.
Suppose that the underlying $R$-module of $\fg$ is free of finite rank.
Further suppose that $c! \in R^\times$.
Then $\fg$ naturally gives rise to a unipotent group scheme $\sG$ over~$R$
via the Baker-Campbell-Hausdorff series.
For each $R$-algebra $\fO$ which is a compact \DVR{}, we may express the class
counting zeta function of $\Zeta^{\cc}_{\sG\otimes \fO}(T)$ in terms of the
ask zeta function associated with the (image of the) adjoint representation of $\fg\otimes
\fO$. 

\paragraph{Free nilpotent Lie algebras and associated group schemes.}
Let $\mathfrak f_{c,d}$ be the free nilpotent Lie $\ZZ[1/c!]$-algebra of class
at most $c$ on $d$ generators.
(This algebra can be described explicitly in terms of Hall bases;
cf.\ \cite[Ch.\ 4]{Reu93}.)
Let $\sF_{c,d}$ be the associated unipotent group scheme over $\ZZ[1/c!]$.
For each prime $p > c$, the group $\sF_{c,d}(\ZZ_p)$ is the free
nilpotent pro-$p$ group of class at most $c$ on $d$ generators.
The class counting zeta functions associated with the group schemes
$\sF_{c,d}$ are of natural interest, in particular due to recent work of O'Brien and
Voll~\cite[\S\S 2,5]{O'BV15} on ``class vectors'' and ``character vectors''
of~$\sF_{c,d}(\FF_q)$.

\paragraph{Class counting zeta functions of $\sF_{c,d}$.}
Apart from trivial cases ($c\le 1$ or $d\le 1$) and the example
$\sF_{3,2}$ (see below), the class counting zeta functions
associated with $\sF_{c,d}$ have only been previously known for $c = 2$:

\begin{prop}[{\cite[Cor.\ 1.5]{Lins2/20}; \cite[Ex.\ 7.3]{ask2}}]
  \label{prop:F2d}
  Let $\fO$ be a compact \DVR{} with odd residue cardinality $q$.
    Then
    \[
      \Zeta^{\cc}_{\sF_{2,d}\otimes \fO}(T) =
      \frac{1 - q^{\binom{d-1} 2}T}{\bigl(1-q^{\binom d 2}T\bigr)\bigl(1-q^{\binom d 2 + 1}T\bigr)}.
    \]
\end{prop}

We note that one can construct a natural group scheme associated with any finitely
generated free class-$2$-nilpotent Lie algebra over $\ZZ$ (whose underlying $\ZZ$-module
is free of finite rank);
see \cite[\S 2.4]{SV14}. That is, in case of nilpotency class $2$, it is not
necessary to pass to the ring $\ZZ[1/2]$.
The preceding proposition extends to even residue characteristic.

\begin{thmabc}
  \label{thm:jacobi}
    Let $p\ge 5$ be a prime and let $q = p^f$.
    Let $\fO$ be a compact \DVR{} with residue cardinality $q$.
    Then:
  \begin{equation}
    \label{eq:F3d}
    \Zeta^{\cc}_{\sF_{3,d}\otimes \fO}(T) =
    \frac{
      \Bigl(1 - q^{\frac{(d-1)(d^2 + d-3)} 3}T\Bigr)
      \Bigl(1-q^{\frac{(d-2)d(d+2)} 3}T\Bigr)
    }
    {
      \Bigl(1 - q^{\frac{(d-1)d(d+1)} 3}T\Bigr)
      \Bigl(1 - q^{\frac{d^3-d+3} 3}T\Bigr)
      \Bigl(1 - q^{\frac{(2d^2+3d-11)d} 6}T\Bigr)
    }.
    \tag{\ensuremath{\ast}}
  \end{equation}
\end{thmabc}

\begin{ex}
    For $d = 2$, \eqref{eq:F3d} becomes
    \[
      \Zeta^{\cc}_{\sF_{3,2}\otimes \fO}(T) = \frac{1-T}{(1-q^2T)(1-q^3T)},
    \]
    in accordance with \cite[\S 9.3, Table~1]{ask}.
\end{ex}

In the setting of Theorem~\ref{thm:jacobi},
we may rewrite \eqref{eq:F3d} more conveniently as
$\Zeta^\cc_{\sF_{3,d}\otimes \fO}(T) = W_d\bigl(q,q^{\frac{(d-2)d(d+2)}{3}}T\bigr)$,
where $W_d(X,T) =\frac{(1-T)(1-XT)}{\bigl(1-X^dT\bigr)\bigl(1-X^{d+1}T\bigr)\bigl(1-X^{\binom d
    2}T\bigr)}$.
In particular, 
\begin{equation}
  \label{eq:F3d_expansion}
  \Zeta^{\cc}_{\sF_{3,d}\otimes \fO}(T) =
  1 + q^{\frac{(d-2)d(d+2)} 3}\!\left(q^{\binom d 2} + q^{d+1} + q^d - q - 1\right)T + \mathcal O(T^2)
  \tag{\ensuremath{\dagger}}
\end{equation}
and the class number $\concnt(\sF_{3,d}(\FF_q))$ is the coefficient of $T$ in
\eqref{eq:F3d_expansion}.
O'Brien and Voll~\cite[Thm~2.6]{O'BV15} showed that for all $c$ and $d$, there
exists an explicit $f_{c,d}(X) \in \ZZ[X]$ such that
$\concnt(\sF_{c,d}(\FF_q)) = f_{c,d}(q)$
whenever $\gcd(q,c!) = 1$.
Our formula for~$\concnt(\sF_{3,d}(\FF_q))$
in \eqref{eq:F3d_expansion} agrees with theirs for $c = 3$.

\paragraph{Linear relations with disjoint supports and Theorem~\ref{thm:jacobi}.}
We will now sketch how Theorem~\ref{thm:jacobi} fits
into our study of linear relations with disjoint supports.
First, it turns out to be advantageous to attach ask zeta functions not merely
to modules of matrices but, more generally, to ``module representations'':
homomorphisms from abstract modules into Hom-spaces.
One major advantage of this point of view is that it provides a natural framework
for operations dubbed ``Knuth dualities'' in \cite{ask2}.
(We may think of these dualities as
analogues of the classical identity
$\concnt(H) = \#\operatorname{Irr}(H)$ for finite groups~$H$.)

Turning to Theorem~\ref{thm:jacobi},
let $\mathfrak a_d$ be the largest $\ZZ[1/6]$-algebra which is generated by
$d$ elements and which satisfies the identities $v^2 = 0$ and $v(w(xy)) = 0$
(``class-$3$-nilpotency'') for all $v,w,x,y$.
We may identify $\mathfrak f_{3,d}  = \mathfrak a_d / \mathfrak j_d$,
where $\mathfrak j_d$ is the ideal of $\mathfrak a_d$ corresponding
to the Jacobi identity.
Taking $\hat\alpha_d$ to be the ``$\bullet$-dual'' (see \S\ref{ss:mreps}) of
the adjoint representation of $\mathfrak f_{3,d}$, we may express the class
counting zeta function of $\sF_{3,d}\otimes \fO$ in Theorem~\ref{thm:jacobi}
in terms of the ask zeta function attached to $\hat\alpha_d$ over $\fO$.
We similarly define $\alpha_d$ as the $\bullet$-dual of the adjoint
representation of $\mathfrak a_d$.
Up to harmless transformations, we may think of
$\hat\alpha_d$ as the restriction of $\alpha_d$ to a submodule.
In terms of matrices, this submodule is obtained by imposing linear relations
with unit coefficients and disjoint supports, one for each unordered triple of
defining generators of $\mathfrak a_d$.
The machinery developed in this article then allows us to show that $\alpha_d$
and $\hat\alpha_d$ give rise to the same ask zeta functions.

It then only remains to determine the ask zeta functions associated
with $\alpha_d$.
We will see that this problem has already been solved in \cite{cico}.
Indeed, $\alpha_d$ turns out to (essentially) be an ``adjacency
representation'' of a threshold graph as in \cite[\S 8.4]{cico}, and this
observation will allow us to finish our proof of Theorem~\ref{thm:jacobi}.

\begin{rem}
  \label{rem:F42}
  We now briefly explain why in our work towards Theorem~\ref{thm:jacobi},
  we restricted attention to nilpotency class $3$.
  Our definition of $\mathfrak a_d$ naturally extends to higher nilpotency class
  $c$, giving rise to $\ZZ[1/c!]$-algebras $\mathfrak a_{c,d}$.
  We can then identify $\mathfrak f_{c,d} = \mathfrak a_{c,d}/\mathfrak
  j_{c,d}$, where $\mathfrak j_{c,d}$ encodes the Jacobi identity.
  Similarly, we can define $\alpha_{c,d}$ and $\hat\alpha_{c,d}$,
  extending the definitions of $\alpha_d$ and $\hat\alpha_d$ from above.

  As we will see in \S\ref{ss:jacobi_no_more}, what allows us to essentially
  view $\hat\alpha_{3,d}$ as a restriction of $\alpha_{3,d}$
  is the fact that $\mathfrak j_{3,d}$ is central in $\mathfrak a_{3,d}$.
  This condition is hardly ever satisfied in higher class.
  Indeed, we leave it to the reader to verify that for $c \ge 4$, the ideal $\mathfrak
  j_{c,d}$ is central in $\mathfrak a_{c,d}$ if and only if $d \le 2$.
  This leaves us with the case $(c,d) = (4,2)$ as the only interesting
  candidate for a direct extension of Theorem~\ref{thm:jacobi}.
  However, our successful strategy in class~$3$ fails here since
  $\alpha_{4,2}$ and $\hat\alpha_{4,2}$ turn out to give rise to
  different ask zeta functions.

  We note that using \textsf{Zeta}, we find that for almost all primes $p$ and
  all powers $q$ of $p$,
  if $\fO$ is a compact \DVR{} with residue cardinality $q$, then
  \[
    \Zeta^{\cc}_{\sF_{4,2}\otimes \fO}(T) =
    \frac{q^{7} T^{3} - q^{6} T^{2} - q^{5} T^{2} + q^{4} T^{2} + q^{3} T - q^{2} T - q T + 1}
    {(1 - q^7 T^2 )(1 - q^4T )^2}.
  \]
\end{rem}

\subsection{Examples, non-examples, and rank distributions---reprise}
\label{ss:intro_reprise}

As our final remark on rank distributions in the spirit of
\ref{motivation_rank} from \S\ref{ss:motivation},
we now briefly explain how the modules in
Examples~\ref{ex:adm_porc}--\ref{ex:nonporc} give rise to group actions with
simple orbit structures, but arithmetically interesting numbers of fixed
points.

\paragraph{Two perspectives on orbits.}
Let $G$ be a finite group acting on a finite set $X$.
By the Cauchy-Frobenius lemma,
we can express $\card{X/G}$ in terms of the numbers of elements
of $G$ with prescribed numbers of fixed points.
Alternatively, we may express $\card{X/G}$ in terms of the numbers
of elements of $X$ that belong to $G$-orbits of prescribed sizes.

\paragraph{Average sizes of kernels and linear orbits.}
Apart from conjugacy classes, average sizes of kernels
are also related to counting orbits of unipotent linear groups;
see \cite[\S 8]{ask}.
We briefly recall the elementary part of this connection.
For a module $N\subset \Mat_{d\times e}(R)$ over a ring $R$ and $x\in R^d$,
let $\cent_N(x) := \{ n\in N : xn = 0\}$.
Let $M\subset \Mat_{d\times e}(\ZZ)$ be a submodule.
Let $\sL_M$ be the group scheme with $\sL_M(R) = M\otimes R$
(additive group) for each commutative ring $R$.
The action  $(x,y) m = (x,xm+y)$
of $M$ on $\ZZ^{d+e} = \ZZ^d\oplus \ZZ^e$
naturally extends to an action of $\sL_M$ on~$\AA^{d+e}$.
Let $\dtimes_R\colon M \otimes R \to \Mat_{d\times e}(R)$ be the natural map
and let $M_R$ denote its image.
The set of fixed points of $m\in \sL_M(R)$ on $R^{d+e}$
is $\Ker(m_R) \oplus R^e$.
Let $R$ be finite.
Then, by the Cauchy-Frobenius lemma, $\sL_M(R)$ has
$\card{R}^e \ask{M_R}$ orbits on $R^{d+e}$; cf.\ \cite[\S 2.2]{ask}.
The $\sL_M(R)$-orbit of $(x,y)\in R^{d+e}$  has size
$\card{M_R/\!\cent_{M_R}(x)}$.
We conclude that
the elements of $\sL_M(\FF_q)$ with precisely $q^{e+i}$ fixed points on
$\FF_q^{d+e}$ are precisely those whose images in $M_{\FF_q}$ have rank $d-i$.
Similarly, the $\sL_M(\FF_q)$-orbit of $(x,y)\in \FF_q^{d+e}$ consists
of precisely $q^i$ elements if and only if $\dim_{\FF_q}(M_{\FF_q}/\cent_{M_{\FF_q}}(x)) = i$.
The latter condition can also be expressed in terms of the rank loci of a
matrix of linear forms; cf.\ \cite[\S 4.3.5]{ask}.

\paragraph{Non-polynomiality, orbits, and fixed points.}
Let $\cA$ be a partial colouring of $[d]\times [e]$.
By minor abuse of notation, write $\sL_\cA := \sL_{\Board_{d\times e}(\cA,\ZZ)}$.
It is easy to see that we may identify
$\sL_{\cA}(R) = \Board_{d\times e}(\cA,R)$ for each commutative ring; cf.\
Lemma~\ref{lem:ranger}.
In particular, we may identify $\sL_{\cA}$ and the subgroup scheme
$R\mapsto \left[\begin{smallmatrix} 1 & \Board_{d\times e}(\cA,R) \\ 0 & 1 \end{smallmatrix}\right]$
of $\GL_{d+e}$.

Let $q$ be a prime power.
Clearly, $\sL_{\cA}(\FF_q)$ fixes each element of $\{0\} \times \FF_q^e$.
Moreover, if $\cA$ is admissible, then it will follow from our proof of
Corollary~\ref{cor:rec} that the orbit of each $(x,y)\in \FF_q^{d+e}$ with
$x\not = 0$ has size $q^e$. (See Corollary~\ref{cor:Rho_blueprint} and Theorem~\ref{thm:crk}.)

Let $\cA$ be as in Example~\ref{ex:adm_porc}.
Then $\sL_{\cA}(\FF_q)$ fixes precisely $q^3$ elements of $\FF_q^6$,
and the remaining $q^6-q^3$ elements all have $\sL_{\cA}(\FF_q)$-orbits of
size $q^3$.
(Hence, $\card{\FF_q^6/\sL_{\cA}(\FF_q)} = q^3 + (q^6-q^3)/q^3 = 2q^3-1$.)
On the other hand, Example~\ref{ex:adm_porc} shows that the number of
elements of $\sL_{\cA}(\FF_q)$ with precisely $q^5$ fixed points on $\FF_q^6$
depends on $q$ modulo $8$.

Next, let $\cN$ be as in Example~\ref{ex:nonporc}.
Let $(x,y)\in \FF_q^6$ with $x\not= 0$.
We leave it to the reader to verify that
the $\sL_{\cN}(\FF_q)$-orbit of $(x,y)$ consists of precisely $q^3$ points,
unless $x_1 = x_2 = x_3$, in which cases this orbit consists of $q^2$ points.
(Hence, $\card{\FF_q^6/\sL_{\cN}(\FF_q)} = 2q^3 + q^2 - 2q$.)
On the other hand, Example~\ref{ex:nonporc} shows that the number of
elements of $\sL_{\cA}(\FF_q)$ with precisely $q^5$ fixed points on $\FF_q^6$
does \textit{not} depend quasi-polynomially on $q$.

\subsection{Outline}

In \S\ref{s:ask}, we recall basic facts on module representations and ask zeta
functions.
Our approach revolves around what we call \itemph{orbit modules}, a concept
based on a ``cokernel formalism'' developed in \cite[\S 2.5]{cico}.
In \S\ref{s:orbital}, we introduce \itemph{orbital subrepresentations} of
module representation.
These provide a sufficient condition for proving equality between ask
zeta functions in terms of Fitting ideals of orbit modules.
Reversing the order of our exposition from above,
in \S\ref{s:jacobi}, we then prove Theorem~\ref{thm:jacobi} by implementing the
strategy outlined in \S\ref{ss:adjoint}.
Our proof will motivate several techniques developed in later sections.

Our proof of Theorem~\ref{thm:embedded} is based on a recursion
involving the deletion of rows and columns of matrices.
In \S\ref{s:coh}, we develop an abstract formalism for studying the effects of
these operations on orbit modules within \itemph{coherent families of
module representations}.
In~\S\ref{s:rel}, we then use partial colourings to impose linear relations with
disjoint supports and unit coefficients on modules occurring in suitable
coherent families of module representations.
This, in particular, yields a proof of Corollary~\ref{cor:rec}.
The final \S\ref{s:board} is devoted to linear relations among entries of
alternating and symmetric matrices. 
The key ingredient is a general notion of admissibility for partial
colourings.
This concept takes the form of a ``board game'' played on partially coloured
grids.
A recursion inspired by our proof of Corollary~\ref{cor:rec} then yields
proofs of Corollaries~\ref{cor:asym}--\ref{cor:sym}.
Finally, we combine several of our results and deduce Theorem~\ref{thm:embedded}.

\subsection*{Acknowledgements}

We began working on the research described here 
while we both visited the \href{https://www.him.uni-bonn.de}{Hausdorff Research
  Institute for Mathematics} as part of the
Trimester Program
\href{https://www.him.uni-bonn.de/programs/past-programs/past-trimester-programs/logic-algorithms-groups/description/}{``Logic
  and Algorithms in Group Theory''}.
We are grateful to the organisers and to the institute
for providing a stimulating research environment.
AC acknowledges support from the
{Irish Research Council}
through grant no.\ GOIPD/2018/319.
We would like to thank Christopher Voll for valuable comments and discussions
on the work described here.
Finally, we are grateful to the anonymous referee for a careful reading of the
manuscript and for providing many helpful suggestions.

\subsection*{\textit{Notation}}

\paragraph{Sets and maps.}
We write $[d] := \{1,2,\dotsc,d\}$.
The symbol ``$\subset$'' indicates not necessarily proper inclusion.
Maps usually act on the right and are composed from left to right.
If $\alpha\colon A\to B$ and $B' \subset B$, then we denote the preimage of
$B'$ under $\beta$ by $B'\beta^-$.
We denote the set of $k$-element subsets of $A$ by $\binom A k$
and the set of all finite subsets of $A$ by $\Powf(A)$.

\paragraph{Rings and modules.}
All rings are assumed to be associative, commutative, and unital.
Throughout, $R$ is a ring and $\fO$ is a discrete valuation ring (\DVR) with
maximal ideal $\fP$.
When $\fO$ is compact, we write $q$ for the residue cardinality of $\fO$.
We denote the free $R$-module on a set $A$ by $RA$.
To avoid ambiguities, we often write $\std_a$ for the 
element of $RA$ corresponding to $a\in A$.
We identify $R^d = R[d]$.
For $x\in RA$, we write $x = \sum_{a\in A} x_a \std_a$ ($x_a\in R$).
For $\fa \normal R$, we also write $\fa A = \sum_{a\in A} \fa \std_a \subset RA$.
Let $I$ and $J$ be finite sets.
We regard the elements of $I$ and $J$ as the row and column indices of
the elements of $\Hom(RI,RJ)$, regarded as $I\times J$ matrices.
Let $i\in I$, and $j\in J$.
When the reference to $I$ is clear, we let $\std_i^*\in (RI)^*$ be the
functional with $\std_j \std_i^* = \delta_{ij}$ (``Kronecker delta'').
We write $\std_{ij} := \std_i^* \std_j \in \Hom(RI,RJ)$ for an ``elementary
matrix''.
We let $X_I := (X_i)_{i\in I}$ consist of algebraically independent elements
over $R$.

\paragraph{Further notation}

\noindent
{\begin{longtable}{r|l|r}
  Notation\phantom{$1_1$} & comment & reference \\
  \hline
  $\inc$, $\ret$, $\infl$, $\resn$ & inclusion, retraction, inflation, restriction & \S\S\ref{ss:mreps}, \ref{ss:cohdef}\\
  $\theta^S$ & extension of scalars & \S\ref{ss:mreps} \\
  $\bullet$, $\circ$ & Knuth duals & \S\ref{ss:mreps} \\
  $\ask\theta$, $\Zeta^{\ak}_{\theta}(T)$ & average size of kernel, ask zeta function & \S\ref{ss:orbit_modules} \\
  $\XX(m,\theta)$, $\orbit(\theta)$, $\Orbit(\theta)$ & orbit modules & \S\ref{ss:orbit_modules} \\
  $S_x$, $M_x$ & specialisation of a ring or module & Remark~\ref{rem:orbit}\\
  $\Fit_i(M)$ & Fitting ideal & \S\ref{ss:Fit} \\
  $\Omega(I,J)$, $\Omega^\times(I,J)$ & $\Orbit(\theta(I,J))$, $\Omega(I,J)\otimes_{R[X_I]} R[X_I^{\pm 1}]$ & \S\ref{ss:crk} \\
  $\bRho$, $\rho(I,J)$ & generic matrices & Example~\ref{ex:rec} \\
  $\sE(I,J)$, $\bGamma$, $\gamma(I,J)$ & alternating matrices & Example~\ref{ex:graph} \\
  $\sS(I,J)$, $\bSigma$, $\sigma(I,J)$ & symmetric matrices &
  Example~\ref{ex:Sigma}\\
  $\beta[B]$ & colours contained entirely in $B$ & \S\ref{ss:relmod}\\
  $\Board(B \sslash\beta)$ & relation module & Definition~\ref{d:relmod} \\
  $\theta\sslash\beta$, $\bTheta\sslash\beta$
  & restrictions to relation modules & \S\ref{ss:relmod} \\
  $\beta(I,J)$&induced partial colouring &Definition~\ref{d:betaIJ}\\
  $\cG(I,J)$ & grid w.r.t\ $\bTheta$ & \S\ref{ss:combfam} \\
  $\xto[\bTheta,\beta]{}$ & move & Definition~\ref{d:reduce} \\
  $\beta^\top$ & transpose partial colouring & \S\ref{ss:rec_board} \\
  $\hat\beta$ & ``(anti)symmetrisation'' of $\beta$ & Definition~\ref{d:betahat}
\end{longtable}
}

\section{Average sizes of kernels and orbit modules}
\label{s:ask}

We summarise the basics of ``ask zeta functions'' using the formalism
from \cite{ask2} and \cite{cico}.

\subsection{Module representations}
\label{ss:mreps}

For more about most of the following, see \cite[\S\S 2, 4]{ask2}.

\paragraph{Basics.}
A \emph{module representation} over $R$ is a linear map $\theta\colon M\to
\Hom(N,O)$, where $M$, $N$, and $O$ are $R$-modules.
For an $R$-algebra $S$, we let $\theta^S\colon M\otimes S \to \Hom(N\otimes
S,O\otimes S)$ be the induced \emph{extension of scalars} of $\theta$ given by
$(m\otimes s) \theta^S = m\theta \otimes s$ ($m\in M$, $s\in S$).
Let $\theta'\colon M'\to\Hom(N',O')$ be another module representation over
$R$.
A \emph{homotopy} $\theta\to \theta'$ is a triple of linear maps
$(\mu\colon M\to M', \phi\colon N\to N',\psi\colon O \to O')$ such that that
for all $m\in M$, $(m\theta) \psi = \phi ((m\mu)\theta')$.
Homotopies can be composed as expected.
An \emph{isotopy} is an invertible homotopy.
We say that $\theta$ and $\theta'$ are \emph{isotopic} if an isotopy
$\theta\to\theta'$ exists.
We will often only be interested in module representations up to isotopy.

\paragraph{Matrices I.}
Each module of matrices $M\subset \Mat_{d\times e}(R)$ gives rise to a module
representation $M\incl \Mat_{d\times e}(R) = \Hom(R^d,R^e)$.
In this article, our main focus will be on module representations
$\theta\colon M\to
\Hom(RI,RJ)$,
where $I$ and $J$ are finite subsets of $\NN$ and $M$ is free of finite rank
(but perhaps without a canonical basis).
Under such a $\theta$, each element of $M$ is sent to a matrix with rows indexed
by $I$ and columns indexed by $J$.
This setup allows us to conveniently add or delete rows or columns without
relabelling indices.

Let $\theta\colon R B \to \Hom(R I, RJ)$ be a module representation involving
free modules with given bases.
Let $X_B = (X_b)_{b\in B}$ be algebraically independent over $R$.
Then $\theta$ gives rise to an $I\times J$ matrix $\sA(X_B) = [ \sum_{b\in B}
X_b a_{bij} ]_{i\in I, j\in J}$ such that for each $b\in B$, $[a_{bij}]_{i\in
  I,j\in J}$ is the matrix of $\std_b\theta$ with respect to the given bases
$I$ and $J$.
Conversely, any $I\times J$ matrix $\sA'(X_B)$ whose entries are homogeneous
linear forms from $R[X_B]$ gives rise to a module representation $R B\to
\Hom(RI,RJ)$ by specialisation $x\mapsto \sA'(x)$. Up to isotopy, these
constructions are mutually inverse.

\paragraph{Knuth duality.} (See \cite[\S 4]{ask2}.)
Let $(\dtimes)^* = \Hom(\dtimes, R)$ denote the dual of $R$-modules.
Each module representation $\theta\colon M \to \Hom(N,O)$ over $R$ gives rise
to its \emph{Knuth duals}
\[
  \begin{aligned}
    \theta^\circ&\colon  N  &\xto{\phantom{XYZ}}& \,\Hom(M,O), \quad &
    n \mapsto\,& (m\mapsto n(m\theta)) \quad\text{and} \\
    \theta^\bullet&\colon  O^* &\xto{\phantom{XYZ}}& \,\Hom(N,M^*), \quad & \omega\mapsto\,& (n\mapsto (m\mapsto (n(m\theta))\omega)).
  \end{aligned}
\]
Suppose that $M = RB$, $N = RI$, and $O = RJ$, where $B$, $I$, and $J$ are
finite.
Let $\sA(X_B) = [\sum_{b\in B} X_b a_{bij}]_{i\in I,j\in J}$ be the matrix of linear forms
associated with $\theta$ as above.
Then $\sC(X_I) := [\sum_{i \in I} X_i a_{bij}]_{b\in B, j\in J}$ is the matrix of
linear forms associated with $\theta^\circ$.
Using dual bases, we may regard $\theta^\bullet$ as a module representation
$RJ \to \Hom(RI,RB)$.
The matrix of linear forms associated with $\theta^\bullet$ is then
$\sB(X_J) := [\sum_{j\in J} a_{bij}]_{i\in I,b\in B}$.

\paragraph{Restriction and inflation.}
For sets $A\subset B$, let $\inc = \inc_{A,B}$ denote the inclusion $RA \incl
R B$.
The \emph{retraction} map $\ret = \ret_{B,A}\colon RB \onto R A$ fixes $A$
elementwise and vanishes on $B\setminus A$.
Let $I\subset \tilde I$ and $J\subset \tilde J$.
Let $\theta\colon M\to\Hom(RI,RJ)$ and $\tilde\theta\colon \tilde M\to\Hom(R\tilde
I,R\tilde J)$ be module representations.
The \emph{$(I,J)$-restriction} of $\tilde\theta$ is the composite
\[
  \xymatrix@C+2em{
  \resn^{\tilde I,\tilde J}_{I,J}(\tilde\theta)\colon \tilde M \ar[r]^{\tilde\theta} & \Hom(R\tilde I,R\tilde J) \ar[r]^{\Hom(\inc,\ret)}& \Hom(R I, RJ)
}
\]
In terms of matrices, 
for $\tilde m\in\tilde M$, the matrix of
$m \,\resn^{\tilde I,\tilde J}_{I,J}(\tilde\theta)$ is obtained from that
of $\tilde m\tilde \theta$ by deleting all rows except those in $I$ and all
columns except those in $J$.
Dually, the \emph{$(\tilde I,\tilde J)$-inflation} of $\theta$ is
the composite
\[
  \xymatrix@C+2em{
    \infl^{\tilde I,\tilde J}_{I,J}(\theta)\colon
    M\ar[r]^{\theta} & \Hom(RI,RJ) \ar[r]^{\Hom(\ret,\inc)} & \Hom(R \tilde I,R\tilde J)
  }
\]
If $\tilde I = \tilde J$ and $I = J$, we also simply write $\inf^{\tilde
  I}_I(\theta) = \inf^{\tilde I,\tilde J}_{I,J}(\theta)$.

\paragraph{Matrices II.}
Let $\theta$ and $\tilde \theta$ be as above with $M = RB$ and $\tilde M =
R\tilde B$.
Write $X = (X_b)_{b\in B}$ and $\tilde X = (X_{\tilde b})_{\tilde b \in \tilde B}$.
Let $A(X)$ and $\tilde A(\tilde X)$ be the matrices of linear forms associated
with $\theta$ and $\tilde\theta$. 
Then the matrix of linear forms associated with $\resn^{\tilde I,\tilde
  J}_{I,J}(\tilde\theta)$ is obtained from $\tilde A(\tilde X)$ by deleting
all rows indexed by $\tilde I\setminus I$ and all columns indexed by $\tilde
J\setminus J$.
Dually, the matrix of linear forms associated with $\infl^{\tilde I,\tilde
  J}_{I,J}(\theta)$ coincides with $A(X)$ in positions indexed by $I\times J$
and has zero entries elsewhere.

\subsection{Orbit modules and average sizes of kernels}
\label{ss:orbit_modules}

The following collects and combines material from
\cite[\S\S 3--4]{ask}, \cite[\S\S 3, 5]{ask2}, and \cite[\S 2]{cico}.

\paragraph{Average sizes of kernels.}
Let $\theta\colon M\to \Hom(N,O)$ be a module representation.
If $M$ and $N$ are both finite as sets, we define
\[
  \ask\theta := \frac 1 {\card M}\sum_{m\in M} \card{\Ker(m\theta)}
\]
to be the average size of the kernels of the elements of $M$ acting on $N$ via
$\theta$.
Note that $\ask\theta$ only depends on the image $M \theta\subset \Hom(N,O)$.

\paragraph{Orbit modules.}
Let $I$ and $J$ be finite sets.
Let $\theta\colon M \to \Hom(RI,RJ)$ be a module representation.
For $x\in RI$,
$x(M\theta) = \{ x(m\theta) : m \in M \} \subset RJ$
is the \emph{additive orbit} of $x$ under $M$ acting via $\theta$.
Let $X_I := (X_i)_{i\in I}$ consist of independent variables over $R$.
Write
\begin{align}
  \label{eq:Xatheta}
  \XX(m,\theta)
  & :=
    \sum_{i \in I} X_i \std_i \Bigl(m\theta^{R[X_I]}\Bigr)
    \in R[X_I] J
  & (m\in M),
\end{align}
where we identified $M\subset M\otimes R[X_I]$ via the natural embedding.
Let
$\orbit(\theta) :=
\left\langle
  \XX(m,\theta)
  : m \in M\right\rangle
\le R[X_I] J$.

\begin{defn}
  The \emph{orbit module} of $\theta$ is $\Orbit(\theta) := {R[X_I] J}/{\orbit(\theta)}$.
\end{defn}

Equivalently, $\Orbit(\theta)$ is the cokernel of the map $M\otimes R[X_I] \to
R[X_I] J$ induced by $M\to R[X_I]J, \,m\mapsto \XX(m,\theta)$.

\begin{rem}
  \label{rem:orbit}
  \quad
  \begin{enumerate}[label=(\roman{*})]
  \item
    \label{rem:orbit1}
    Strictly speaking, $\orbit(\theta)$ and $\Orbit(\theta)$ not only depend
    on $\theta$ but also on the basis~$I$.
    Moreover, $\Orbit(\theta)$ only depends on $M\theta$, not on $\theta$
    itself.
  \item
    \label{rem:orbit2}
    $\Orbit(\theta)$ specialises to quotients by additive orbits
    as follows.
    Let $S$ be an $R$-algebra and
    $x\in SI$.
    Let $S_x$ denote $S$ regarded as an $R[X_I]$-module via
    $X_i s = x_i s$ ($s\in S$).
    For an $R[X_I]$-module $M$, write $M_x = M\otimes_{R[X_I]} S_x$.
    Then $\Orbit(\theta)_x \approx \frac{S J}{x ((M\otimes S)
      \theta^S)}$.
    (Similarly to the second proof of \cite[Lemma 2.1]{ask}, one may then relate
    orbit modules to orbits of linear group actions as in \S\ref{ss:intro_reprise}.)
  \item
    \label{rem:orbit3}
    If $R\to S$ is a ring map, then we may identify
    $\Orbit(\theta^S) = \Orbit(\theta)\otimes_{R[X_I]} S[X_I]$.
  \end{enumerate}
\end{rem}

\begin{lemma}[{Cf.\ \cite[\S 2.2]{cico}}]
  \label{lem:Cmatrix}
  Let $\theta\colon RB \to \Hom(RI,RJ)$ be a module representation, where $B$,
  $I$, and $J$ are finite.
  Let $\sC(X_I)$ be the matrix of linear forms associated with $\theta^\circ$
  (w.r.t.\ the given bases) as in \S\ref{ss:mreps}.
  Then $\Orbit(\theta) = \Coker(\sC(X_I))$.
\end{lemma}

The role of orbit modules in the study of average sizes of kernels is due
to the following.

\begin{lemma}[{Cf.\ \cite[\S 2.5]{cico}}]
  \label{lem:ask_orbit}
  Let $S$ be an $R$-algebra which is finite as a set.
  Then
  \[
    \ask{\theta^S} = \frac 1{\card{S J}}  \sum_{x\in SI}
    \card{\Orbit(\theta)_x}.
  \]
\end{lemma}

\paragraph{Ask zeta functions.}
Let $\fO$ be a compact \DVR{} with residue cardinality $q$ and maximal ideal
$\fP$.
Let $\theta\colon M \to \Hom(N,O)$ be a module representation over $\fO$,
where $M$ and $N$ are finitely generated.
\begin{defn}
The \emph{ask zeta function} of $\theta$ is
the formal power series
\[
  \Zeta^{\ak}_{\theta}(T) := \sum_{n=0}^\infty \ask{\theta^{\fO/\fP^n}}T^n.
\]
\end{defn}

Denoting the inclusion of a submodule $M\incl \Hom(N,O)$ simply by $M$, this
notation is consistent with that from the introduction.
Lemma~\ref{lem:ask_orbit} has the following analogue.

\begin{prop}
  \label{prop:orbint}
  Let $\theta\colon M\to\Hom(\fO I,\fO J)$ be a module representation over $\fO$, where
  $I$ and $J$ are finite sets and $M$ is finitely generated.
  Let $d := \card I$ and $e := \card J$.
  Then for all $s\in \CC$ with $\Real(s) \gg 0$,
  \[
    \bigl(1-q^{-s}\bigr) \, \Zeta^{\ak}_{\theta}(q^{-s}) =
    1 + \bigl(1-q^{-1}\bigr)^{-1} \int\limits_{(\fO I\setminus \fP I)\times \fP} \abs{y}^{s-d+e-1}
    \, \card{\Orbit(\theta)_x\otimes \fO/y}\,\dd\mu(x,y),
  \]
  where $\mu$ denotes the normalised Haar measure on $\fO I \times \fO$.
\end{prop}
\begin{proof}
  Combine \cite[Prop.\ 4.17]{ask} and
  \cite[Cor.~2.10]{cico}.
\end{proof}

The computation of the integral in Proposition~\ref{prop:orbint} is particularly
simple whenever the isomorphism type of $\Orbit(\theta)_x$ is independent of $x$.
\begin{cor}[{\cite[\S 5.1]{ask}, \cite[\S 3.6]{ask2}}]
  \label{cor:crkzeta}
  Let the notation be as in Proposition~\ref{prop:orbint}.
  Suppose that there exists $\ell \ge 0$ such that for all $x\in \fO I
  \setminus \fP I$ outside of a set of measure zero,
  we have $\Orbit(\theta)_x \approx \fO^\ell$.
  Then
  $
    \Zeta^{\ak}_{\theta}(T) =
    \frac{1-q^{\ell-e}T}{(1-T)(1-q^{\ell+d-e}T)}$. \qed
\end{cor}

The bulk of the present article is devoted to showing that large and
interesting classes of module representations satisfy the assumptions in
Corollary~\ref{cor:crkzeta} for $\ell = 0$ or $\ell = 1$.

\section{Orbital subrepresentations of module representations}
\label{s:orbital}

As a key ingredient of Theorem~\ref{thm:embedded} and Theorem~\ref{thm:jacobi},
we formulate a sufficient condition which ensures that restricting a module
representation $M\to\Hom(RI,RJ)$ to a submodule $M'\subset M$ preserves
associated ask zeta functions. 

Throughout this section, we assume that all modules are finitely generated and 
all sets $I$, $\tilde I$, $J$, and $\tilde J$ are finite.

\subsection{Orbital subrepresentations}

Let $\theta\colon M\to\Hom(RI,RJ)$ be a module representation. 
If $M'\subset M$ is a submodule and~$\theta'$ denotes the restriction of
$\theta$ to $M'$, then $\orbit(\theta')\subset\orbit(\theta)$.
We therefore obtain a natural $R[X_I]$-module epimorphism $\Orbit(\theta')
\onto \Orbit(\theta)$.

\begin{defn}
  \label{d:orbital}
  $\theta'$ is an \emph{orbital subrepresentation} of $\theta$ if
  for each $R$-algebra $\fO$ which is a
  \DVR{} and each $x\in \fO I\setminus \fP I$ with $\prod x\not= 0$, 
  the natural map $\Orbit(\theta') \onto \Orbit(\theta)$ induces an
  isomorphism $\Orbit(\theta')_x \approx \Orbit(\theta)_x$
  of $\fO$-modules by specialisation.
\end{defn}

Proposition~\ref{prop:orbint} has the following immediate consequence.

\begin{lemma}
  \label{lem:orbital_same_zeta}
  Let $\theta'$ be an orbital subrepresentation of $\theta$.
  Let $\fO$ be an $R$-algebra which is a compact \DVR{}.
  Then $\Zeta^{\ak}_{\theta^\fO}(T) = \Zeta^{\ak}_{(\theta')^\fO}(T)$.
  \qed
\end{lemma}

We will now derive a number of equivalent characterisations of orbital
subrepresentations that we will use in this paper.
Recall that a module $M$ is \emph{hopfian} if each epimorphism of $M$ onto
itself is an automorphism.

\begin{prop}[{\stacks{05G8}}]
  \label{prop:hopf}
  Finitely generated modules over (commutative) rings are hopfian.
\end{prop}

We may thus relax Definition~\ref{d:orbital} as follows.

\begin{cor}
  \label{cor:orbital_iso_epi}
  Let $\theta$ and $\theta'$ be as above.
  Then $\theta'$ is an orbital subrepresentation of~$\theta$ if and only
  if $\Orbit(\theta)_x\approx \Orbit(\theta')_x$
  for each $R$-algebra $\fO$ which is a \DVR{} and all $x\in\fO I\setminus
  \fP I$ with $\prod x\not= 0$.
\end{cor}
\begin{proof}
  Proposition~\ref{prop:hopf} implies that if
  $\Orbit(\theta')_x\approx\Orbit(\theta)_x$, then the natural epimorphism
  from the first onto the second of these modules is an isomorphism.
\end{proof}

Moreover, the restriction to $x\in \fO I$ with $x\not\in\fP I$ in
Definition~\ref{d:orbital} is also unnecessary.

\begin{lemma}
  \label{lem:orbital_general_x}
  Let $\theta'$ be an orbital subrepresentation of $\theta \colon M\to\Hom(RI,RJ)$.
  Let $\fO$ be an $R$-algebra which is a~\DVR{} and let $x\in \fO I$ with $\prod x\not= 0$.
  Then the natural epimorphism $\Orbit(\theta') \onto
  \Orbit(\theta)$ induces an isomorphism $\Orbit(\theta')_x \approx
  \Orbit(\theta)_x$.
\end{lemma}
\begin{proof}
  Let $\omega\colon R[X_I] J \onto \Orbit(\theta)$ and
  $\omega'\colon R[X_I] J \onto \Orbit(\theta')$ be the natural maps.
  Write $x = r y$ for $r\in \fO\setminus\{0\}$ and $y\in \fO
  I\setminus \fP I$ with $\prod y\not= 0$. 
  Let $U := \Ker(\omega_y)$ and $U' := \Ker(\omega_y')$ so
  that $U'\subset U \subset \fO J$.
  By Proposition~\ref{prop:hopf}
  and as $\theta'$ is an orbital subrepresentation of $\theta$, 
  $U = U'$.
  By the definition of $\orbit(\theta)$, $\Ker(\omega_x) = r U$ and analogously
  $\Ker(\omega_x') = r U' = rU$.
  Thus, $\Orbit(\theta')_x \approx\Orbit(\theta)_x$ and the claim follows from
  Corollary~\ref{cor:orbital_iso_epi}.
\end{proof}

We may also characterise orbital subrepresentations using
$\orbit(\dtimes)$ instead of $\Orbit(\dtimes)$.

\begin{lemma}
  \label{lem:orbital_same_orbit}
  Let $\theta'$ be the restriction of $\theta\colon M\to\Hom(RI,RJ)$ to a (finitely
  generated) submodule of $M$.
  Then $\theta'$ is an orbital subrepresentation of $\theta$ if and only if
  the following condition is satisfied:
  for each $R$-algebra $\fO$ which is a \DVR{} and
  each $x\in \fO I$ with $\prod x \not= 0$,
  the images of $\orbit(\theta) \otimes_{R[X_I]} \fO_x$ and $\orbit(\theta')
  \otimes_{R[X_I]} \fO_x$ in $\fO J$ coincide.
\end{lemma}
\begin{proof}
  Sufficiency of the condition is clear.
  Let $\theta'$ be an orbital subrepresentation
  of~$\theta$.
  Let $U$ and $U'$ be the images of $\orbit(\theta) \otimes_{R[X_I]} \fO_x$ and $\orbit(\theta')
  \otimes_{R[X_I]} \fO_x$ in $\fO J$, respectively.
  As $U'\subset U$ and $\fO J/U\approx \fO J/U'$
  by Lemma~\ref{lem:orbital_general_x},
  Proposition~\ref{prop:hopf} shows that $U = U'$.
\end{proof}

\subsection{Fitting ideals}
\label{ss:Fit}

We recall properties of Fitting ideals and provide an indication of
their relevance here.

Let $M$ be a finitely presented $R$-module.
Choose $A\in \Mat_{m\times n}(R)$ with $M \approx \Coker(A)$.
The $i$th \emph{Fitting ideal} $\Fit_i(M)$ of $M$ is the ideal of
$R$ defined as follows.
For $i = 0,\dotsc,n$, $\Fit_i(M)$ is generated by the
$(n-i)\times(n-i)$ minors of $A$;
for $i \ge n$, $\Fit_i(M) := R$.
For more about Fitting ideals and a proof that they are independent of
the chosen presentation, see \cite[\S 1]{Fit36}, \cite[\S 20.2]{Eis95},
\cite[\S 3.1]{Nor76}, or \stacks{07Z8}.

\begin{ex}[{\stacks{07ZB}}]
  \label{ex:Fit_free}
  $\Fit_i(R^n) = 0$ for $i < n$ and $\Fit_i(R^n) = R$ for $i\ge n$.
\end{ex}

\clearpage

\begin{prop}[{\stacks{07ZA}}]
  \label{prop:fitbasic}
  \quad
  \begin{enumerate}[label=(\roman{*})]
  \item
    \label{prop:fitbasic1}
    If $M$ can be generated by $n$ or fewer elements, then $\Fit_i(R) =
    R$ for $i \ge n$.
  \item
    \label{prop:fitbasic2}
    If $Q$ is a quotient of $M$, then $\Fit_i(M)\subset \Fit_i(Q)$.
  \item
    \label{prop:fitbasic3}
    If $S$ is an $R$-algebra, then $\Fit_i(M\otimes S) \normal S$ is generated
    by the image of  $\Fit_i(M)$. 
  \end{enumerate}
\end{prop}

\begin{defn}
  We say that finitely presented $R$-modules $M$ and $N$ are
  \emph{Fitting equivalent} if $\Fit_i(M) = \Fit_i(N)$ for all $i\ge 0$.
\end{defn}

\begin{prop}[{Cf.\ \cite[Satz 10]{Fit36}}]
  \label{prop:Fit_iso}
  Let $R$ be a \PID.
  Let $M$ and $N$ be finitely generated $R$-modules.
  Then $M$ and $N$ are isomorphic if and only if they are Fitting equivalent.
\end{prop}

Our proof of Theorem~\ref{thm:jacobi} will rely on the following.

\begin{cor}
  \label{cor:fiteq_orbit_same_zeta}
  Let $\theta\colon M\to\Hom(RI,RJ)$ be a module representation.
  Let $M'$ be a finitely generated submodule of $M$ and let $\theta'$ denote
  the restriction of $\theta$ to $M'$.
  If $\Orbit(\theta)$ and $\Orbit(\theta')$ are Fitting equivalent,
  then $\theta'$ is an orbital subrepresentation of $\theta$.
\end{cor}
\begin{proof}
  Combine Corollary~\ref{cor:orbital_iso_epi},
  Proposition~\ref{prop:fitbasic}\ref{prop:fitbasic3}, and
  Proposition~\ref{prop:Fit_iso}.
\end{proof}

\subsection{New orbital subrepresentations from old}

Later on, we will use the following two recipes for constructing orbital
subrepresentations.
Recall the definition of an inflation of a module representation from
\S\ref{ss:mreps}.

\begin{lemma}
  \label{lem:orbital_enlarge}
  Let $I \subset \tilde I$ and $J \subset \tilde J$ be finite sets.
  Let $\theta\colon M \to \Hom(RI,RJ)$ be a module representation
  and let $\theta'\colon M'\to \Hom(RI,RJ)$ be an orbital subrepresentation
  of~$\theta$.
  Let $\tilde\theta := \inf^{\tilde I,\tilde J}_{I,J}(\theta)$ and
  $\tilde\theta' := \inf^{\tilde I,\tilde J}_{I,J}(\theta')$ .
  Then $\tilde\theta'$ is an orbital subrepresentation of $\tilde\theta$.
\end{lemma}
\begin{proof}
  As $\orbit(\tilde\theta)$ is the $R[X_{\tilde I}]$-span of the subset
  $\orbit(\theta)\subset R[X_{\tilde I}] \tilde J$,
  the natural map
  $R[X_{\tilde I}]\tilde J = R[X_{\tilde I}] J \oplus R[X_{\tilde I}](\tilde J\setminus J)
    \onto \Orbit(\tilde \theta)$
  induces an isomorphism
  $\Orbit(\tilde\theta) \approx \Orbit(\theta)\otimes_{R[X_I]} R[X_{\tilde I}] \oplus R[X_{\tilde I}]
  (\tilde J\setminus J)$.
  Let $\fO$ be an $R$-algebra which is a \DVR{}.
  Let $\tilde x\in \fO \tilde I$ with $\prod\tilde x\not= 0$
  and let $x\in \fO I$ be the image of $\tilde x$ under
  $\ret\colon \fO \tilde I  \onto \fO I$ (see \S\ref{ss:mreps}).
  Then $\Orbit(\tilde\theta)_{\tilde x} \approx \Orbit(\theta)_x \oplus \fO
  (\tilde J\setminus J)$ as $\fO$-modules and analogously for $\tilde\theta'$.
  The claim now follows from
  Corollary~\ref{cor:orbital_iso_epi} and
  Lemma~\ref{lem:orbital_general_x}.
\end{proof}

Let $(M_a)_{a\in A}$ be a family of finitely generated $R$-modules.
For $a\in A$, let $\theta_a\colon M_a \to \Hom(RI,RJ)$ be a module
representation.
let $[ \theta_a]_{a\in A}^\top$  be the module representation
$\bigoplus_{a\in A} M_a \to \Hom(RI, RJ)$ which sends $(m_a)_{a\in A}$ to
$\sum_{a\in A} m_a\theta_a \in \Hom(RI,RJ)$.
\begin{lemma}
  \label{lem:orbital_sum}
  For $a\in A$, let $\theta_a'$ be an orbital subrepresentation of $\theta_a$.
  Then $[\theta_a']_{a\in A}^\top$ is an orbital subrepresentation
  of $[\theta_a]_{a\in A}^\top$.
\end{lemma}
\begin{proof}
  Let $\bm \theta := [\theta_a]_{a\in A}^\top$ and 
  $\bm \theta' := [\theta_a']_{a\in A}^\top$.
  For $\bm m = (m_a)_{a\in A} \in \bigoplus_{a\in A} M_a$, 
  $\XX(\bm m, \bm \theta) = \sum_{a\in A}
  \XX(m_a,\theta_a)$ and thus
  $\orbit(\bm\theta) = \sum_{a\in A}\orbit(\theta_a)$; analogously for $\bm\theta'$.
  Let $\fO$ be an $R$-algebra which is a \DVR{} and let $x\in \fO I$ with
  $\prod x\not= 0$.
  For $a \in A$,
  as $\theta_a'$ is an orbital subrepresentation of $\theta_a$,
  by Lemma~\ref{lem:orbital_same_orbit},
  the images of $\orbit(\theta_a) \otimes_{R[X_I]} \fO_x$
  and $\orbit(\theta_a') \otimes_{R[X_I]} \fO_x$ in $\fO J$ coincide.
  The same then applies to the images of
  $\orbit(\bm\theta)\otimes_{R[X_I]} \fO_x$ and
  $\orbit(\bm\theta)\otimes_{R[X_I]}\fO_x$ whence the claim follows from Lemma~\ref{lem:orbital_same_orbit}.
\end{proof}

\section{Free class-$3$-nilpotent groups and the Jacobi identity}
\label{s:jacobi}

In this section, we prove Theorem~\ref{thm:jacobi} and we anticipate some
of the ideas and techniques that will eventually lead to our proof of
Theorem~\ref{thm:embedded} in \S\ref{s:board}.

\subsection{Class counting and ask zeta functions}
\label{ss:cc_ask}

We briefly describe the use of Knuth duality (see \S\ref{s:ask}) in the study
of ask zeta functions of adjoint representations. We also recall a
relationship between the latter functions and class counting zeta functions of
unipotent groups.

\paragraph{Ask zeta functions of adjoint representations.}
Given an $R$-algebra $\cA$, not necessarily associative,
the $R$-submodule $\cA^2 := \langle xy : x,y\in \cA\rangle$ is a $2$-sided ideal of $\cA$.
The \emph{centre} of $\cA$ is the $2$-sided ideal $\Zent(\cA) := \{ z\in \cA : z\cA = \cA z = 0\}$.
The \emph{(right) adjoint representation} of $\cA$ is the module
representation
$\ad_{\cA}\colon \cA \to \Hom_R(\cA,\cA), y\mapsto (x\mapsto xy)$.
Let $\cZ \subset \Zent(\cA)$ and $\cA^2 \subset \cD \subset \cA$ be
submodules.
Then $\ad_{\cA}$ induces a module representation $\cA/\cZ \to
\Hom(\cA/\cZ,\cD)$
whose $\bullet$-dual (see \S\ref{ss:mreps})
is
$\alpha\colon \cD^* \to \Hom(\cA/\cZ,(\cA/\cZ)^*),
\delta \mapsto
(
x +\cZ \mapsto
(
(y +\cZ \mapsto (xy) \delta)
)
)$.

\begin{lemma}
  \label{lem:adj_zeta_shift}
  Let each of $\cA$, $\cA/\cZ$, $\cZ$, and $\cD$ be
  a free $R$-module of finite rank. 
  Let $\fO$ be an $R$-algebra which is a compact \DVR{}.
  Then $\Zeta^{\ak}_{\ad_{\cA}^\fO}(T) = \Zeta^{\ak}_{\alpha^\fO}(q^r T)$,
  where $r = \rank_R(\cZ)$.
\end{lemma}
\begin{proof}
  This follows from \cite[Cor.\ 5.6]{ask2} and \cite[Cor.\ 2.3]{ask}.
\end{proof}

For an explicit description of $\alpha$,
choose bases to identify $\cD = RD$ and $\cA/\cZ = R A$.
Let $D^*$ be the corresponding dual basis.
For $a,a'\in A$, let $\std_{a} \std_{a'} = \sum_{d \in D} m(a,a',d)
\std_d$ in $\cA$;
that is, $m(a,a',d) = (\std_{a}\std_{a'})\std_d^*$.
Then $\alpha$ is the module representation 
\[
  R D^* \to \Hom(RA, (RA)^*), \quad \std_d^* \mapsto
  \Bigl(\std_{a} \mapsto \bigl(\std_{a'}\mapsto m(a,a',d)\bigr)\Bigr).
\]

\paragraph{Class counting zeta functions of unipotent group schemes.}
We sketched the following in \S\ref{ss:adjoint}.
Let $\fg$ be a nilpotent Lie $R$-algebra of class at most $c$ whose
underlying $R$-module is free of finite rank.
Suppose that $c!\in R^\times$.
Let $\exp(\fg)$ be the group
attached to $\fg$ via the Lazard correspondence.
That is, the underlying set of $\exp(\fg)$ is $\fg$ and 
the group multiplication is given by the Baker-Campbell-Hausdorff formula.
For an $R$-algebra $S$, let $\sG(S) := \exp(\fg\otimes S)$.
Then $\sG$ ``is'' (i.e.\ represents) a group scheme over $R$.

\begin{prop}[{\cite[Cor.\ 6.6]{ask2}}]
  \label{prop:cc_via_ad}
  Let $\fO$ be an $R$-algebra which is a compact \DVR{}.
  Then
  \[
    \Zeta^{\cc}_{\sG\otimes \fO}(T) = \Zeta^{\ak}_{\ad_{\fg}^\fO}(T).
  \]
\end{prop}
  
\subsection{Free class-$3$-nilpotent Lie algebras: discarding the Jacobi identity}
\label{ss:jacobi_no_more}

By Proposition~\ref{prop:cc_via_ad}, the class counting zeta functions in
Theorem~\ref{thm:jacobi} are ask zeta functions associated with
adjoint representations of free class-$3$-nilpotent Lie algebras.
We will now see that we may replace the latter by free
class-$3$-nilpotent algebras (not necessarily associative) that merely satisfy the identity
$x^2 = 0$---that is, we may discard the Jacobi identity.

\paragraph{Notation for sets of natural numbers.}
To simplify our notation, in this section (and only here), for $h,i,j,k\in \NN$ with $i < j <k$,
we write $\{i<j\} = \{i,j\}$, $\{i<j<k\} = \{i,j,k\}$,
and $(h,i<j) = (h,\{i,j\}) \in \NN\times \binom \NN 2$.

\paragraph{Defining the algebra $\cA(I)$.}
For $I \in \Powf(\NN)$, let
$\cA(I)$ be the largest not necessarily associative $\ZZ$-algebra generated (as an algebra)
by symbols $\std_i$ for $i\in I$ and which satisfies the identities $x^2 = 0$
and $x_1(x_2(x_3x_4)) = 0$ for all $x,x_1,x_2,x_3,x_4$.
Explicitly, we may write $\cA(I) = \ZZ I \oplus \ZZ \binom I 2 \oplus \cZ(I)$,
where $\cZ(I) := \ZZ(I\times \binom I 2)$ is central and for $h,i,j\in I$ with $i
< j$, we have $\std_i\std_j = \std_{\{i<j\}}$ and $\std_h\std_{\{i<j\}} =
\std_{(h,i<j)}\in \cZ(I)$.

\paragraph{Defining $\hat\cA(I)$: free class-$3$-nilpotent Lie algebras.}
The \emph{Jacobi ideal} of $\cA(I)$ is
\[
  \cJ(I) :=
  \left\langle
    \std_{(i,j<k)} -
    \std_{(j,i<k)} +
    \std_{(k,i<j)}
    \,:\, \{i<j<k\} \in \binom I 3
  \right\rangle_\ZZ
  \subset\cZ(I).
\]

Note that $\cJ(I)$ is a direct summand of $\cZ(I)$ as a $\ZZ$-module
and that $\cJ(I)$ and $\cZ(I)/\cJ(I)$ are both free.
Let $\hat\dtimes\colon \cA(I) \onto \hat\cA(I) := \cA(I)/\cJ(I)$
be the quotient map.
It is easy to see that the Jacobi identity
$x(yz) + y(zx) + z(xy) = 0$
holds for all $x,y,z \in \hat\cA(I)$. We conclude that $\hat\cA(I)$ is the free nilpotent Lie
$\ZZ$-algebra of nilpotency class at most~$3$ (freely) generated by the
$\std_i$ ($i\in I$).

\paragraph{The adjoint representations of $\cA(I)$ and $\hat\cA(I)$: defining
  $\alpha(I)$ and $\hat\alpha(I)$.}
Let $\cD(I) := \cA(I)^2$ and $\hat\cD(I) := \hat\cA(I)^2$.
Clearly, $\cD(I) = \ZZ \binom I 2 \oplus \cZ(I)$ and $\hat\cD(I) =
\ZZ\binom I 2 \oplus \hat\cZ(I)$, where $\hat\cZ(I) := \cZ(I)/\cJ(I)$.
In particular, the $\ZZ$-modules $\cD(I)$ and $\hat\cD(I)$ are both free.
As in \S\ref{ss:adjoint}, the adjoint representation of $\cA(I)$ gives rise to a module
representation $\cA(I)/\cZ(I) \to \Hom\bigl(\cA(I)/\cZ(I),\cD(I)\bigr)$.
Let
\[
  \alpha(I)\colon \cD(I)^* \to \Hom\bigl(\cA(I)/\cZ(I),(\cA(I)/\cZ(I))^*\bigr)
  \]
be its $\bullet$-dual.
By ``adding $\widehat{\,\dtimes\,}$s'',
we analogously define $\hat\alpha(I)$.
Since $\cJ(I) \subset \cZ(I)$, we may canonically identify $\cA(I)/\cZ(I) =
\hat\cA(I)/\hat\cZ(I)$.
We may further identify $\hat\cD(I)^* = (\cD(I)/\cJ(I))^*$ and the
``orthogonal complement'' $\cJ(I)^\perp \subset \cD(I)^*$;
see e.g.\ \cite[Ch.~II, \S 2, no.\ 6]{Bou70}
In the following, we thus regard $\hat\alpha(I)$ as the restriction of
$\alpha(I)$ to $\hat\cD(I)^*\subset \cD(I)^*$.

\paragraph{Relating $\alpha(I)$ and $\hat\alpha(I)$.}
The following result, proved below, is the main contribution of this section
towards a proof of Theorem~\ref{thm:jacobi}.

\begin{prop}
  \label{prop:jacobi_orbital}
  $\hat\alpha(I)$ is an orbital subrepresentation of $\alpha(I)$
  for each $I\in \Powf(\NN)$.
\end{prop}

Note that $\alpha(I) = \hat\alpha(I)$ if and only if $\card I < 3$.
We may thus assume that $\card I \ge 3$ in the following.
In fact, the first step of our proof will be a reduction to the case $\card I
= 3$.
We first record the following consequence of
Proposition~\ref{prop:jacobi_orbital}.

\begin{cor}
  \label{cor:jacobi_via_alpha}
  Let $\fO$ be a compact \DVR{} with residue cardinality $q$.
  If $\gcd(q,6) = 1$, then
  \[
    \Zeta^\cc_{\sF_{3,d}\otimes\fO}(T)
    = \Zeta^\ak_{\alpha([d])^{\fO}}(q^{2\binom{d+1} 3}T).
  \]
\end{cor}
\begin{proof}
  Clearly, $\hat\cZ([d])$ and $\cJ([d])$ are free $\ZZ$-modules,
  the latter of rank $\binom d 3$ and the former of rank
  $d\binom d 2 - \binom d 3 = 2\binom{d+1} 3$.
  The claim thus follows from Proposition~\ref{prop:cc_via_ad},
  Lemma~\ref{lem:adj_zeta_shift}, and Lemma~\ref{lem:orbital_same_zeta}.
\end{proof}

In the final stage of our proof of Theorem~\ref{thm:jacobi} in \S\ref{ss:graphs},
we will then interpret $\alpha([d])$ in terms of ``adjacency representations'' of
threshold graphs from \cite{cico}.

\paragraph{An explicit description of $\alpha(I)$.}
By choosing bases, we can describe $\alpha(I)$ more explicitly as follows.
Let \fbox{$\sD(I) := \binom I 2 \cup (I\times\binom I 2)$} (disjoint union).
Using the notation for subsets of $\NN$ from above,
$\cD(I)$ has a $\ZZ$-basis consisting of the elements $\std_{\{i<j\}}$ for
$\{i<j\}\in \binom I 2$ and $\std_{(h,i<j)}$ for $(h,i<j)\in I\times\binom I 2$;
let the $\std_{\{i<j\}}^*$ and $\std_{(h,i<j)}^*$ comprise the corresponding
dual basis of $\cD(I)^*$.
Using said dual basis, we henceforth identify $\cD(I)^* = \ZZ\, \sD(I)$.

Let \fbox{$\sB(I) := I \cup \binom I 2$} (disjoint union).
The images of the $\std_i$ 
and the $\std_{\{i<j\}}$ 
form a basis of $\cA(I)/\cZ(I)$.
By identifying $(\cA(I)/\cZ(I))^* = \cA(I)/\cZ(I) = \ZZ \,\sB(I)$ via the corresponding
dual bases, we regard $\alpha(I)$ as a map $\ZZ\,\sD(I) \to
\Hom(\ZZ \,\sB(I),\ZZ \,\sB(I))$.
Explicitly, $\std_{\{i<j\}} \alpha(I) = \std_{ij} -
\std_{ji}$ and $\std_{(h,i<j)}\alpha(I) = \std_{h,{\{i<j\}}} -
\std_{{\{i<j\}},h}$.
The following is clear; recall the definitions of $\ret$ and $\inc$ from \S\ref{ss:mreps}.

\begin{lemma}
  \label{lem:alpha_compat}
  Let $I\subset \tilde I \in\Powf(\NN)$.
  Then the following diagram commutes:
  \[
    \begin{gathered}[b]
    \xymatrix@C+2em{
      \ZZ\, \sD(I) \ar[r]^{\alpha(I)\phantom{XYZXY}}\ar@{^{(}->}[d]_{\inc} & \Hom(\ZZ\,\sB(I),\ZZ\,\sB(I)) \ar[d]^{\Hom(\ret,\inc)} \\
      \ZZ\, \sD(\tilde I)\ar[r]_{\alpha(\tilde I)\phantom{XYZXY}} & \Hom(\ZZ\,\sB(\tilde I),\ZZ\,\sB(\tilde I)).
    }
  \\[-\dp\strutbox]
  \end{gathered}
  \pushQED{\qed}
  \qedhere
  \popQED
  \]
\end{lemma}

\paragraph{An explicit description of $\hat\alpha(I)$.}
Prior to describing $\hat\alpha(I)$ (similarly to $\alpha(I)$ from above),
we first investigate $\hat\cD(I)^*$.
To that end, we will use the following simple observation.

\begin{lemma}
  \label{lem:ranger}
  Let $a = [a_{ij}]\in \Mat_{d\times e}(R)$.
  For $j=1,\dotsc,e$, let $\Sigma_j := \{ i : 1\le i\le
  d, a_{ij} \not= 0\}$ be the support of
  the $j$th column of $a$.
  Suppose that $\Sigma_j\cap \Sigma_k = \emptyset$ for $j\not= k$.
  Further suppose that each non-zero entry of $a$ is a unit of $R$.
  \begin{enumerate}[label=(\roman{*})]
  \item \label{lem:ranger1}
    Let $J := \{ j : \Sigma_j \not= \emptyset\}$.
    For each $j\in I$, choose $\sigma(j) \in \Sigma_j$.
    Then the following elements of $R^d$ comprise a basis of $\Ker(a)$:

    \begin{align*}
      \std_i - \frac{a_{ij}}{a_{\sigma(j)j}} \std_{\sigma(j)} \quad
      &
        (j \in J, i\in \Sigma_j, i\not= \sigma(j)),\\
      \std_i \quad & (1\le i\le d, i\not\in \Sigma_1\cup\dotsb\cup\Sigma_e).
    \end{align*}
    Moreover, $R^d = \Ker(a) \oplus \langle \std_{\sigma(j)} : j\in J\rangle$.
  \item \label{lem:ranger2}
    Let $\cB$ be a basis of $\Ker(a)$ as in \ref{lem:ranger1}.
    Let $S$ be an $R$-algebra.
    Then the natural map $\Ker(a) \otimes S \to \Ker(a\otimes S)$ (induced by
    $\Ker(a) \incl R^d
    \to S^d$) is an $S$-module isomorphism
    and the images of the elements of $\cB$ in $S^d$ form 
    an $S$-basis of $\Ker(a\otimes S)$.
  \end{enumerate}
\end{lemma}
\begin{proof}
  Part \ref{lem:ranger1} is clear.
  Let $\cB$ be a basis as defined there.
  As $\cB$ is an $R$-basis of $\Ker(a)$,
  $\cB\otimes S$ is an $S$-basis of $\Ker(a)\otimes S$.
  By applying \ref{lem:ranger1} to $a\otimes S$ over $S$,
  we see that the natural map $\Ker(a)\otimes S \to S^d$
  maps $\cB\otimes S$ onto an $S$-basis of $\Ker(a\otimes S)$.
\end{proof}

Lemma~\ref{lem:ranger} shows that $\hat\cD(I)^* = \cJ(I)^\perp\subset
\cD(I)^*$ (see above) has a basis consisting of the following elements of
$\cD(I)^* = \ZZ\, \sD(I)$:
\begin{align}
  \label{eq:Dhat_basis}
  \begin{split}
  \std_{\{i<j\}}, \quad
  \std_{(i,i<j)}, \quad
  \std_{(j,i<j)},
  \qquad
  & \left(\{i<j\} \in \binom I 2\right),
  \\
  \std_{(i,j<k)} - \std_{(k,i<j)},
  \quad
  \std_{(j,i<k)} + \std_{(k,i<j)}
  \qquad
  & \left(\{i<j<k\} \in \binom I 3\right).
  \end{split}
\end{align}
Let \fbox{$\hat\sD(I) := (\binom I 2 \times [3]) \cup (\binom I 3 \times [2])$} and
identify $\hat\cD(I)^* = \ZZ\,\hat\sD(I)$
using \eqref{eq:Dhat_basis}
and the order suggested by our notation.
(For example, $(\{i<j\},2)\in \binom I 2\times[3]$ corresponds to
$\std_{(i,i<j)}$.)
Each element of $\hat\sD(I)$ corresponds to one of the elements
of $\cD(I)^* = \ZZ\, \sD(I)$ in~\eqref{eq:Dhat_basis}.
This gives rise to an injection $\eta(I)\colon \ZZ\, \hat\sD(I) \into \ZZ\,
\sD(I)$ with $\hat\alpha(I) = \eta(I) \alpha(I)$.
By construction, for $I\subset \tilde I\in\Powf(\NN)$, the following diagram commutes:
\[
  \xymatrix@C+1em@H+0.5em{
    \ZZ\, \hat\sD(I) \phantom{.} \ar@{>->}[r]^{\eta(I)} \ar@{^{(}->}[d]^{\inc} & \ZZ \, \sD(I) \ar@{^{(}->}[d]^{\inc} \\
    \ZZ\, \hat\sD(\tilde I) \phantom{.} \ar@{>->}[r]_{\eta(\tilde I)} &\ZZ \, \sD(\tilde I).
  }
\]

The following is thus a consequence of Lemma~\ref{lem:alpha_compat}.

\begin{lemma}
  \label{lem:alphahat_compat}
  Let $I\subset \tilde I \in\Powf(\NN)$.
  Then the following diagram commutes:
  \[
    \begin{gathered}[b]
    \xymatrix@C+2em{
      \ZZ\, \hat\sD(I) \ar[r]^{\hat\alpha(I)\phantom{XYZXY}}\ar@{^{(}->}[d]_{\inc} & \Hom(\ZZ\,\sB(I),\ZZ\,\sB(I)) \ar[d]^{\Hom(\ret,\inc)} \\
      \ZZ\, \hat\sD(\tilde I)\ar[r]_{\hat\alpha(\tilde I)\phantom{XYZXY}} & \Hom(\ZZ\,\sB(\tilde I),\ZZ\,\sB(\tilde I)).
    }
  \\[-\dp\strutbox]
  \end{gathered}
  \pushQED{\qed}
  \qedhere
  \popQED
  \]
\end{lemma}

\paragraph{Reduction of Proposition~\ref{prop:jacobi_orbital} to the case $\card I = 3$.}
Let $I\in \Powf(\NN)$ with $\card I \ge 3$.
Clearly, $\sD(I) = \bigcup_{T\in \binom I 3} \sD(T)$ and
$\hat\sD(I) = \bigcup_{T\in \binom I 3} \hat\sD(T)$.
Let $\pi(I)\colon \bigoplus_{T\in \binom I 3}\ZZ\,\sD(T) \onto \ZZ \,\sD(I)$ be
induced by the inclusions $\sD(T) \incl \sD(I)$. 
Define $\hat\pi(I) \colon \bigoplus_{T\in \binom I 3}\ZZ\,\hat\sD(T) \onto \ZZ
\,\hat\sD(I)$ analogously.
Recall the definition of an inflation of a module representation from
\S\ref{ss:mreps}.
For $T \in \binom I 3$,
Lemma~\ref{lem:alpha_compat} shows that
the restriction of $\alpha(I)$ to $\ZZ\,\sD(T)$ coincides with
${\inf^{\sB(I)}_{\sB(T)}(\alpha(T))}$.
We conclude that
$\pi(I)\alpha(I) = \bigl[\inf^{\sD(I)}_{\sD(T)}(\alpha(T))\bigr]_{T\in \binom I 3}^\top$
and, using Lemma~\ref{lem:alphahat_compat},
analogously for~$\hat\pi(I)\hat\alpha(I)$.
Note that we may regard $\hat\pi(I)\hat\alpha(I)$ as the restriction of
$\pi(I)\alpha(I)$ to a submodule.
Since an orbit module $\Orbit(\theta)$ only depends on the image of $\theta$,
Proposition~\ref{prop:jacobi_orbital} is equivalent to
$\hat\pi(I)\hat\alpha(I)$ being an orbital subrepresentation of $\pi(I)\alpha(I)$.
Lemmas~\ref{lem:orbital_enlarge}--\ref{lem:orbital_sum} now reduce the latter property to the case $\card I = 3$.

\paragraph{Final step towards Proposition~\ref{prop:jacobi_orbital}: the case $\card I = 3$.}
\begin{lemma}
  Let $T \in \binom \NN 3$.
  Then $\hat\alpha(T)$ is an orbital subrepresentation of $\alpha(T)$.
\end{lemma}
\begin{proof}
  We may assume that $T = \{ 1,2,3\}$.
  Using a suitable computer algebra system, one may verify that
  $\Orbit(\hat\alpha(T))$ and $\Orbit(\alpha(T))$ are Fitting equivalent 
  whence the claim follows by Corollary~\ref{cor:fiteq_orbit_same_zeta}.
  In the following, we include explicit details to allow the reader to repeat
  this calculation.
  Order the elements of $\sB(T)$ and $\sD(T)$ lexicographically as
  $(1,2,3,\{1<2\},\{1<3\},\{2<3\})$ and
  $(\{1<2\},\{1<3\},\{2<3\}, (1, 1<2), (1,1<3),(1,2<3), \dotsc, (3,1<2),
  (3,1<3), (3,2<3))$, respectively.
  We see that $\alpha(T)$ is isotopic to the module representation associated
  with the matrix  of linear forms

  {\small
    \[
      \sA(X_1,\dotsc,X_{12}) = 
      \begin{bmatrix}
        0 & X_1 & X_2 & X_4 & X_5 & X_6 \\
        -X_1 & 0 & X_3 & X_7 & X_8 & X_9 \\
        -X_2 & -X_3 & 0 & X_{10}& X_{11} & X_{12} \\
        -X_4 & -X_7 & -X_{10} & 0 & 0 & 0 \\
        -X_5 & -X_8 & -X_{11} & 0 & 0 & 0 \\
        -X_6 & -X_9 & -X_{12} & 0 & 0 & 0
      \end{bmatrix}.
    \]
  }

  Regarding $\hat\alpha(T)$,
  by ordering the basis \eqref{eq:Dhat_basis} as

  {\small
    \begin{align*}
      \Bigl(
      \std_{\{1<2\}},\,
      \std_{\{1<3\}},\,
      \std_{\{2<3\}},\,
      \std_{(1,1<2)},\,
      \std_{(1,1<3)},\
      &\std_{(1,2<3)} - \std_{(3,1<2)},\,
        \std_{(2,1<2)},\\
      &\std_{(2,1<3)} + \std_{(3,1<2)},\,
        \std_{(2,2<3)},\,
        \std_{(3,1<3)},\,
        \std_{(3,2<3)}
        \Bigr),
    \end{align*}
  }

    we find that $\hat\alpha(T)$ is isotopic to the module representation
    associated with

    {\small
      \[
        \hat \sA(X_1,\dotsc,X_{11})= 
        \begin{bmatrix}
          0 & X_1 & X_2 & X_4 & X_5 & X_6 \\
          -X_1 & 0 & X_3 & X_7 & X_8 & X_9 \\
          - X_2 & -X_3 & 0 & -X_6 +X_8 & X_{10} & X_{11} \\
          -X_4 & -X_7 & +X_6 - X_8 & 0 & 0 & 0 \\
          -X_5 & -X_8 & -X_{10} & 0 & 0 & 0 \\
          -X_6 & -X_9 & -X_{11} & 0 & 0 & 0
        \end{bmatrix}.
      \]
    }
    
  As in Lemma~\ref{lem:Cmatrix},
  let $\sC = \sC(X_1,\dotsc,X_6)$ and $\hat \sC = \hat \sC(X_1,\dotsc,X_6)$ be the
   $\circ$-dual matrices (see \S\ref{ss:mreps})
  associated with $\sA(X_1,\dotsc,X_{12})$ and $\hat \sA(X_1,\dotsc,X_{11})$,
  respectively.
  Explicitly,
  
  {\scriptsize
  \begin{align*}
    \sC &=
    \begin{bmatrix}
      -X_{2} & X_{1} & 0 & 0 & 0 & 0 \\
      -X_{3} & 0 & X_{1} & 0 & 0 & 0 \\
      0 & -X_{3} & X_{2} & 0 & 0 & 0 \\
      -X_{4} & 0 & 0 & X_{1} & 0 & 0 \\
      -X_{5} & 0 & 0 & 0 & X_{1} & 0 \\
      -X_{6} & 0 & 0 & 0 & 0 & X_{1} \\
      0 & -X_{4} & 0 & X_{2} & 0 & 0 \\
      0 & -X_{5} & 0 & 0 & X_{2} & 0 \\
      0 & -X_{6} & 0 & 0 & 0 & X_{2} \\
      0 & 0 & -X_{4} & X_{3} & 0 & 0 \\
      0 & 0 & -X_{5} & 0 & X_{3} & 0 \\
      0 & 0 & -X_{6} & 0 & 0 & X_{3}
    \end{bmatrix} \text{ and }
    \hat \sC =
    \begin{bmatrix}
      -X_{2} & X_{1} & 0 & 0 & 0 & 0 \\
      -X_{3} & 0 & X_{1} & 0 & 0 & 0 \\
      0 & -X_{3} & X_{2} & 0 & 0 & 0 \\
      -X_{4} & 0 & 0 & X_{1} & 0 & 0 \\
      -X_{5} & 0 & 0 & 0 & X_{1} & 0 \\
      -X_{6} & 0 & X_{4} & -X_{3} & 0 & X_{1} \\
      0 & -X_{4} & 0 & X_{2} & 0 & 0 \\
      0 & -X_{5} & -X_{4} & X_{3} & X_{2} & 0 \\
      0 & -X_{6} & 0 & 0 & 0 & X_{2} \\
      0 & 0 & -X_{5} & 0 & X_{3} & 0 \\
      0 & 0 & -X_{6} & 0 & 0 & X_{3}
    \end{bmatrix}.
  \end{align*}
}

  By Gr\"obner bases calculations using
  Macaulay2~\cite{M2} or SageMath~\cite{SageMath} (which uses Singular~\cite{Singular}),
  the matrices $\sC$ and $\hat \sC$ have the same ideals of $i\times i$ minors within
  $\ZZ[X_1,\dotsc,X_6]$ for all $i$.
  The claim thus follows by combining Lemma~\ref{lem:Cmatrix} and
  Corollary~\ref{cor:fiteq_orbit_same_zeta}.
\end{proof}

This completes the proof of Proposition~\ref{prop:jacobi_orbital}.

\begin{rem}
  O'Brien and Voll~\cite[Prop.\ 5.9]{O'BV15} determined the ``character
  vector'' of~$\sF_{3,3}(\FF_q)$ (for $\gcd(q,6) = 1$) by studying the rank
  loci of a suitable ``commutator matrix''~$B(\bf Y)$.
  (Character vectors specialise to class numbers via the well-known identity
  $\concnt(G) = \#\operatorname{Irr}(G)$ for a finite group $G$.)
  Up to harmless transformations and an (equally harmless) sign error, the matrix $B(\bf Y)$
  in \cite[Prop.\ 5.9]{O'BV15} coincides with our
  $\hat \sA(X_1,\dotsc,X_{11})$ in the preceding proof.
  We note that the interplay between the enumeration of conjugacy classes and
  characters of unipotent groups in \cite{O'BV15} can be expressed in terms of
  the duality operation $\bullet$ for module representations;
  see \cite[\S 6.2]{ask2}.
\end{rem}

\begin{rem}[Universal Jacobi identities]
  The strategy underlying our proof of Proposition~\ref{prop:jacobi_orbital}
  admits the following generalisation in the spirit of
  Theorem~\ref{thm:embedded}.
  Let $R := \ZZ[ \acute u_{hij}^{\pm 1} : \{h,i,j\} \in \binom \NN 3, i < j]$,
  where the $\acute u_{hij}$ are algebraically independent over $\ZZ$.
  For $I\in \Powf(\NN)$,
  we then obtain a ``universal Jacobi ideal'' with unit coefficients
  \[
    \cJ_u(I) :=
    \left\langle
      \acute u_{ijk} \,\std_{(i,j<k)} +
    \acute u_{jik} \,\std_{(j,i<k)} +
    \acute u_{kij} \,\std_{(k,i<j)}
    : \{i<j<k\} \in \binom I 3
  \right\rangle_{\!R}
  \subset \cA(I)\otimes R.
  \]
  Let $\hat\cA_u(I) := (\cA(I)\otimes R)/\cJ_u(I)$ and define a module
  representation $\hat\alpha_u(I)$ over $R$ analogously to the construction of
  $\hat\alpha$.
  A suitable specialisation $R\to \ZZ$ then provides identifications
  $\hat\cA(I) = \hat\cA_u(I)\otimes \ZZ$ and $\hat\alpha = \hat\alpha_u^\ZZ$,
  the latter of which is based on Lemma~\ref{lem:ranger}\ref{lem:ranger2}.
  Following the same strategy as above and using Macaulay2~\cite{M2} to perform
  Gr\"obner bases calculations over (finitely generated subrings of) $R$, we
  find that $\hat\alpha_u$ is an orbital subrepresentation of $\alpha^R$.
  Hence, if $\fO$ is a compact \DVR{} endowed with a ring map $R\to \fO$,
  then $\Zeta^{\ak}_{\hat\alpha(I)_u^\fO}(T) =
  \Zeta^{\ak}_{\alpha(I)^\fO}(T)$, irrespective of the specific choice of
  units of $\fO$ that defines the map $R \to \fO$.
  In \S\S\ref{s:rel}--\ref{s:board},
  we will make very similar use of ``large'' Laurent polynomial rings over $\ZZ$ to
  model universal linear relations with unit coefficients as in
  Theorem~\ref{thm:embedded}. 
\end{rem}

\subsection{Graphs and a proof of Theorem~\ref{thm:jacobi}}
\label{ss:graphs}

Having established Proposition~\ref{prop:jacobi_orbital} (and thus
Corollary~\ref{cor:jacobi_via_alpha}), 
the final step in our proof of Theorem~\ref{thm:jacobi} is to determine the
ask zeta functions associated with the module representations~$\alpha(I)$.
As we will now explain, the latter goal has been achieved in~\cite{cico}.

\paragraph{Adjacency representations of graphs.}
Let $G = (V,E)$ be a simple graph, where $V$ is finite and $E\subset \binom V
2$.
Following \cite{cico}, the \emph{(negative) adjacency representation}
associated with $G$ is the module representation
\[
  \gamma\colon \ZZ E \to \Hom(\ZZ V, \ZZ V),
  \quad
  \std_{\{ v < w \}} \mapsto \std_{vw} - \std_{wv},
\]
where $<$ is an arbitrary total order on $V$ and, as above,
we write $\{ v < w\} = \{v,w\}$ for~$v < w$.
Up to isotopy, our definition of $\gamma$ is independent of $<$.
Indeed, an alternative, intrinsic construction of adjacency representations is
provided in \cite[\S 3.3]{cico}.
An isotopy between the two constructions can be found in the proof of \cite[Prop.\ 3.7]{cico}.
The following is one of the main results of \cite{cico}.

\begin{thm}[{\cite[Thm~A(ii)]{cico}}]
  \label{thm:cico_uniformity}
  Let $G$ be a finite simple graph with adjacency representation $\gamma$ as above.
  Then there exists $W_G(X,T)\in \QQ(X,T)$
  such that for each compact \DVR{} $\fO$ with residue cardinality $q$, we have
  $\Zeta^{\ak}_{\gamma^{\fO}}(T) = W_G(q,T)$.
\end{thm}

\paragraph{Viewing $\alpha(I)$ as an adjacency representation.}
Let $\CG_n$ and $\DG_n$ denote the complete graph and discrete graph on $n$
vertices, respectively.
(The latter graph has no edges and is also referred to as a ``null graph'' or an
``empty graph'' in the literature.)
Given graphs $G = (V,E)$ and $G'=(V',E')$, their \emph{join} $G\join G'$ is
the graph constructed as follows.
The vertex set of $G\join G'$ is the disjoint union of $V$ and $V'$.
Two vertices of $G\join G'$ are adjacent if and only if either (a) they both
belong to $V$ (resp.~$V'$) and are adjacent in $G$ (resp.~$G'$) or
(b) one of them belongs to $V$ and the other to $V'$.

Let $I\in \Powf(\NN)$ with $d = \card I$.
The explicit description of $\alpha(I)$ that precedes
Lemma~\ref{lem:alpha_compat} shows that $\alpha(I)$ is isotopic to the
adjacency representation of $\DG_{\binom d 2} \join \CG_d$.
Using the notation from \cite[\S 8.4]{cico}, $\Delta_m \join \CG_n$ is the
threshold graph $\mathsf{Thr}(m,n)$.
The following is now an immediate consequence of \cite[Thm~8.18]{cico}.

\begin{prop}
  \label{prop:kite}
  $\displaystyle
  W_{\DG_m \join \CG_n}(X,T) =
  \frac{ (1 - X^{1-n}T)(1-X^{-n}T)}{(1-T)(1-XT)(1-X^{m-n}T)}.
  $
\end{prop}

Theorem~\ref{thm:jacobi} follows from Corollary~\ref{cor:jacobi_via_alpha} and
Proposition~\ref{prop:kite} (with $(m,n) = \bigl(\binom d 2,d\bigr)$).

\section{Coherent families of module representations}
\label{s:coh}

In this section, we describe the effects of the
operations of deleting rows or columns on orbit modules associated with a
given module representation.
This will constitute a key ingredient of our recursive proofs of
Theorem~\ref{thm:embedded} and Corollaries~\ref{cor:rec}--\ref{cor:sym}.

\subsection{Definitions}
\label{ss:cohdef}

Let $A\subset B$ be sets.
Analogously to the notation from \S\ref{ss:mreps},
we denote the canonical \emph{retraction} $R[X_B] \onto R[X_B]/\langle
X_{B\setminus A}\rangle \approx R[X_A]$ by $\ret = \ret_{B,A}$.
We tacitly regard each $R[X_A]$-module as an $R[X_B]$-module by restriction of
scalars via $\ret$.

\begin{defn}
  \label{d:cohfam}
  A \emph{coherent family of module representations}
  \[\bTheta =
  \Bigl(\theta(I,J);\varphi^{\tilde I,\tilde J}_{I,J}\Bigr)_{I\subset \tilde
    I,J\subset \tilde J\in \Powf(\NN)}\]
  over $R$ consists of the following data:
  \begin{enumerate}[label=(\alph*)]
  \item
    \label{d:natfama}
    For all $I,J\in \Powf(\NN)$, 
    a finitely generated $R$-module $M(I,J)$ and a module representation
    $\theta(I,J)\colon M(I,J) \to \Hom(RI,RJ)$.
  \item
    \label{d:natfamb}
    For all $I \subset \tilde I \in \Powf(\NN)$ and $J\subset \tilde J\in \Powf(\NN)$,
    a \emph{transition map}
    ${\varphi_{I,J}^{\tilde I,\tilde J}}\colon M(\tilde  I,\tilde J) \to M(I,J)$
    such that
    $\varphi^{\tilde I,\tilde J}_{I,J} \dtimes \theta(I,J) = \resn^{\tilde
      I,\tilde J}_{I,J}(\theta(\tilde I,\tilde J))$ (see \S\ref{ss:mreps}).
  \end{enumerate}
  For notation simplicity, we usually simply write $\bTheta =
  \bigl(\theta(I,J)\bigr)_{I,J\in \Powf(\NN)}$ in the following.
\end{defn}

\begin{defn}
  \label{d:basfam}
  Let $\bTheta$ be as in Definition~\ref{d:cohfam}.
  We say that $\bTheta$ is a \emph{basic family of module
    representations} if, in addition to \ref{d:natfama}--\ref{d:natfamb} in
  Definition~\ref{d:cohfam}, the following conditions are satisfied: 
    \begin{enumerate}[start=3]
    \item
      For $I,J\in \Powf(\NN)$, the module
      $M(I,J)$ is free of the form $M(I,J) = R \,\sB(I,J)$ for a (designated) finite set
      $\sB(I,J)$,
    \item
      $\sB(I,J) \subset \sB(\tilde I,\tilde J)$
      for all $I \subset \tilde I \in \Powf(\NN)$ and $J\subset \tilde J\in
      \Powf(\NN)$.
    \item
      All transition maps are retractions (see \S\ref{ss:mreps}):
      $\varphi^{\tilde I,\tilde J}_{I,J} = \ret\colon R \, \sB(\tilde I,\tilde J) \onto
      R \, \sB(I,J)$.
    \end{enumerate}
    We regard the datum $(\sB(I,J))_{I,J\in\Powf(\NN)}$
    as part of a basic family of module representations.
\end{defn}

\begin{rem}
  Let $\bTheta$ be a basic family of module representations as above.
  Let~$\sA_{I,J}$ be the $I\times J$ matrix of linear forms in $R[X_{\sB(I,J)}]$
  associated with $\theta(I,J)$ as in~\S\ref{ss:mreps}.
  For $I\subset \tilde I$ and $J\subset \tilde J$, 
  $\sA_{I,J}$ is then obtained from $\sA_{\tilde I,\tilde J}$ by
  deleting all rows indexed by $\tilde I\setminus I$ and
  all columns indexed by $\tilde J\setminus J$.
  Note that,
  by construction, no variable~$X_{\tilde b}$ for $\tilde b\in \sB(\tilde I,\tilde J)\setminus\sB(I,J)$
  features in the submatrix of $\sA_{\tilde I,\tilde J}$ indexed by $I\times J$.
\end{rem}

\begin{rem}
  Although we shall not pursue this further in the present article, we note
  that there are various natural ways of rephrasing the preceding notions
  in categorical language.
  In particular, a coherent family of module representations
  gives rise to a $\Powf(\NN)^2$-indexed direct system in the category
  $\mathsf{mod}_{\downarrow\uparrow\downarrow}(R)$ from \cite[Defn~2.2]{ask2}.
  Much about such families could then be expressed in terms of limits.
\end{rem}

\subsection{Main examples}
\label{ss:basfam_exs}

We construct basic families of module
representations $\bRho$, $\bGamma$, and $\bSigma$ related to
Theorem~\ref{thm:embedded}.
For each of these, in the setting of Definition~\ref{d:basfam},
we specify $\sB(I,J)$ and $\theta(I,J)$ and leave the verification
of the transition conditions in Definition~\ref{d:cohfam}\ref{d:natfamb} to the reader.
We also describe the orbit modules associated with each $\theta(I,J)$ for later
use.

\begin{ex}[Generic rectangular matrices]
  \label{ex:rec}
  We define a basic family
  \[
    \bRho = \Bigl(\rho(I,J)\Bigr)_{I,J\in\Powf(\NN)}
  \]
  of module representations over $R$ as follows.
  Let $\sB(I,J) := I\times J$ and
  define the map ${\rho(I,J)}\colon R \,\sB(I,J)\to\Hom(R I,R J)$ via
  $(i,j)\, \rho(I,J) = \std_i^* \std_j = \std_{ij}$.
  The $I\times J$ matrix associated with $\rho(I,J)$ (see~\S\ref{ss:mreps})
  is the generic matrix $[ X_{(i,j)} ]_{i\in I,j\in J}$.
  Hence, $\rho(I,J)$ is isotopic to the identity on $\Mat_{\card I\times \card J}(R)$.
  For $(i,j)\in I\times J$, we have
  $\XX( (i,j), \rho(I,J)) = X_i\std_j$
  whence
  $\Orbit(\rho(I,J)) = 
  R[X_\emptyset] J = R J$,
  an $R[X_I]$-module annihilated by each $X_i$ ($i\in I$).
\end{ex}

For a set $A$, we let $\binom A k$ be the set of $k$-element subsets of $A$.

\begin{ex}[Generic alternating matrices]
  \label{ex:graph}
  We define a basic family
  \[
    \bGamma = \bigl(\gamma(I,J)\bigr)_{I,J\in\Powf(\NN)}
  \]
  of module representations over $R$ as follows.
  First, let
  \[
    \sE(I,J) := \left\{
      A\in \binom {I\cup J} 2 : A \cap I \not= \emptyset \not= A \cap J
    \right\}
    \overset ! =
    \Bigl\{
      \{i,j\} : i\in I, j\in J, i\not= j
    \Bigr\}.
  \]
  Define ${\gamma(I,J)}\colon R\, \sE(I,J) \to\Hom(R I,R J)$ as follows.
  For $\{u,v\} \in \sE(I,J)$ with $u < v$,
  \[
    \{u,v\}\,\gamma(I,J) :=
    \begin{cases}
      \std_{uv} - \std_{vu}, & \text{if } u,v\in I\cap J,\\
      +\std_{uv}, & \text{if } u\in I \text{ and } v\in J, \text{ but } v\not\in I \text{ or }u\not\in J,\\
      -\std_{vu},& \text{if } u\in J \text{ and } v\in I, \text{ but } u\not\in I \text{ or } v\not\in J.
    \end{cases}
  \]
  For $\{ u,v\}\in \sE(I,J)$ with $u < v$,
  the matrix associated with $\gamma(I,J)$ has an entry $X_{\{u,v\}}$ in
  position $(u,v)$ if $(u,v)\in I\times J$ and an entry $-X_{\{u,v\}}$ in
  position $(v,u)$ if $(v,u)\in I\times J$; all other entries vanish.
  In particular, $\gamma(I,I)$ is isotopic to the inclusion $\Alt_{\card I}(R) \incl
  \Mat_{\card I}(R)$.
  
  Next, for $i\in I$ and $j\in J$ with $i\not= j$,
  \[
    \XX(\{i,j\}, \gamma(I,J)) =
    \begin{cases}
      \pm X_i \std_j \mp X_j\std_i, & \text{if } i,j\in I \cap J,\\
      \pm X_i \std_j, & \text{if }
      i \not\in J \text{ or } j\not\in I.
    \end{cases}
  \]
  For instance,
  \[
    \Orbit(\gamma(I,I)) = \frac{R[X_I] I}{\bigl\langle X_i\std_j - X_j\std_i : i,j\in I
      \text{ with } i < j\bigr\rangle}
  \]
  is the (negative) adjacency module of the complete graph with vertex set $I$ in
  the sense of \cite[\S 3.3]{cico}; cf.\ \S\ref{ss:graphs}.
  On the other hand, if $I\cap J = \emptyset$, then
  $\Orbit(\gamma(I,J)) = \Orbit(\rho(I,J))$.
\end{ex}

\begin{ex}[Generic symmetric matrices]
  \label{ex:Sigma}
  We define a basic family
  \[
    \bSigma = \bigl(\sigma(I,J)\bigr)_{I,J\in\Powf(\NN)}
  \]
  of module representations over $R$ as follows.
  First, define
  \[
    \sS(I,J) := \left\{
      A\in \binom {I\cup J} 1 \cup \binom {I\cup J} 2 : A \cap I \not= \emptyset \not= A \cap J
    \right\}
    =
    \Bigl\{
    \{i,j\} : i\in I, j\in J
    \Bigr\}.
  \]
  Define ${\sigma(I,J)}\colon R \,\sS(I,J) \to\Hom(R I,R J)$ as follows.
  For $i\in I$ and $j\in J$, let
  \[
    \{i, j\}\,\sigma(I,J) :=
    \begin{cases}
      \std_{ii} , & \text{if }
      i = j,\\ 
      \std_{ij} + \std_{ji}, & \text{if } i,j\in I\cap J \text{ and } i\not= j,\\
      \std_{ij}, & \text{if }
      i \not\in J \text{ or } j \not\in I.
    \end{cases}
  \]
  Thus, the matrix associated with $\sigma(I,J)$ has an entry $X_{\{i,j\}}$ in
  position $(i,j)$.
  In particular, $\sigma(I,I)$ is isotopic to the inclusion $\Sym_{\card I}(R)
  \incl \Mat_{\card I}(R)$.
  Next,
  \[
    \XX(\{i,j\}, \sigma(I,J)) =
    \begin{cases}
      X_i \std_i, & \text{if } i = j,\\
      X_i \std_j + X_j\std_i, & \text{if } i,j\in I \cap J \text{ and }i\not= j,\\
      X_i \std_j, & \text{if }
      i \not\in J \text{ or } j\not\in I.
    \end{cases}
  \]
  Similar to Example~\ref{ex:graph},
  $\Orbit(\sigma(I,I))$ is the (positive) adjacency module associated with the
  graph $\Bigl(I, \binom I 1 \cup \binom I 2\Bigr)$ as in \cite[\S
  3.3]{cico}.
  For $I\cap J = \emptyset$, $\Orbit(\sigma(I,J)) = \Orbit(\rho(I,J))$.
\end{ex}

\subsection{The constant rank theorem}
\label{ss:crk}

As before, we regard $R = R[X_\emptyset]$ as an $R[X_I]$-module annihilated by
each $X_i$.
In this section, we devise a sufficient criterion for recognising
when $\Orbit(\theta(I,J))$ is  ``akin'' to $R[X_I] G \oplus R(J\setminus
G)$ for some $G\subset J$ in a suitable sense involving Fitting ideals
(see \S\ref{ss:Fit}).

\begin{defn}
  \label{d:surjfam}
  Let $\bTheta$ be as in Definition~\ref{d:cohfam}.
  We say that $\bTheta$ is \emph{surjective} if each
  transition map $\varphi^{\tilde I,\tilde J}_{I,J}$ is surjective.
\end{defn}

Let $\bTheta$ be a coherent family of module representations as in
Definition~\ref{d:cohfam}.
For $I,J\in \Powf(\NN)$, write $\Omega(I,J) := \Orbit(\theta(I,J))$.
Further let $\Omega^\times(I,J) := \Omega(I,J)\otimes_{R[X_I]} R[X_I^{\pm 1}]$.

\begin{defn}
  \label{d:IJconst}
  Let $\bTheta$ be surjective, $I,J\in \Powf(\NN)$, and $\ell \ge 0$.
  Define $\Omega(I,J)$ and $\Omega^\times(I,J)$ as above.
  We say that $\bTheta$ is \emph{$(I,J)$-constant of rank $\ell$} if 
    \begin{enumerate}[label=(\alph{*})]
  \item $\Fit_i(\Omega(I,J)) = \langle 0\rangle$ for $i=0,\dotsc,\ell-1$ and
  \item $\Fit_\ell(\Omega^\times(H,J)) = \langle 1\rangle$
    for all non-empty $H\subset I$.
  \end{enumerate}
\end{defn}

Our terminology is motivated by Theorem~\ref{thm:crk} below
and Remarks~\ref{rem:crkmot}--\ref{rem:vbundles}.

\begin{ex}
  \label{ex:Rho_crk}
  Let $I,J\in \Powf(\NN)$.
  Then $\bRho$ from Example~\ref{ex:rec} is $(I,J)$-constant of rank~$0$.
  Indeed, for $\emptyset\not= H\subset I$,
  we have $\Omega(H,J) \approx R[X_\emptyset] J = RJ$ whence $\Omega^\times(H,J) = 0$.
\end{ex}

\begin{ex}
  \label{ex:Sigma_crk}
  Let $I,J\in\Powf(\NN)$.
  Then $\bSigma$ from Example~\ref{ex:Sigma} is $(I,J)$-constant of rank~$0$.
  To see this, let $\emptyset \not= H\subset I$ and $h\in H$.
  If $h\not\in J$, then
  $\XX(\{h,j\},\sigma(H,J)) = X_h\std_j$ for all $j\in J$ whence
  $X_h \Omega(H,J) = 0$ and thus
  $\Omega^\times(H,J) = 0$.
  On the other hand,
  if $h\in J$, then 
  \[
    \XX(\{h,j\},\sigma(H,J)) =
    \begin{cases}
      X_h \std_j + X_j \std_h, & \text{if } j \in H,\\
      X_h \std_j, & \text{if }j\not\in H
    \end{cases}
  \]
  for $j\in J\setminus\{h\}$.
  As $\XX(\{h\},\sigma(H,J)) = X_h\std_h$, we conclude that
  $\Omega^\times(H,J) = 0$.
\end{ex}

The case $\bTheta = \bGamma$ is more interesting.
First, 
if $I \cap J = \emptyset$, then $\bGamma$ is $(I,J)$-constant of rank
$0$ for the same reason as $\bRho$.
(Recall that $\gamma(I,J)$ and $\rho(I,J)$ have identical orbit modules when
$I\cap J = \emptyset$.)
On the other hand, we will see in Corollary~\ref{cor:Gamma_crk}
that $\bGamma$ is $(I,I)$-constant of rank $1$;
this is essentially an algebraic version of \cite[Prop.\ 5.11]{ask}. 

\begin{thm}[Constant rank theorem]
  \label{thm:crk}
  Let $R$ be a ring.
  Let $\bTheta$ be a surjective coherent family of module representations over $R$.
  Let $I,J\in \Powf(\NN)$ and $\ell \ge 0$.
  Suppose that~$\bTheta$ is $(I,J)$-constant of rank $\ell$.
  Let $\fO$ be an $R$-algebra which is a \DVR{}
  with maximal ideal~$\fP$.
  Then
  $\Omega(I,J)\otimes_{R[X_I]} \fO_x \approx \fO^\ell$
  for all $x\in \fO I\setminus \fP I$.
\end{thm}

Note that the conclusion of Theorem~\ref{thm:crk} is vacuously true when $I = \emptyset$.
We will prove Theorem~\ref{thm:crk} in \S\ref{ss:relating_orbits}.
Using Corollary~\ref{cor:crkzeta}, we obtain the following consequence.

\begin{cor}
  \label{cor:crk_zeta}
  Let the notation and assumptions be as in Theorem~\ref{thm:crk}.
  Suppose that $\fO$ is compact.
  Let $d = \card{I}$, $e = \card{J}$,
  and $q = \card{\fO/\fP}$.
  Then
  $\Zeta^{\ak}_{\theta(I,J)^\fO}(T) =
    \frac{1-q^{\ell-e}T}{(1-T)(1-q^{\ell+d-e}T)}$. \qed
\end{cor}

\begin{rem}
  \label{rem:crkmot}
  The conclusion of Theorem~\ref{thm:crk} clearly implies that
  $\Omega(I,J) \otimes \RF_x\approx \RF^\ell$ for each field $\RF$ and
  non-zero $x\in \RF I$.
  Moreover, one can show that the conclusion of Theorem~\ref{thm:crk} is equivalent to
  ``O-maximality'' of $\theta(I,J)^\fO$ in the sense of \cite[\S 5.1]{ask};
  cf.~\cite[Lemma 5.6]{ask} and \cite[Prop.\ 3.8]{ask2}.
  The connection between O-maximality and the ``constant rank spaces'' extensively
  studied in the literature is explained in \cite[\S 5.3]{ask}. 
\end{rem}

\begin{rem}
  \label{rem:vbundles}
  Although we shall not pursue this point of view here,
  Theorem~\ref{thm:crk} admits a geometric interpretation which we now briefly
  sketch.
  Suppose that the assumptions of Theorem~\ref{thm:crk} are satisfied.
  For the sake of simplicity, further suppose that $R = \CC$.
  Then $\Omega(I,J)$ defines a vector bundle (= locally free sheaf of modules)
  of rank $\ell$ on
  the projective space, $P$ say, of lines in $\CC I$.
  Indeed, let $\cF$ be the coherent sheaf on $P$ associated with the
  $\CC[X_I]$-module $\Omega(I,J)$.
  The conclusions of Theorem~\ref{thm:crk} show that the
  module of sections of $\cF$ over each affine chart $x_i \not= 0$ ($i\in I$) is
  projective;
  cf.~\cite[Exercise~20.13]{Eis95} or \stacks{00NV}.
  The problem of constructing
  vector bundles on
  projective spaces has a long and rich history; see \cite{OSS11} and
  references therein.
  The study of such vector bundles has also long been known to be related to
  the construction of spaces of matrices satisfying rank conditions;
  see e.g.\ \cite{EH88}.
\end{rem}

\subsection{Reminder: pushouts of modules}
\label{ss:pushouts}

We collect some basic facts on pushouts of modules.
Given module homomorphisms $A\xto{\beta_i} B_i$ and $A_i\xto{\alpha_i} B$ for
$i = 1,2$, we obtain module homomorphisms
\[\begin{aligned}
  A \,\xto{\small\begin{bmatrix}\beta_1& \beta_2\end{bmatrix}}\,
  &
    B_1\oplus B_2,
     &a \,\mapsto\,& (a\beta_1, a\beta_2)
    \text{ and}\\
  A_1\oplus A_2 \,\xto{\phantom{X}\small\begin{bmatrix}\alpha_1\\\alpha_2\end{bmatrix}\phantom{X}}\,
                                     & B,
  &(a_1,a_2)\,\mapsto\,& a_1\alpha_1 + a_2 \alpha_2.
\end{aligned}\]

\begin{prop}[{Cf.\ \stacks{08N2}}]
  \label{prop:pushout_char}
  A commutative square of modules 
  \begin{equation}
    \label{eq:module_po}
    \begin{gathered}
      \begin{xy}
        \xymatrix{
          A \ar[r]^\phi \ar[d]_\psi & B \ar[d]^{\psi'} \\
          A' \ar[r]_{\phi'} & B'
        }
      \end{xy}
    \end{gathered}
  \end{equation}
  is a pushout if and only if the following sequence is exact:
  \[
    \xymatrix@C+2.5em{
    A \ar[r]^{\small\begin{bmatrix}\psi& -\phi\end{bmatrix}\phantom{XY}} & A'\oplus B \ar[r]^{\small\begin{bmatrix}\phi'\\\psi'\end{bmatrix}}
    & B' \ar[r] & 0.
    }
  \]
\end{prop}

If $\psi$ is an epimorphism, then so is $\psi'$.
In the following, we will also freely use the fact that extension of scalars
(being left adjoint to restriction of scalars, see \stacks{05DQ}) preserves
pushouts and epimorphisms of modules. 

\begin{cor}
  \label{cor:pushout_kernel}
  If \eqref{eq:module_po} is a pushout,
  then $\Ker(\psi') = \Ker(\psi)\phi$. \qed
\end{cor}

We will use the following simple observation in our proof of Corollary~\ref{cor:asym}.
\begin{lemma}
  \label{lem:Fit1}
  Let $\pi\colon R^n \onto M$ be a finite presentation of an $R$-module $M$.
  Let $\ell\le n$ and 
  let $\rho\colon R^n\onto R^{n-\ell}$ be the projection onto the first $n-\ell$
  coordinates.
  Form the pushout
  \[
    \begin{xy}
      \xymatrix{
        R^n \ar@{->>}[r]^\pi \ar@{->>}[d]_{\rho} & M \ar@{->>}[d]^{\rho'}\\
        R^{n-\ell} \ar@{->>}[r]_{\pi'} & M'.
      }
    \end{xy}
  \]
  If $M' = 0$, then $\Fit_\ell(M) = \langle 1\rangle$.
\end{lemma}
\begin{proof}
  We may assume that $M = \Coker(a)$ for
  $a\in \Mat_{m\times n}(R)$.
  Let $a'\in \Mat_{m\times(n-\ell)}(R)$ be obtained from $a$ by deleting the final
  $\ell$ columns.
  Then $M' \approx \Coker(a')$.
  As $\Fit_0(M') = \langle 1\rangle$,
  the $(n-\ell)\times (n-\ell)$ minors of $a'$ generate the unit ideal of $R$.
  Hence, $\Fit_\ell(M) = \langle 1\rangle$.
\end{proof}

\subsection{Relating orbit modules and a proof of Theorem~\ref{thm:crk}}
\label{ss:relating_orbits}

In this section, let $\bTheta$ be a fixed \itemph{surjective} coherent family of module
representations over~$R$ as in Definition~\ref{d:cohfam}.
As one of the main ingredients of our proof of Theorem~\ref{thm:crk},
we now relate the orbit modules $\Omega(I,J) = \Orbit(\theta(I,J))$ as $I$
and $J$ vary.
Let $\omega(I,J)$ denote the projection $R[X_I]J \onto \Omega(I,J)$.
For $I\subset \tilde I\in \Powf(\NN)$ and $J\subset \tilde J\in \Powf(\NN)$,
let ${\gamma^{\tilde I,\tilde J}_{I,J}}\colon R[X_{\tilde I}\tilde J] \to R[X_I]J$
be the diagonal of  the commutative diagram
\[
  \small
   \begin{xy}
     \xymatrix@C+1.5em@R+0.5em{
       R[X_{\tilde I}] \tilde J \ar@{->>}[r]^{\bigoplus\limits_{\tilde J} \ret} \ar@{->>}[d]_{\ret} \ar@{->>}[dr] &  R[X_I] \tilde J\ar@{->>}[d]^{\ret}\\
       R[X_{\tilde I}] J \ar@{->>}[r]_{\bigoplus\limits_J\ret} & R[X_I] J.
     }
   \end{xy}
 \]
In particular,
\[
  \left(X_i\std_j\right) \gamma^{\tilde I,\tilde J}_{I,J} = \begin{cases}
    X_i \std_j, & \text{if }i\in I\text{ and }j\in J,\\
    0, & \text{otherwise.}
  \end{cases}
\]

\begin{prop}
  \label{prop:pushout}
  Let $I \subset \tilde I \in \Powf(\NN)$ and $J \subset \tilde J \in \Powf(\NN)$.
    There exists a (unique) $R[X_{\tilde I}]$-module epimorphism
    ${\pi^{\tilde I,\tilde J}_{I,J}}\colon \Omega(\tilde I,\tilde J) \onto \Omega(I,J)$
    such that the diagram
    \begin{equation}
      \small
      \label{eq:Orbit_pushout}
      \begin{gathered}
        \begin{xy}
          \xymatrix@C+2.5em@R+1em{
            R[X_{\tilde I}]\tilde J
            \ar@{->>}[r]^{\omega(\tilde I,\tilde J)}
            \ar@{->>}[d]_{\gamma^{\tilde I,\tilde J}_{I,J}}
            & \Omega(\tilde I,\tilde J)
            \ar[d]^{\pi^{\tilde I,\tilde J}_{I,J}}\\
            R[X_I] J
            \ar@{->>}[r]_{\omega(I,J)}
            & \Omega(I,J) \\
          }
        \end{xy}
      \end{gathered}
    \end{equation}
    commutes.
    Moreover, \eqref{eq:Orbit_pushout} is a pushout of $R[X_{\tilde I}]$-modules.
\end{prop}
\begin{proof}
  A simple calculation shows that
  for $\tilde m \in M(\tilde I,\tilde J)$,
  we have
  $\XX\!\Bigl(\tilde m, \theta(\tilde I,\tilde J)\Bigr)
  \, \gamma^{\tilde I,\tilde J}_{I,J}
  = \XX\!\Bigl(\tilde m \,\varphi^{\tilde I,\tilde J}_{I,J},\,\theta(I,J)\Bigr)$.
  Hence, $\gamma^{\tilde I,\tilde J}_{I,J}$ maps $\orbit(\theta(\tilde
  I,\tilde J))$ onto $\orbit(\theta(I,J))$ and the first claim follows.
  The second claim follows since for any commutative diagram of modules
  \[
    \begin{xy}
      \xymatrix{
        A \ar[r]  \ar[d]^\alpha & B \ar[r]^\pi \ar[d]^\beta & C \ar[r] \ar[d]^\gamma & 0 \\
        A' \ar[r] & B' \ar[r]_{\pi'} & C' \ar[r] & 0, \\
      }
    \end{xy}
  \]
  if $\alpha$ is an epimorphism and the rows are exact,
  then the right square (with top left corner $B$) is a pushout---indeed, this
  e.g.\ follows from Proposition~\ref{prop:pushout_char} by diagram chasing. 
\end{proof}

\begin{rem}
  \label{rem:pi_product}
  Note that for $I'\subset I \subset \tilde I \in \Powf(\NN)$
  and $J'\subset J \subset \tilde J \in \Powf(\NN)$,
  we clearly have $\pi^{\tilde I,\tilde J}_{I',J'} = \pi^{\tilde I,\tilde
    J}_{I,J}\, \pi^{I,J}_{I',J'}$.
\end{rem}

It is well-known that if $M$ is an $R$-module and $\fa \normal R$, then
$M\otimes_R R/\fa \approx M/\fa M$ (naturally).
Together with Proposition~\ref{prop:pushout}, this simple fact
and the identification
$R[X_{\tilde I}] = R[X_I]/\bigl\langle X_{\tilde I\setminus I}\bigr\rangle$ now imply the
following.

\begin{cor}
  \label{cor:shrink_I}
  ${\pi^{\tilde I,J}_{I,J}} \colon \Omega(\tilde I,J) \onto\Omega(I,J)$ induces
  an $R[X_{\tilde I}]$-module isomorphism
  \[
    \Omega(\tilde I,J)\otimes_{R[X_{\tilde I}]} R[X_I] \approx \Omega(I,J).
    \pushQED{\qed}
    \qedhere
    \popQED
  \]
\end{cor}

As in \S\ref{ss:crk}, let $\Omega^\times(I,J) = \Omega(I,J)\otimes_{R[X_I]}
R[X_I^{\pm 1}]$.
Recall that extension of scalars preserves pushouts.
For $J\subset \tilde J$, the map $\pi^{I,\tilde J}_{I,J}$
induces a map $\Omega^\times(I,\tilde J)\onto \Omega^\times(I,J)$
which we also denote by $\pi^{I,\tilde J}_{I,J}$.

\begin{defn}
  \label{d:orbit_times}
  Let $\orbit^\times\!(\theta(I,J)) \subset R[X_I^{\pm 1}] J$ be the image of
  $\orbit(\theta(I,J))\otimes R[X_I^{\pm 1}]$.
\end{defn}
\begin{cor}
  \label{cor:colred}
  $\pi^{I,\tilde J}_{I,J}\colon \Omega^\times(I,\tilde J) \onto \Omega^\times(I,J)$
  is an isomorphism if and only if
  $\std_j \in \orbit^\times(\theta(I,\tilde J))$
  for all $j\in \tilde J\setminus J$.
\end{cor}
\begin{proof}
  Combine Corollary~\ref{cor:pushout_kernel}, extension of scalars $R[X_I]\to
  R[X_I^{\pm 1}]$, and Proposition~\ref{prop:pushout}.
\end{proof}

\begin{cor}
  \label{cor:shiftfit}
  Let $J' \subset J$.
  If $\Omega^\times(I,J') = 0$,
  then
  $\Fit_{\card{J\setminus J'}}\bigl(\Omega^\times(I,J)\bigr) = \langle 1\rangle$.
\end{cor}
\begin{proof}
  Combine Lemma~\ref{lem:Fit1} and  Proposition~\ref{prop:pushout}.
\end{proof}

\begin{proof}[Proof of Theorem~\ref{thm:crk}]
  Let $x\in \fO I \setminus \fP I$.
  Let $\RF$ be the residue field of $\fO$.
  Let $\bar x\in \RF I$ be the image of $x$ and $H := \{ i\in I : \bar
  x_i\not= 0\}$.
  Let $\bar x(H) := \sum\limits_{h\in H} \bar x_h \std_h \in \RF H$ be the image of
  $\bar x$ under $\RF I \xonto{\ret}\RF H$.
  Let $M_x := \Omega(I,J)\otimes_{R[X_I]} \fO_x$.
  For $i < \ell$, $\Fit_i(M_x) = \langle 0\rangle$.
  Next,
  \begin{align*}
    M_x \otimes_{\fO} \RF
    & \approx 
      \Omega(I,J) \otimes_{R[X_I]} \RF_{\bar x} \\
    & \approx \left(\Omega(I,J) \otimes_{R[X_I]} R[X_H^{\pm 1}]\right)\otimes_{R[X_H^{\pm 1}]} \RF_{\bar x(H)} \\
    &\underset{(\dagger)} \approx
      \Omega^\times(H,J)\otimes_{R[X_H^{\pm 1}]} \RF_{\bar x(H)}
  \end{align*}
  where $(\dagger)$ is due to Corollary~\ref{cor:shrink_I}.
  Hence, $\Fit_\ell(M_x)$ maps onto the unit ideal
  of~$\RF$ so that in fact $\Fit_i(M_x) = \langle 1\rangle$.
  Thus, $M_x \approx \fO^\ell$
  by Proposition~\ref{prop:Fit_iso} and Example~\ref{ex:Fit_free}.
\end{proof}

The following application of Proposition~\ref{prop:pushout} will become
important in \S\ref{s:board}.
Recall the definition of $\bGamma = (\gamma(I,J))_{I,J\in\Powf(\NN)}$ from Example~\ref{ex:graph}.

\begin{lemma}
  \label{lem:perp}
  Let $I\subset J \in \Powf(\NN)$ with $J \not= \emptyset$.
  Then $\Fit_0(\Orbit(\gamma(I,J))) = \langle 0\rangle$.
\end{lemma}
\begin{proof}
  It follows from Remark~\ref{rem:orbit}\ref{rem:orbit3} and
  Proposition~\ref{prop:fitbasic}\ref{prop:fitbasic3}
  that it suffices to prove the claim for $R = \ZZ$.
  Suppose that $I\not= \emptyset$.
  As Proposition~\ref{prop:pushout} provides
  an epimorphism $\Orbit(\gamma(I,J)) \onto \Orbit(\gamma(I,I))$,
  by Proposition~\ref{prop:fitbasic}\ref{prop:fitbasic2},
  it thus suffices to show
  that $\Fit_0(\Orbit(\gamma(I,I)) = 0$.
  Indeed, each $w \in \orbit(\gamma(I,I)) = \langle X_i \std_j - X_j \std_i :
  i,j\in I, i < j\rangle$ (cf.~Example~\ref{ex:graph})
  satisfies the non-trivial linear relation $\sum_{i\in I} X_i w_i = 0$.
  We conclude that
  $\Orbit(\gamma(I,I))\otimes \QQ(X_I) \not= 0$ whence
  $\Fit_0(\Orbit(\gamma(I,I))) = 0$
  (e.g.~by Example~\ref{ex:Fit_free}).
  Finally, if $I = \emptyset$, then $\Orbit(\gamma(I,J)) = RJ\not= 0$ and the
  claim follows from Example~\ref{ex:Fit_free}.
\end{proof}

\section{Linear relations with disjoint supports}
\label{s:rel}

In this section, we develop an abstract setting for studying ask zeta
functions associated with the modules
$\Board_{d\times e}(\cA,u,R)$ (see \S\ref{ss:intro_relations}) using the machinery
from \S\ref{s:coh}.
This will, in particular, allow us to prove Corollary~\ref{cor:rec}.

To simplify our exposition, we henceforth assume that $\UU$ is a
``sufficiently large'' infinite set.
For our purposes, it will be enough to assume that each of $\NN\times \NN$, $\binom
\NN 1$, and $\binom \NN 2$ is a subset of $\UU$.
For a ring $R$, let $\acute R := R[\acute u_x^{\pm 1} : x \in \UU]$,
where the $\acute u_x$ are algebraically independent over $R$.
We also let $\acute{\dtimes}$ denote extension of scalars $R \to \acute R$.
We further assume that $\Colours$ is an infinite set of \emph{colours} and that
the \emph{blank} symbol $\blank$ does not belong to $\Colours$.

\subsection{Relation modules}
\label{ss:relmod}

By a \emph{partial colouring} of a subset $U\subset \UU$, we mean a function
$\beta\colon U \to \Coloursb$ such that the $\beta$-fibre
$\{c\}\beta^{-}$ of each
element $c\in \Colours$ is finite;
we tacitly extend $\beta$ to a partial colouring of all of $\UU$ by setting
$x\beta := \blank$ for $x\in \UU\setminus U$.
Given $\beta$, we say that $x\in \UU$ is \emph{$\beta$-blank} if $x\beta = \blank$
and \emph{$\beta$-coloured} (with colour $x\beta$) otherwise.
When the reference to $\beta$ is clear, we also simply talk about $x$ being
blank or coloured, respectively.
For a set $B\subset \UU$, let $\beta[B] := \{ c\in \Colours :
\{c\}\beta^-\subset B\}$,
the set of colours confined entirely within $B$.

\begin{defn}
  \label{d:relmod}
  The \emph{$\beta$-relation module} associated with a set $B\subset \UU$
  (over $R$) is 
  the $\acute R$-module
  \[
    \Board(B \sslash\beta;R) := \left\{ x\in \acute RB :
      \forall c\in \beta[B]. \sum_{b\in \{c\}\beta^-} \acute u_b x_b = 0\right\}\le \acute R B.
  \]
\end{defn}

When the reference to $R$ is clear, we simply write
$\Board(B \sslash \beta) := \Board(B\sslash\beta;R)$.
Relation modules are well-behaved with respect to inclusions.

\begin{prop}
  \label{prop:betaquo}
  Let $B\subset\tilde B \subset \UU$.
  Then the retraction map $\ret\colon \acute R \, \tilde B  \onto \acute R \, B$
  (see \S\ref{ss:mreps})
  maps $\Board(\tilde B \sslash \beta)$ onto $\Board(B \sslash \beta)$.
\end{prop}
\begin{proof}
  For $A\subset \UU$ and $c\in \Colours$, let
  $A_c := \{c\}\beta^-$ if $c\in \beta[A]$ and $A_c := \emptyset$ otherwise.
  Let $A_\blank := A \setminus \bigcup_{c\in \beta[A]} A_c$;
  equivalently,
  $A_\blank = \{ a\in A : a\beta = \blank \text{ or } a\beta =: c \in
  \Colours \text{ but } \{c\}\beta^- \not\subset A\}$.
    Note that
  $\Board(A_c \sslash \beta) = \bigl\{
  x\in \acute R A_c : \sum_{a\in A_c} \acute u_a x_a = 0
  \bigr\}$
  and
  $\Board(A \sslash \beta) = \acute R \, A_\blank \,\oplus\, \bigoplus_{c\in
    \beta[A]} \Board(A_c\sslash \beta)$.
  Clearly, $\beta[B]\subset\beta[\tilde B]$.
  If $c\in \beta[B]$, then $\tilde B_c = B_c$ and $\ret\colon \acute R \tilde
  B\onto\acute R B$ maps $\Board(\tilde B_c\sslash\beta)$
  (isomorphically) onto $\Board(B_c\sslash \beta)$.
  On the other hand, if $c\in \beta[\tilde B]\setminus \beta[B]$,
  then $B_c = \emptyset$,
  $\tilde B_c\not\subset B$, and
  $\ret$ maps $\Board(\tilde B_c \sslash\beta)$ onto $\acute R (\tilde B_c
  \cap B) \subset \acute R B_\blank$.
  Indeed, fix $w \in \tilde B_c\setminus B$.
  Then for $x \in \Board(\tilde B_c \sslash \beta) \subset \acute R \tilde B_c$,
  the coordinates $x_b\in \acute R$ with $b\not= w$ can take arbitrary values
  as we can solve for $x_w$ in $\acute R$.
  Finally, 
  $B_\blank = \left(
      \tilde B_\blank \cup\, \bigcup_{c\in \beta[\tilde B]\setminus \beta[B]} \tilde B_c
    \right) \cap B$
  whence the claim follows by applying $\ret$ to the decomposition
  \[
    \Board(\tilde B\sslash\beta) = \acute R \tilde B_\blank \,\,\oplus
    \bigoplus\limits_{c\in \beta[\tilde B]\setminus \beta[B]}\Board(\tilde
    B_c\sslash\beta) \,\oplus\, \bigoplus\limits_{c\in \beta[B]}\Board(\tilde B_c\sslash \beta).
    \qedhere
  \]
\end{proof}

Relation modules capture linear relations with disjoint support and
unit coefficients.

\begin{prop}
  \label{prop:ranger}
  Let $B\subset \UU$.
  \begin{enumerate}[label=(\roman{*})]
  \item
    \label{prop:ranger1}
    $\Board(B \sslash \beta)$ is a free $\acute R$-module.
    If $B$ is finite, then $\Board(B \sslash\beta)$ has rank $\card B - m$,
    where $m = \#\{ c\in \beta[B] : \{c\}\beta^-\not=\emptyset\}$.
  \item
    \label{prop:ranger2}
    Let $S$ be an $\acute R$-algebra.
    For $x\in \UU$, let $u_x\in S^\times$ denote the image of $\acute u_x\in
    \acute R^\times$.
    Then the natural map $\Board(B \sslash \beta)\otimes_{\acute R} S \to S
    B$ (induced by $\Board(B \sslash \beta)\incl \acute R B\to S B$)
    is injective with image 
    $\bigl\{
    x \in S B: \forall c\in \beta[B]. \sum\limits_{b \in \{c\}\beta^-} \!\!u_bx_b = 0 
    \bigr\}$.
    \end{enumerate}
\end{prop}
\begin{proof}
  This is analogous to Lemma~\ref{lem:ranger} with $\acute R$ in place of $R$.
\end{proof}

\subsection{Restricting module representations to relation modules}
\label{ss:relmodfam}

In line with notation from above, for a module representation $\theta$ over
$R$, we write $\acute \theta := \theta^{\acute R}$.

\paragraph{Individual module representations: defining $\theta\sslash\beta$.}
\begin{defn}
  \label{d:sslash}
  Let $\theta\colon RB \to \Hom(RI, RJ)$ be a module representation, where $B\subset
  \UU$ and $I,J\in\Powf(\NN)$.
  Let $\beta\colon B \to \Coloursb$ be a partial colouring.
  Define $\theta\sslash \beta$ to be the composite
  $\Board(B \sslash\beta) \incl \acute R B \xto{\acute\theta} \Hom(\acute R I,
  \acute R J)$,
  i.e.\ the restriction of $\acute\theta$ to $\Board(B \sslash\beta)$.
\end{defn}

In order to apply Theorem~\ref{thm:crk} later on, we will require the following
simple observation.

\begin{lemma}
  \label{lem:Fit0_beta}
  If $\Fit_i(\Orbit(\theta)) = 0$, then $\Fit_i(\Orbit(\theta\sslash\beta)) = 0$.
\end{lemma}
\begin{proof}
  By construction, $\Orbit(\acute\theta)$ is a quotient of
  $\Orbit(\theta\sslash\beta)$.
  Now apply Remark~\ref{rem:orbit}\ref{rem:orbit3}
  and Proposition~\ref{prop:fitbasic}\ref{prop:fitbasic2}--\ref{prop:fitbasic3}.
\end{proof}

We may now interpret the modules $\Board_{d\times e}(\cA,u,R)$ from
\S\ref{ss:intro_relations} in the present setting.

\begin{ex}
  \label{ex:abstract_intro_relations}
  Let $I = [d]$, $J = [e]$, and let $\beta\colon [d]\times [e]\to \Coloursb$ be
  a partial colouring.
  Let $\cA = \bigl(\{c\}\beta^-\bigr)_{c\in \Colours}$ be the set of $\beta$-fibres of the
  elements of $\Colours$.
  It is clear that $\beta$ and~$\cA$ determine one another.
  Note that $\cA$ is a ``partial colouring of $[d]\times [e]$'' in the sense
  of~\S\ref{ss:intro_relations}.
  Let $S$ be an $R$-algebra.
  Let $u$ be a $d\times e$ matrix whose entries are units of $S$.
  Let $S(u)$ denote $S$ regarded as an $\acute R$-algebra via the
  ring map $\acute R \to S$ which sends $\acute u_{(i,j)}$ to
  $u_{ij}$ for $(i,j)\in [d]\times [e]$ and which sends all
  other $\acute u_x$ to $1$.
  Recall the definition of $\rho(I,J)$ from \S\ref{ex:rec}.
  Then Proposition~\ref{prop:ranger} allows us to identify
  $(\rho(I,J)\sslash\beta)^{S(u)}$ and $\Board_{d\times
    e}(\cA,u,S)\incl \Mat_{d\times e}(S)$.
\end{ex}

\paragraph{Families of module representations: defining $\bTheta\sslash\beta$.}
Let $\bTheta$ be a basic family of module representations as in
Definition~\ref{d:basfam}.
We assume that
$\sB_\infty := \bigcup\limits_{I,J\in \Powf(\NN)}\sB(I,J) \subset \UU$.
Let $\acute{\bTheta} := \bigl(\acute\theta(I,J)\bigr)_{I,J\in\Powf(\NN)}$
be obtained from $\bTheta$ by extension of scalars along $R \to \acute R$.

\begin{defn}
  \label{d:betaIJ}
Let $\beta\colon \sB_\infty\to \Coloursb$ be a partial colouring.
Define $\beta(I,J)\colon \sB(I,J) \to \Coloursb$ to be the partial
colouring of $\sB(I,J)$ given by
\[
  x\beta(I,J) =
  \begin{cases}
    c, & \text{if $c := x\beta\in \beta[\sB(I,J)]$},
    \\
    \blank, & \text{otherwise.}
  \end{cases}
\]
\end{defn}

That is, $x\beta(I,J) \not= \blank$ if and only if $c := x\beta\not= \blank$ and
$\{c\}\beta^-\subset \sB(I,J)$, in which case $x\beta(I,J) = c$.
Of course, the definition of $\beta(I,J)$ depends on $\bTheta$.
(We extend $\beta(I,J)$ to all of $\UU$ as in \S\ref{ss:relmod}.)
Note that if $I'\subset I$ and $J'\subset J$, then $(\beta(I,J))(I',J') = \beta(I',J')$.
Further note that, by construction, $\theta(I,J)\sslash \beta = \theta(I,J)
\sslash \beta(I,J)$.
By Proposition~\ref{prop:betaquo}, we obtain a surjective coherent family
of module representations
\[
  \bTheta \sslash \beta := \Bigl(\theta(I,J) \sslash \beta\Bigr)_{I,J\in\Powf(\NN)} = \Bigl(\theta(I,J)\sslash \beta(I,J)\Bigr)_{I,J\in\Powf(\NN)}
\]
with transition homomorphisms induced by retractions
$\acute R \, \sB(\tilde I,\tilde J)\to \acute R \, \sB(I,J)$.

In Example~\ref{ex:abstract_intro_relations}, we rephrased the setting of
Corollary~\ref{cor:rec} in terms of the ``$\dtimes \sslash \beta$'' operation defined
above.
In order to deduce results such as Theorem~\ref{thm:embedded}, we will combine the machinery from
\S\ref{s:orbital} and \S\ref{s:coh}.
A first step in this direction is the following.

\begin{lemma}
  \label{lem:orbital_same_crk}
  Suppose that $\bTheta$ and $\bTheta\sslash\beta$ are
  both $(I,J)$-constant of the same rank $\ell$ (see Definition~\ref{d:IJconst}).
  Then $\theta(I,J)\sslash\beta(I,J)$ is
  an orbital subrepresentation of $\acute \theta(I,J)$.
\end{lemma}
\begin{proof}
  Proposition~\ref{prop:fitbasic}\ref{prop:fitbasic3} and
  Remark~\ref{rem:orbit}\ref{rem:orbit3}  
  imply that $\acute\bTheta$ is $(I,J)$-constant of rank $\ell$.
  Now combine Theorem~\ref{thm:crk} and Corollary~\ref{cor:orbital_iso_epi}.
\end{proof}

\subsection{Closure, admissibility, and a proof of Corollary~\ref{cor:rec}}
\label{ss:proof_rec}

A proof of Corollary~\ref{cor:rec} is now within easy reach.
Let $\beta\colon \NN\times \NN \to \Coloursb$ be a partial colouring;
recall that we assume that $\NN\times\NN \subset \UU$.
For $I,J\in\Powf(\NN)$, let $\beta(I,J)\colon I\times J\to \Coloursb$ as in
Definition~\ref{d:betaIJ}.

\begin{defn}
  \label{d:colclosed}
  Let $I,J\in \Powf(\NN)$,
  $I'\subset I$, and $J'\subset J$.
  We say that $(I',J')$ is \emph{$\beta$-closed} in $(I,J)$ if the
  following condition is satisfied: for
  each $c\in \Colours$, whenever $(i',j')\in I'\times J'$ has
  $\beta(I,J)$-colour~$c$, then the $\beta(I,J)$-fibre of $c$ is contained
  entirely within $I'\times J'$.
\end{defn}

In other words, 
$(I',J')$ is $\beta$-closed in $(I,J)$ if and only if for each
$\beta(I,J)$-coloured element of $I'\times J'$, all elements of the same
colour in $I\times J$ belong to $I'\times J'$. Equivalently:

\begin{lemma}
  \label{lem:closed_res}
  $(I',J')$ is $\beta$-closed in $(I,J)$ if and only if
  $\beta(I',J')\colon I'\times J'\to \Coloursb$ is
  the set-theoretic restriction of $\beta(I,J)\colon I\times J\to \Coloursb$.
  \qed
\end{lemma}

\begin{defn}
  \label{d:Rho_adm}
  Let $I,J\in \Powf(\NN)$.
  We say that $\beta$ is \emph{$(I,J)$-admissible}
  if the following condition is satisfied:
  for all non-empty $I',J'\in \Powf(\NN)$ such that $(I',J')$ is
  $\beta$-closed in~$(I,J)$, the set $I'\times J'$ contains a
  $\beta(I,J)$-blank (and hence also $\beta(I',J')$-blank) element.
\end{defn}

When $I = [d]$ and $J = [e]$,
the preceding concept of $(I,J)$-admissibility agrees with admissibility 
as defined in~\S\ref{ss:intro_relations}.
For instance, Figure~\ref{fig:PW3} is an example of a $([3],[3])$-admissible
partial colouring.

\begin{lemma}
  \label{lem:ex_blank}
  Let $I,J\in \Powf(\NN)$ such that $\beta$ is $(I,J)$-admissible.
  Let $I'\subset I$ and $J'\subset J$ such that $I'\times J'\not= \emptyset$.
  Then $I'\times J'$ contains a $\beta(I',J')$-blank element.
\end{lemma}
\begin{proof}
  By definition, if $(I',J')$ is $\beta$-closed in $(I,J)$, then there exists $(i',j')\in
  I'\times J'$ with $(i',j')\beta(I,J) = (i',j') \beta(I',J') = \blank$.
  Otherwise, there exists $c\in \Colours$ and $(i',j')\in I'\times J'$
  with $(i',j')\beta(I,J) = c$ such that
  $I\times J$ contains an element with $\beta(I,J)$-colour $c$ outside of
  $I'\times J'$. In that case, $(i',j') \beta(I',J') = \blank$ by the
  definition of $\beta(I',J')$.
\end{proof}

\begin{cor}
  \label{cor:adm_subsets}
  Let $I'\subset I\in \Powf(\NN)$ and $J'\subset J\in\Powf(\NN)$.
  If $\beta$ is $(I,J)$-admissible, then~$\beta$ is $(I',J')$-admissible.
\end{cor}
\begin{proof}
  Let $I''\subset I'$  and $J''\subset J'$ both be non-empty and suppose that
  $(I'',J'')$ is $\beta$-closed in $(I',J')$.
  By Lemma~\ref{lem:ex_blank}, there exists $z \in I''\times J''$
  with $z\beta(I'',J'') = \blank$.
  By Lemma~\ref{lem:closed_res}, since $(I'',J'')$ is $\beta$-closed in
  $(I',J')$, we have $\blank = z \beta(I'',J'') = z\beta(I',J')$.
\end{proof}

Let $\bRho$ be the basic family of module representations from
Example~\ref{ex:rec}.
Define
$\bTheta = \bigl(\theta(I,J)\bigr)_{I,J\in\Powf(\NN)} :=
\bRho\sslash \beta = 
\bigl(\rho(I,J)\sslash \beta(I,J)\bigr)_{I,J\in\Powf(\NN)}$;
recall from \S\ref{ss:relmodfam} that $\rho(I,J) \sslash\beta(I,J)$ is the
restriction of $\acute\rho(I,J) := \rho(I,J)^{\acute R}$ to
$\Board((I\times J)\sslash \beta(I,J))$. 
Given~$\bTheta$, define $\Omega(I,J)$ and $\Omega^\times(I,J)$ as in
\S\ref{ss:crk}.
The following innocuous vanishing result is the key ingredient of our proof of
Corollary~\ref{cor:rec}.
\begin{lemma}
  \label{lem:Omega0}
  Let $I,J\in\Powf(\NN)$ and
  suppose that $\beta$ is $(I,J)$-admissible.
  Then
  \[
    \Omega^\times(I',J') = 0
  \]
  for all non-empty $I'\subset I$ and all $J'\subset J$.
\end{lemma}
\begin{proof}
  If $J' = \emptyset$, then $\Omega^\times(I',J')$ is a quotient of $\acute
  R[X_{I'}^{\pm 1}] \emptyset = 0$.
  Suppose that $I' \not=\emptyset \not= J'$.
  By Lemma~\ref{lem:ex_blank}, there exists $(i',j') \in I'\times J'$ with
  $(i',j') \beta(I',J') = \blank$.
  Hence, $\std_{(i',j')} \in \Board((I'\times J') \sslash\beta(I',J'))$
  and thus $X_{i'} \std_{j'} \in \orbit^\times(\rho(I',J')
  \sslash\beta(I',J'))$
  (see Example~\ref{ex:rec} and Definition~\ref{d:orbit_times}).
  By Corollary~\ref{cor:colred}, $\Omega^\times(I',J') \approx
  \Omega^\times(I',J'\setminus\{j'\})$ whence the claim follows by induction on $\card{J'}$.
\end{proof}

\begin{cor}
  \label{cor:Rho_blueprint}
  Let the assumptions be as in Lemma~\ref{lem:Omega0}.
  Then:
  \begin{enumerate}[label=(\roman{*})]
  \item
    \label{cor:Rho_blueprint1}
    $\bRho \sslash\beta$ is $(I,J)$-constant of rank $0$ (see
    Definition~\ref{d:IJconst}).
  \item
    \label{cor:Rho_blueprint2}
    $\rho(I,J) \sslash \beta(I,J)$ is an orbital subrepresentation of $\acute\rho(I,J)$.
  \end{enumerate}
\end{cor}
\begin{proof}
  Part \ref{cor:Rho_blueprint1} follows immediately from Lemma~\ref{lem:Omega0}.
  For \ref{cor:Rho_blueprint2}, combine \ref{cor:Rho_blueprint1},
  Example~\ref{ex:Rho_crk}, and Lemma~\ref{lem:orbital_same_crk}.
\end{proof}

\begin{cor}
  \label{cor:Rho_zeta}
  Let $I,J\in\Powf(\NN)$.
  Write $d = \card I$ and $e = \card J$.
  Let $\beta$ be an $(I,J)$-admissible partial colouring $\NN\times \NN\to
  \Coloursb$.
  Define $\bRho$ as in Example~\ref{ex:rec} for $R = \ZZ$.
  Let $\fO$ be a $\acute \ZZ$-algebra which is a compact \DVR{}.
  Then
  $\Zeta^{\ak}_{(\rho(I,J)\sslash\beta(I,J))^\fO}(T) = 
  \frac{1-q^{-e}T}{(1-T)(1-q^{d-e}T)}$.
\end{cor}
\begin{proof}
  Combine Corollary~\ref{cor:Rho_blueprint} and Corollary~\ref{cor:crk_zeta}.
\end{proof}

Corollary~\ref{cor:rec} follows by combining the preceding corollary and
Example~\ref{ex:abstract_intro_relations} (with $R = \ZZ$ and $S = \fO$).

\section{Board games}
\label{s:board}

In this section, we further develop the ideas from \S\ref{ss:proof_rec}
in order to prove Corollary~\ref{cor:asym}--\ref{cor:sym}
and Theorem~\ref{thm:embedded}.
Our narrative will revolve around moves applied to the cells of
grids associated with suitable families of module representations. 

\subsection{Combinatorial families of module representations and isolated cells}
\label{ss:combfam}

We seek to generalise the recursive strategy that we used in our
proof of Corollary~\ref{cor:rec} (see Lemma~\ref{lem:Omega0})
in~\S\ref{ss:proof_rec}.
For the moment, let the notation be as in \S\ref{ss:proof_rec}.
In particular, $\beta\colon \NN\times \NN\to\Coloursb$ is a partial colouring
which gives rise to a partial colouring $\beta(I,J)\colon I\times J\to \Coloursb$
for all $I,J\in \Powf(\NN)$.
Apart from the machinery developed in previous sections,
our proof of Corollary~\ref{cor:rec} relied on the following ingredients:
\begin{itemize}
\item
  \textit{Deleting columns:}
  any $\beta(I,J)$-blank element $(i,j)\in I\times J$ yields
  an isomorphism \[
    \Omega^\times(I,J)\approx\Omega^\times(I,J\setminus\{j\}).
  \]
\item
  \textit{Admissibility:}
  combinatorial assumptions  ensure the existence of
  enough blanks to eventually delete all columns via repeated applications of
  the preceding step.
\end{itemize}

We will now begin to generalise both of these ingredients.
Henceforth, let $\bTheta$ be a basic family of module representations as in
Definition~\ref{d:basfam}.
Write $\sB_\infty = \bigcup_{I,J\in\Powf(\NN)}\sB(I,J)$.
We begin by  describing situations in which we may ``delete columns'' (via 
isomorphisms as above) using suitable elements of $\sB(I,J)$.

\paragraph{Grids and cells.}
Let $I,J\in \Powf(\NN)$.
Let $b\in \sB(I,J)$.
Let $\bigl[b^{IJ}_{ij}\bigr]_{i\in I,j\in J}$ be the matrix of
the map ${b\theta(I,J)}\colon RI \to RJ$ with
respect to the defining bases
and let $G_b(I,J) := \{ (i,j)\in I\times J : b^{IJ}_{ij} \not= 0\}$
be its support.
By abuse of notation, in the following, we often write $ij$ instead of $(i,j)$
for an element of $I\times J$.
The \emph{grid} of $(I,J)$ (w.r.t.\ $\bTheta$)
is $\cG(I,J) := \bigcup_{b\in \sB(I,J)}G_b(I,J)\subset I\times J$.
The elements of $\cG(I,J)$ are its \emph{cells}.
Given a cell $ij \in \cG(I,J)$, we refer to $i$ as its \emph{row} and to $j$
as its \emph{column}.

\begin{defn}
  \label{d:combfam}
  We say that $\bTheta$ is a \emph{combinatorial family of module
    representations} if the following conditions are satisfied in addition to
  those in Definition~\ref{d:basfam}:
  \begin{enumerate}
  \item
    The sets $G_b(I,J)$ for $b\in \sB(I,J)$ are pairwise disjoint and
    non-empty.
  \item
    For each $b\in \sB(I,J)$,
    each non-zero coefficient $b^{IJ}_{ij}$ ($i\in I$, $j\in J$) is a unit of
    $R$.
  \end{enumerate}
\end{defn}

If $\bTheta$ is combinatorial,
then we call the sets $G_b(I,J)$ the \emph{cell classes} of $\cG(I,J)$.

\begin{ex}
  \label{ex:basic3}
  Each of the basic families $\bRho$, $\bGamma$, and $\bSigma$ of module
  representations from~\S\ref{ss:basfam_exs} is combinatorial.
  \begin{enumerate}
  \item
    For $\bTheta = \bRho$, we have $\cG(I,J) = I\times J$.
    Cell classes are singletons.
  \item
    For $\bTheta = \bGamma$, we have $\cG(I,J) = I\odot J := \{ ij\in I\times J : i\not= j\}$. 
    Let $ij\in \cG(I,J)$.
    If $ji\in \cG(I,J)$, then $\{ ij,ji\}$ is a cell class; otherwise,
    $\{ij\}$ is a cell class.
  \item
    For $\bTheta = \bSigma$, we have $\cG(I,J) = I\times J$.
    Let $ij\in \cG(I,J)$.
    If
    $ji\in \cG(I,J)$, then $\{ ij,ji\}$ is a cell class (which might be a singleton); otherwise,
    $\{ij\}$ is a cell class.
  \end{enumerate}
\end{ex}

Henceforth, suppose that $\bTheta$ is combinatorial.

\begin{lemma}
  \label{lem:grid_facts}
  Let $I\subset \tilde I\in \Powf(\NN)$ and $J\subset\tilde J\in \Powf(\NN)$.
  \begin{enumerate}[label=(\roman{*})]
  \item \label{lem:grid_facts1}
    Let $\tilde b\in \sB(\tilde I,\tilde J)$.
    Then $\tilde b\in \sB(I,J)$ if and only if $G_{\tilde b}(\tilde I,\tilde J) \cap (I\times
    J)\not= \emptyset$.
  \item \label{lem:grid_facts2}
    $\cG(I,J) = \cG(\tilde I,\tilde J)\cap (I\times J)$.
  \end{enumerate}
\end{lemma}
\begin{proof}
  By the definition of a basic family of module representations,
  the following diagram commutes:
  \[
    \small
    \xymatrix@C+2em{
      R \,\sB(\tilde I,\tilde J) \ar[r]^{\theta(\tilde I,\tilde J)}\ar[d]^{\ret} &\Hom(R \tilde I,R \tilde J) \ar[d]^{\Hom(\inc,\ret)} \\
      R \,\sB(I,J) \ar[r]_{\theta(I,J)} & \Hom(R I,R J).
    }
  \]
  Hence, if $\tilde b\in \sB(I,J)$, then $G_{\tilde b}(\tilde I,\tilde J) \cap
  (I\times  J) \overset!= G_{\tilde b}(I,J) \not= \emptyset$.
  If, on the other hand, $\tilde b\in \sB(\tilde I,\tilde J)\setminus
  \sB(I,J)$, then $G_{\tilde b}(\tilde
  I,\tilde J) \cap (I\times J) = \emptyset$.
  Both parts follow immediately.
\end{proof}

\paragraph{Partial colourings of grids.}
Define a surjection ${\varepsilon(I,J)}\colon \cG(I,J) \onto \sB(I,J)$ by sending
$ij\in \cG(I,J)$ to the unique $b\in \sB(I,J)$ with $ij\in G_b(I,J)$.
Let $\beta\colon \sB_\infty\to \Coloursb$ be a partial colouring.
Define $\beta(I,J)\colon \sB(I,J) \to \Coloursb$ as in \S\ref{ss:relmodfam}.
The composite $\varepsilon(I,J) \beta(I,J)$ is then a partial colouring of
$\cG(I,J)$ which, by abuse of notation, we again simply denote by $\beta(I,J)$.
Conversely, every partial colouring of $\cG(I,J)$ that is constant on cell
classes induces a partial colouring of $\sB(I,J)$.
The effects of deleting rows or columns on partial colourings of grids are easily described.

\begin{lemma}
  \label{lem:gridcol}
  Let $I\subset \tilde I\in \Powf(\NN)$ and $J\subset\tilde J\in \Powf(\NN)$.
  Let $ij\in \cG(I,J)$.
  Then
  \[
    ij \, \beta(I,J) = \begin{cases}
      c, & \text{if $c := ij\,\beta(\tilde I,\tilde J)\in\Colours$ and every cell
        class $C$ of $\cG(\tilde I,\tilde J)$ with}\\&\text{$\beta(\tilde
        I,\tilde J)$-colour $c$ satisfies $C\cap (I\times J) \not= \emptyset$},\\
      \blank, & \text{otherwise}.
    \end{cases}
  \]
\end{lemma}
\begin{proof}
  Let $b := ij \, \varepsilon(I,J)\in \sB(I,J)$.
  As $\beta(I,J) = (\beta(\tilde I,\tilde J))(I,J)$,
  \[
    b\beta(I,J) = \begin{cases}
      c, & \text{if $c := b\beta(\tilde I,\tilde J)\in \Colours$ and $\{c\}\beta(\tilde
        I,\tilde J)^-\subset \sB(I,J)$,}\\
      \blank, & \text{otherwise}
    \end{cases}
  \]
  whence the claim follows from Lemma~\ref{lem:grid_facts}\ref{lem:grid_facts1}.
\end{proof}

\paragraph{Isolated cells.}
We say that $ij \in \cG(I,J)$ is \emph{$(I,J)$-isolated}
(or simply \emph{isolated} if $I$ and $J$ are clear from the context)
if $ij$ is the sole member of its cell class within $\cG(I,J)$.

\begin{ex}
  \label{ex:isolated3}
  We can easily describe isolated cells related to each
  of the basic families $\bRho$, $\bGamma$, and $\bSigma$ of module
  representations from~\S\ref{ss:basfam_exs}; cf.\ Example~\ref{ex:basic3}.
  \begin{itemize}
  \item
    For $\bTheta = \bRho$, each cell of $\cG(I,J) = I\times J$ is isolated.
  \item
    For $\bTheta = \bGamma$, a cell $ij \in \cG(I,J)$ is isolated if and only
    if $ji\not\in \cG(I,J)$.
  \item
    For $\bTheta = \bSigma$, a cell $ij \in \cG(I,J)$ is isolated if and only
    if $i = j$ or $ji \not\in \cG(I,J)$.
  \end{itemize}
\end{ex}

We now anticipate the deletion of columns in the following
subsections as follows.
Let $ij\in \cG(I,J)$ be an isolated cell.
Let $b := ij \, \varepsilon(I,J)$ be the corresponding element of~$\sB(I,J)$.
Then $b \theta(I,J) = u\std_{ij}$ for some $u\in R^\times$ and
therefore $\XX(b,\theta(I,J)) = u X_i\std_j$.
As in \S\ref{ss:proof_rec},
we thus obtain an isomorphism
$\Omega^\times(I,J) \approx \Omega^\times(I,J\setminus\{j\})$.

\subsection{Admissible partial colourings of grids}
\label{ss:adm_grid}

In this subsection,
we will derive a generalisation (Corollary~\ref{cor:adm_blueprint}) of Corollary~\ref{cor:Rho_blueprint}
for combinatorial families of module representations.
The crucial new ingredient is a suitably general notion of ``admissible'' colourings.

\paragraph{Setup.}
Let $\bTheta$ be a combinatorial family of module representations
as in \S\ref{ss:combfam}.
Write $\sB_\infty = \bigcup_{I,J\in\Powf(\NN)} \sB(I,J)$.
Let $\beta\colon \sB_\infty\to \Coloursb$ be a partial colouring.
For $I,J\in\Powf(\NN)$, define $\beta(I,J)\colon \sB(I,J) \to \Coloursb$ as in
\S\ref{ss:relmodfam}.
As in \S\ref{ss:combfam}, we also let $\beta(I,J)$ denote the induced partial
colouring of the grid $\cG(I,J)$.
Using \S\ref{ss:relmodfam},
we obtain a surjective coherent family
$\bTheta \sslash \beta = \bigl(
\theta(I,J)\sslash\beta(I,J)\bigr)_{I,J\in\Powf(\NN)}$
of module representations over $\acute R := R[\acute u_x^{\pm 1} : x\in \UU]$.
Let $\acute\Omega(I,J) := \Orbit( (\bTheta\sslash\beta)(I,J))$ and
$\acute\Omega^\times(I,J) := \acute\Omega(I,J) \otimes_{\acute R[X_I]} \acute R[X_I^{\pm 1}]$.
For $I \subset \tilde I \in \Powf(\NN)$ and $J \subset \tilde J \in
\Powf(\NN)$,
define ${\acute\pi^{\tilde I,\tilde J}_{I,J}}\colon \acute\Omega(\tilde
I,\tilde J) \onto \acute\Omega(I,J)$ via Proposition~\ref{prop:pushout}.

\vspace*{0.5em}

We seek to generalise the notion of $(I,J)$-admissibility from
Definition~\ref{d:Rho_adm} to more general combinatorial families of module
representations.
Our first step is to formalise the deletion of columns as outlined at the end
of \S\ref{ss:combfam}.

\begin{defn}[Moves]
  \label{d:reduce}
  Define a binary relation $\xto[\bTheta,\beta]{} \,\subset \Powf(\NN)^2$ (``{move}'') by letting
  \[
    (I,J) \xto[\bTheta,\beta]{} (I,J')
  \]
  if and only if there exists an $(I,J)$-isolated $\beta(I,J)$-blank cell
  $ij\in \cG(I,J)$ such that $J' = J\setminus\{j\}$.
  We let $\xto[\bTheta,\beta]*$ be the reflexive transitive closure of $\xto[\bTheta,\beta]{}$.
\end{defn}

Let $I,J,J_1,J_2\in \Powf(\NN)$.
Clearly, if $(I,J) \xto[\bTheta,\beta]* (I,J_1)$ and $ij\in \cG(I,J_1)
\subset \cG(I,J)$, then whenever $ij$ is $(I,J)$-isolated (resp.\
$\beta(I,J)$-blank), it is also $(I,J_1)$-isolated (resp.~$\beta(I,J_1)$-blank).
Therefore, if $(I,J) \xto[\bTheta,\beta]* (I,J_1)$ and $(I,J) \xto[\bTheta,\beta]* (I,J_2)$, then
$(I,J) \xto[\bTheta,\beta]* (I,J_1\cap J_2)$.
Hence, if $(I,J)\xto[\bTheta,\beta]* (I,\emptyset)$,
then we can construct a sequence $(I,J) = (I,J^{(0)}) \xto[\bTheta,\beta]{} (I,J') \xto[\bTheta,\beta]{} \dotsb
\xto[\bTheta,\beta]{} (I,J^{(k)}) = (I,\emptyset)$ by picking an arbitrary blank
isolated cell $i_uj_u \in \cG(I,J^{(u)})$ and setting $J^{(u+1)} :=
J^{(u)}\setminus\{j_u\}$ for $0\le u < k$.

The following lemma forms the heart of our recursive arguments.

\begin{lemma}
  \label{lem:rediso}
  If $(I,J) \xto[\bTheta,\beta]* (I,J')$,
  then $\acute\Omega^\times(I,J)\approx
  \acute\Omega^\times(I,J')$ via $\acute\pi^{I,J}_{I,J'}$.
\end{lemma}
\begin{proof}
  By Remark~\ref{rem:pi_product},
  we may assume that $(I,J) \xto[\bTheta,\beta]{} (I,J')$ so that $J' =
  J\setminus\{j\}$, where  $ij\in \cG(I,J)$ is isolated and blank.
  Let $b = ij \, \varepsilon(I,J) \in \sB(I,J)$.
  Since $ij$ is blank,
  $\std_b \in \Board(\sB(I,J)\sslash\beta(I,J))$.
  As $ij$ is isolated, $\std_b \,\theta(I,J) = u\std_{ij}$ for $u \in
  R^\times$.
  Thus, $\XX\!\bigl(ij\, \varepsilon(I,J), \theta(I,J)\bigr) = u X_i \std_j$.
  Hence, $\std_j \in \orbit^\times\!((\bTheta\sslash\beta)(I,J))$ and
  the claim follows from Corollary~\ref{cor:colred}.
\end{proof}

\begin{defn}
  \label{d:adm}
  We say that $\beta$ is \emph{$\bTheta(I,J)$-admissible} of level $\ell \ge 0$
  if the following condition is satisfied:
  for all $\emptyset\not= H\subset I$, there exists $D(H) \subset J$ with $\card{D(H)}\le
  \ell$ such that $(H,J\setminus D(H)) \xto[\bTheta,\beta]* (H,\emptyset)$;
  when $\ell = 0$, we simply say that $\beta$ is \emph{$\bTheta(I,J)$-admissible}.
\end{defn}

This notion gives rise to the ``board game'' in the title of the present section:
$\beta$ is $\bTheta(I,J)$-admissible of level $\ell$ if and only if for each
non-empty set of rows $H\subset I$, it is possible to find a set $D(H) \subset J$ of at most 
$\ell$ columns such that some sequence of moves applied to the partially
coloured grid $\cG(H,J\setminus D(H))$ eventually deletes all of its columns.

\begin{ex}
  \label{ex:allblank_adm}
  Let $I,J\in\Powf(\NN)$.
  Then the ``all blank'' partial colouring $\blank$
  is both $\bRho(I,J)$-admissible and $\bSigma(I,J)$-admissible.
  The situation for $\bGamma$ is more complicated.
  If $I\cap J = \emptyset$, then $\blank$ is $\bGamma(I,J)$-admissible.
  On the other hand,
   if $I\not= \emptyset$, then no partial colouring is
  $\bGamma(I,I)$-admissible: the corresponding grid does not contain
  \textit{any} isolated cells.
\end{ex}

\begin{lemma}
  \label{lem:adm_beta}
  Let $\bTheta$ be a combinatorial family of module representations
  as above.
  Let~$\beta\colon\sB_\infty \to \Coloursb$ be a partial colouring.
  Let $I,J\in \Powf(\NN)$ and suppose that~$\beta$ is $\bTheta(I,J)$-admissible of
  level $\ell \ge 0$.
  Then $\Fit_\ell(\acute\Omega^\times(H,J)) = \langle 1\rangle$ for each non-empty
  $H\subset I$.
\end{lemma}
\begin{proof}
  Let $D(H)\subset J$ with $\card{D(H)}\le \ell$ and
  $(H,J\setminus D(H))\xto[\beta]* (H,\emptyset)$.
  By Lemma~\ref{lem:rediso}, $\acute\Omega^\times(H,J\setminus D(H)) = 0$.
  By Corollary~\ref{cor:shiftfit}, 
  $\Fit_\ell(\acute\Omega^\times(H,J))= \langle 1\rangle$.
\end{proof}

\begin{cor}
  \label{cor:adm_blueprint}
  Let $I,J\in\Powf(\NN)$ and $\ell \ge 0$.
  Suppose that
  \begin{enumerate*}[label=(\alph{*})]
  \item
    $\bTheta$ is $(I,J)$-constant of rank $\ell$
    (see Definition~\ref{d:IJconst}) and
  \item $\beta\colon \sB_\infty\to\Coloursb$ is $\bTheta(I,J)$-admissible of level $\ell$.
  \end{enumerate*}
  Then:
  \begin{enumerate}[label=(\roman{*})]
  \item
    \label{cor:adm_blueprint1}
    $\bTheta\sslash\beta$ is $(I,J)$-constant of rank $\ell$.
  \item
    \label{cor:adm_blueprint2}
    $\theta(I,J) \sslash \beta(I,J)$ is an orbital subrepresentation of
    $\theta(I,J)^{\acute R}$. 
  \end{enumerate}
\end{cor}
\begin{proof}
  For $i < \ell$,
  as $\Fit_i(\Orbit(\theta(I,J))) = 0$,
  by Lemma~\ref{lem:Fit0_beta},
  $\Fit_i(\Orbit(\theta(I,J)\sslash\beta) = 0$.
  Part~\ref{cor:adm_blueprint1} thus follows from
  Lemma~\ref{lem:adm_beta}.
  For \ref{cor:adm_blueprint2}, combine \ref{cor:adm_blueprint1}
  and Lemma~\ref{lem:orbital_same_crk}.
\end{proof}

The ask zeta functions associated with
$(\bTheta\sslash\beta)(I,J)$ are thus given by Corollary~\ref{cor:crk_zeta}.

\subsection{Rectangular board games}
\label{ss:rec_board}

For $\bTheta = \bRho$, Definition~\ref{d:adm} is consistent with our
previous notion of $(I,J)$-admissibility.

\begin{prop}
  \label{prop:two_adm}
  Let $\beta\colon \NN \times \NN \to \Coloursb$ be a partial colouring.
  Let $I,J\in\Powf(\NN)$.
  Then $\beta$ is $(I,J)$-admissible (in the sense of
  Definition~\ref{d:Rho_adm}) if and only if $\beta$ is
  $\bRho(I,J)$-admissible
  (in the sense of Definition~\ref{d:adm}).
\end{prop}
\begin{proof}
  Suppose that $\beta$ is $(I,J)$-admissible.
  We show that $\beta$ is $\bRho(I,J)$-admissible by induction on $\card J$.
  We may clearly assume that $J\not= \emptyset$.
  Let $\emptyset \not= H\subset I$.
  By Lemma~\ref{lem:ex_blank}, there exists $hj\in H\times J$ with $hj\,
  \beta(H,J) = \blank$.
  Hence, $(H,J) \xto[\bRho,\beta]{} (H,J\setminus\{j\})$.
  By Corollary~\ref{cor:adm_subsets}, $\beta$ is $(I,J\setminus\{j\})$-admissible
  whence $(H,J\setminus\{j\}) \xto[\bRho,\beta]{*} (H,\emptyset)$ by induction.

  Conversely,
  suppose that $\beta$ is $\bRho(I,J)$-admissible
  but that $\beta$ is not $(I,J)$-admissible.
  Then $I\not= \emptyset \not= J$ and there exist non-empty $H \subset I$ and
  $\bar J\subset J$ such that (a) $(H,\bar J)$ is $\beta$-closed in $(I,J)$ but (b)
  $H\times \bar J$ does not contain any
  $\beta(I,J)$-blank elements.
  By Lemma~\ref{lem:closed_res}, $H\times\bar J$ then does not contain
  any $\beta(H,\bar J)$-blank elements either.
  As $\beta$ is $\bRho(I,J)$-admissible, $(H,J) \xto[\bRho,\beta]*
  (H,\emptyset)$.
  Choose a sequence
  \[
    (H,J) = (H,J^{(0)}) \xto[\bRho,\beta]{} (H,J^{(1)}) \xto[\bRho,\beta]{} \dotsb
    \xto[\bRho,\beta]{} (H,J^{(k)}) = (H,\emptyset)
  \]
  in which $J^{(u+1)} = J^{(u)}\setminus\{h_uj_u\}$ for a blank
  and isolated cell $h_uj_u \in H \times J^{(u)}$.
  Let $u$ be minimal with $j_u\in \bar J$.
  Hence, $\bar J \subset J^{(u)}$.
  Moreover, $h_uj_u$ is $\beta(H,J^{(u)})$-blank and hence $\beta(H,\bar
  J)$-blank.
  This contradicts the fact that
  no cell in $H\times \bar J$ is $\beta(H,\bar J)$-blank.
\end{proof}

Thanks to Proposition~\ref{prop:two_adm}, we see that
Corollary~\ref{cor:adm_blueprint}
generalises Corollary~\ref{cor:Rho_blueprint}.

\paragraph{Transpose colourings.}
Given a partial colouring $\beta\colon \NN\times \NN\to \Coloursb$,
define its \emph{transpose} to be $\beta^\top\colon \NN\times \NN \to \Coloursb$ via $(i,j)\beta^\top =
(j,i)\beta$.

\begin{lemma}
  \label{lem:transpose_adm}
  $\beta$ is $\bRho(I,J)$-admissible if and only if $\beta^\top$ is
  $\bRho(J,I)$-admissible.
\end{lemma}
\begin{proof}
  Clearly, $\beta$ is $(I,J)$-admissible in the sense of
  Definition~\ref{d:Rho_adm} if and only if $\beta^\top$ is $(J,I)$-admissible.
  The claim thus follows from Proposition~\ref{prop:two_adm}.
\end{proof}

\subsection{Symmetric board games}

The following is a symmetric counterpart of Corollary~\ref{cor:Rho_zeta}.

\begin{cor}
  \label{cor:Sigma_zeta}
  Let $I,J\in \Powf(\NN)$. 
  Write $d = \card I$ and $e = \card J$.
  Define $\bSigma$ as in Example~\ref{ex:Sigma} for $R = \ZZ$.
  Define $\sS_\infty = \bigcup_{I,J\in \Powf(\NN)}\sS(I,J)$.
  Let $\beta\colon \sS_\infty \to \Coloursb$ be a $\bSigma(I,J)$-admissible partial colouring.
  Let~$\fO$ be a $\acute \ZZ$-algebra which is a compact \DVR{}.
  Then
  \[
    \Zeta^\ak_{(\sigma(I,J)\sslash\beta(I,J))^\fO}(T) = \frac{1-q^{-e}T}{(1-T)(1-q^{d-e}T)}.
  \]
\end{cor}
\begin{proof}
  Combine Example~\ref{ex:Sigma_crk}, 
  Corollary~\ref{cor:adm_blueprint}, 
  and Corollary~\ref{cor:crk_zeta}.
\end{proof}

\clearpage

\begin{ex}
  \quad
  \begin{enumerate}
  \item 
  Let $I = J = [4]$ and let $\mathsf{blue} \in \Colours$.
  let $\beta$ be the partial colouring of $\sS([4],[4])$ with
  $\{ 1,2\} \beta = \{ 2,3\} \beta = \{3,4\}\beta = \mathsf{blue}$ and such
  that all other points of $\sS([4],[4])$ are blank.
  The induced colouring of the grid $\cG([4],[4]) = [4]\times [4]$ is
  \[
    \begin{tikzpicture}
      [box/.style={rectangle,draw=black,thick, minimum size=1cm},
      scale=0.5, every node/.style={scale=0.5}
      ]
      \foreach \x in {0,1,2,3}{
        \foreach \y in {0,1,2,3}
        \node[box,fill=white] at (\x,\y){};
      }

      \node[box,fill=BlueViolet] at (1,3){};
      \node[box,fill=BlueViolet] at (2,2){};  
      \node[box,fill=BlueViolet] at (3,1){};  

      \node[box,fill=BlueViolet] at (0,2){};  
      \node[box,fill=BlueViolet] at (1,1){};  
      \node[box,fill=BlueViolet] at (2,0){};

      \node at (-1,0) {4};
      \node at (-1,1) {3};
      \node at (-1,2) {2};
      \node at (-1,3) {1};

      \node at (0,4) {1};
      \node at (1,4) {2};
      \node at (2,4) {3};
      \node at (3,4) {4};
    \end{tikzpicture}.
  \]
  We claim that this partial colouring is
  $\bSigma([4],[4])$-admissible.
  (For a substantial generalisation, see Example~\ref{ex:snakes}.)
  For example, since the diagonal cells are blank and isolated, 
  $([4],[4]) \xto[\bSigma,\beta]* ([4],\emptyset)$.
  Next, consider the case $H = \{2,3\}$ in Definition~\ref{d:adm}.
  Note that every cell class of $[4] \times [4]$ intersects
  $\cG(\{2,3\},[4]) = \{2,3\} \times [4]$.
  Hence, the induced colouring on $\{2,3\} \times [4]$ is
  \[
    \begin{tikzpicture}
      [box/.style={rectangle,draw=black,thick, minimum size=1cm},
      scale=0.5, every node/.style={scale=0.5}
      ]
      \foreach \x in {0,1,2,3}{
        \foreach \y in {0,1}
        \node[box,fill=white] at (\x,\y){};
      }
      \node[box,fill=BlueViolet] at (0,1){};  
      \node[box,fill=BlueViolet] at (1,0){};  
      \node[box,fill=BlueViolet] at (2,1){};  
      \node[box,fill=BlueViolet] at (3,0){};

      \node at (-1,1) {2};
      \node at (-1,0) {3};

      \node at (0,2) {1};
      \node at (1,2) {2};
      \node at (2,2) {3};
      \node at (3,2) {4};
    \end{tikzpicture}.
  \]
  
  Note that $(3,4)$ is an isolated \itemph{blue} cell.
  Hence, deleting the $4$th
  column using the isolated \itemph{blank} cell $(2,4)$
  results in the all-blank partial
  colouring on $\cG(\{2,3\},\{1,2,3\})$.
  Thus, $(\{2,3\},[4]) \xto[\bSigma,\beta]* (\{2,3\},\emptyset)$.
  We leave it to the reader to verify that $(H,J)\xto[\bSigma,\beta]*
  (H,\emptyset)$ for the remaining cases of $\emptyset \not= H\subset I$.
  Let $\fO$ be a compact \DVR{} and
  let $M = \{ x\in \Sym_4(\fO) : x_{12} + x_{23} + x_{34} = 0\}$.
  Then $\bSigma([4],[4])$-admissibility of $\beta$ and
  Corollary~\ref{cor:Sigma_zeta} e.g.\ imply that $\Zeta^\ak_M(T) =
  \frac{1-q^{-4}T}{(1-T)^2}$.
\item
  The partial colouring $\beta$ of $\sS([3],[3])$ whose associated grid is
  \[
    \begin{tikzpicture}
      [box/.style={rectangle,draw=black,thick, minimum size=1cm},
      scale=0.5, every node/.style={scale=0.5}
      ]
      \foreach \x in {0,1,2}{
        \foreach \y in {0,1,2}
        \node[box,fill=white] at (\x,\y){};
      }

      \node[box,fill=BlueViolet] at (0,1){};  
      \node[box,fill=BlueViolet] at (1,0){};  
      \node[box,fill=BlueViolet] at (1,2){};  
      \node[box,fill=BlueViolet] at (2,1){};  

      \node at (-1,0) {3};
      \node at (-1,1) {2};
      \node at (-1,2) {1};

      \node at (0,3) {1};
      \node at (1,3) {2};
      \node at (2,3) {3};
    \end{tikzpicture}.
  \]
  is \itemph{not} $\bSigma([3],[3])$-admissible. (Consider the case $H =
  \{2\}$.)
  The conclusion of Corollary~\ref{cor:Sigma_zeta} does not hold either.
  Indeed, letting $N := \{ x\in \Sym_3(\fO) : x_{12} + x_{23} = 0\}$,
  using \textsf{Zeta}, we find that if $\fO$ has sufficiently large residue
  characteristic, then
  \[
    \Zeta^\ak_N(T) = 
    \frac{1 + q^{-1} T - 4 q^{-2} T + q^{-3} T + q^{-4} T^{2}}{(1 - q^{-1}T)(1 - T)^2}.
  \]
\end{enumerate}
\end{ex}

\subsection{Antisymmetric board games}

As we already indicated in Example~\ref{ex:allblank_adm}, 
even the ``all blank'' partial colouring may or may not be
$\bGamma(I,J)$-admissible, depending on $I$ and $J$.
It turns out that this subtlety disappears if we consider admissibility of
level $1$.
Recall that $I\odot J = \{ ij \in I\times J : i\not= j\}$.

\begin{lemma}
  \label{lem:allblank_Gamma}
  Let $I,J\in\Powf(\NN)$. 
  Then the partial colouring $\blank\colon I\odot J\to\Coloursb$ with constant value
  $\blank$ is $\bGamma(I,J)$-admissible of level $1$.
\end{lemma}
\begin{proof}
  Let $H\subset I$ be non-empty.
  If $H \cap J= \emptyset$, then $(H,J)\xto[\bGamma,\blank]* (H,\emptyset)$
  by Example~\ref{ex:allblank_adm} and
  since $\xto[\bGamma,\blank]*$ and $\xto[\bRho,\blank]*$ agree for pairs of
  disjoint sets.
  (Cf.\ Lemma~\ref{lem:sym_vs_rec} below.)
  Let $j\in H\cap J$.
  Then every cell in the $j$th row of $H\odot (J\setminus \{j\})$ is isolated
  (and blank) whence $(H,J\setminus \{j\}) \xto[\bGamma,\blank]* (H,\emptyset)$.
\end{proof}

We already alluded to the following right before Theorem~\ref{thm:crk} in \S\ref{ss:crk}.

\begin{cor}
  \label{cor:Gamma_crk}
  Let $I\subset J\in \Powf(\NN)$ with $J \not= \emptyset$.
  Then $\bGamma$ is $(I,J)$-constant of rank $1$.
\end{cor}
\begin{proof}
  Combine Lemmas~\ref{lem:perp}, \ref{lem:allblank_Gamma}, and \ref{lem:adm_beta}.
\end{proof}

The following is thus an antisymmetric counterpart of
Corollary~\ref{cor:Rho_zeta} and Corollary~\ref{cor:Sigma_zeta}.

\begin{cor}
  \label{cor:Gamma_zeta}
  Define $\sE_\infty = \bigcup_{I,J\in \Powf(\NN)}\sE(I,J)$.
  Let $\beta\colon \sE_\infty \to \Coloursb$ be a partial colouring.
  Let $I,J\in \Powf(\NN)$ with $I\subset J$. 
  Suppose that $\beta$ is $\bGamma(I,J)$-admissible of level $1$.
  Write $d = \card I$ and $e = \card J$.
  Let $\fO$ be a $\acute \ZZ$-algebra which is a compact \DVR{}.
  Then
  \[
    \Zeta^\ak_{(\gamma(I,J)\sslash\beta)^\fO}(T) = \frac{1-q^{1-e}T}{(1-T)(1-q^{d+1-e}T)}.
  \]
\end{cor}
\begin{proof}
  For $J\not= \emptyset$,
  combine Corollary~\ref{cor:Gamma_crk}, 
  Corollary~\ref{cor:adm_blueprint}, 
  and \ref{cor:crk_zeta}.
  The case $I = J = \emptyset$ is trivial.
\end{proof}

\begin{ex}[Rainbows]
  \label{ex:snakes}
  Fix a sequence $c_1,c_2,\dotsc\in \Colours$ of different colours.
  For $0\le b< d$, let $\chi_{b,d}$ be the partial colouring of
  $\sE([d],[d])$ (see Example~\ref{ex:graph}) such that
  $\{ i, i+a\} \chi_{b,d} = c_a$  for $a,i\ge 1$ with $a\le b$ and $i+a \le
  d$.
  Hence, $b$ is precisely the number of different colours used.
  See Figure~\ref{fig:rainbow} for an illustration of these partial colourings.

  \begin{figure}[h]
    \centering
    \begin{subfigure}[t]{0.4\textwidth}
      \centering
      \begin{tikzpicture}
        [box/.style={rectangle,draw=black,thick, minimum size=1cm},
        scale=0.3, every node/.style={scale=0.3}
        ]
        \foreach \x in {0,1,2,3,4,5,6}{
          \foreach \y in {0,1,2,3,4,5,6}
          \node[box,fill=white] at (\x,\y){};
        }

        \node[thick,cross out,draw] at (0,6){};
        \node[thick,cross out,draw] at (1,5){};
        \node[thick,cross out,draw] at (2,4){};
        \node[thick,cross out,draw] at (3,3){};
        \node[thick,cross out,draw] at (4,2){};
        \node[thick,cross out,draw] at (5,1){};
        \node[thick,cross out,draw] at (6,0){};

        \node[box,fill=Blue] at (0,5){};
        \node[box,fill=Blue] at (1,4){};
        \node[box,fill=Blue] at (2,3){};
        \node[box,fill=Blue] at (3,2){};
        \node[box,fill=Blue] at (4,1){};
        \node[box,fill=Blue] at (5,0){};
        \node[box,fill=Blue] at (1,6){};
        \node[box,fill=Blue] at (2,5){};
        \node[box,fill=Blue] at (3,4){};
        \node[box,fill=Blue] at (4,3){};
        \node[box,fill=Blue] at (5,2){};
        \node[box,fill=Blue] at (6,1){};

        \node[box,fill=Cerulean] at (0,4){};
        \node[box,fill=Cerulean] at (1,3){};
        \node[box,fill=Cerulean] at (2,2){};
        \node[box,fill=Cerulean] at (3,1){};
        \node[box,fill=Cerulean] at (4,0){};
        \node[box,fill=Cerulean] at (2,6){};
        \node[box,fill=Cerulean] at (3,5){};
        \node[box,fill=Cerulean] at (4,4){};
        \node[box,fill=Cerulean] at (5,3){};
        \node[box,fill=Cerulean] at (6,2){};

        \node[box,fill=LimeGreen] at (0,3){};
        \node[box,fill=LimeGreen] at (1,2){};
        \node[box,fill=LimeGreen] at (2,1){};
        \node[box,fill=LimeGreen] at (3,0){};
        \node[box,fill=LimeGreen] at (3,6){};
        \node[box,fill=LimeGreen] at (4,5){};
        \node[box,fill=LimeGreen] at (5,4){};
        \node[box,fill=LimeGreen] at (6,3){};

        \node[box,fill=Yellow] at (0,2){};
        \node[box,fill=Yellow] at (1,1){};
        \node[box,fill=Yellow] at (2,0){};
        \node[box,fill=Yellow] at (4,6){};
        \node[box,fill=Yellow] at (5,5){};
        \node[box,fill=Yellow] at (6,4){};
      \end{tikzpicture}.
      \caption{$\chi_{4,7}$ (admissible of level $1$)}
    \end{subfigure}
    \begin{subfigure}[t]{0.4\textwidth}
      \centering
      \begin{tikzpicture}
        [box/.style={rectangle,draw=black,thick, minimum size=1cm},
        scale=0.3, every node/.style={scale=0.3}
        ]
        \foreach \x in {0,1,2,3,4,5,6}{
          \foreach \y in {0,1,2,3,4,5,6}
          \node[box,fill=white] at (\x,\y){};
        }

        \node[thick,cross out,draw] at (0,6){};
        \node[thick,cross out,draw] at (1,5){};
        \node[thick,cross out,draw] at (2,4){};
        \node[thick,cross out,draw] at (3,3){};
        \node[thick,cross out,draw] at (4,2){};
        \node[thick,cross out,draw] at (5,1){};
        \node[thick,cross out,draw] at (6,0){};

        \node[box,fill=Blue] at (0,5){};
        \node[box,fill=Blue] at (1,4){};
        \node[box,fill=Blue] at (2,3){};
        \node[box,fill=Blue] at (3,2){};
        \node[box,fill=Blue] at (4,1){};
        \node[box,fill=Blue] at (5,0){};
        \node[box,fill=Blue] at (1,6){};
        \node[box,fill=Blue] at (2,5){};
        \node[box,fill=Blue] at (3,4){};
        \node[box,fill=Blue] at (4,3){};
        \node[box,fill=Blue] at (5,2){};
        \node[box,fill=Blue] at (6,1){};

        \node[box,fill=Cerulean] at (0,4){};
        \node[box,fill=Cerulean] at (1,3){};
        \node[box,fill=Cerulean] at (2,2){};
        \node[box,fill=Cerulean] at (3,1){};
        \node[box,fill=Cerulean] at (4,0){};
        \node[box,fill=Cerulean] at (2,6){};
        \node[box,fill=Cerulean] at (3,5){};
        \node[box,fill=Cerulean] at (4,4){};
        \node[box,fill=Cerulean] at (5,3){};
        \node[box,fill=Cerulean] at (6,2){};

        \node[box,fill=LimeGreen] at (0,3){};
        \node[box,fill=LimeGreen] at (1,2){};
        \node[box,fill=LimeGreen] at (2,1){};
        \node[box,fill=LimeGreen] at (3,0){};
        \node[box,fill=LimeGreen] at (3,6){};
        \node[box,fill=LimeGreen] at (4,5){};
        \node[box,fill=LimeGreen] at (5,4){};
        \node[box,fill=LimeGreen] at (6,3){};

        \node[box,fill=Yellow] at (0,2){};
        \node[box,fill=Yellow] at (1,1){};
        \node[box,fill=Yellow] at (2,0){};
        \node[box,fill=Yellow] at (4,6){};
        \node[box,fill=Yellow] at (5,5){};
        \node[box,fill=Yellow] at (6,4){};

        \node[box,fill=BurntOrange] at (5,6){};
        \node[box,fill=BurntOrange] at (6,5){};
        \node[box,fill=BurntOrange] at (0,1){};
        \node[box,fill=BurntOrange] at (1,0){};

        \node[box,fill=Red] at (6,6){};
        \node[box,fill=Red] at (0,0){};
      \end{tikzpicture}
      \caption{$\chi_{6,7}$ (non-admissible of level $1$)}
    \end{subfigure}
    \caption{Examples of ``rainbow grids''}
    \label{fig:rainbow}
  \end{figure}
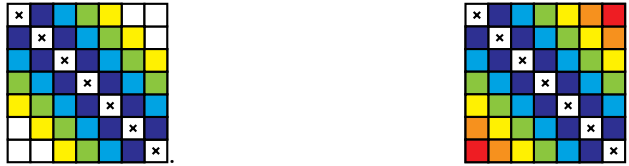
  
  Let $b \le d-3$.
  In the following, we outline a proof that $\chi_{b,d}$ is
  $\bGamma([d],[d])$-admissible of level $1$.
  We proceed by reduction to all-blank partial colourings from
  Lemma~\ref{lem:allblank_Gamma}.
  First, it clearly suffices to consider the case $b = d-3$.
  For $d=3$, $\chi_{b,d} = \chi_{0,3}$ is the all-blank partial colouring.
  Let $d \ge 4$. Let $H\subset [d]$ be non-empty.
  We consider various cases depending on the value of $i := \min(H)$.
  The crucial observation here is that each of the $d-3$ colours appears in
  rows and columns $1$, $2$, $d-1$, and $d$.
  \begin{itemize}
  \item
    Suppose that $i = 1$ or $i = 2$.
    If some $(u,v) \in \mathsf{Corner} := \{ (1,d-1),(1,d),(2,d)\}$ is an isolated
    (necessarily blank) cell of $\cG(H,[d])$,
    then each cell of $\cG(H,[d]\setminus\{v\})$ is blank and the proof of
    Lemma~\ref{lem:allblank_Gamma} finishes this case.
    Otherwise, $d\in H$ and we choose $D(H) := \{ d\}$.
    Both $(d,1)$ and $(d,2)$ are isolated blank cells of $\cG(H,[d-1])$.
    Thus, $(H,[d]\setminus D(H)) \xto[\bGamma,\chi_{d-3,d}]* (H,\{3,\dotsc,d-1\})$.
    Clearly, all cells with colour $c_1$ in $\cG([d],[d])$ are blank in
    $\cG(H,\{3,\dotsc,d-1\})$.
    We conclude that one of $(1,d-1)$ and $(2,3)$ is an isolated blank cell of
    $\cG(H,\{3,\dotsc,d-1\})$.
    By repeatedly using such isolated blank cells,
    it now easily follows that $(H,\{3,\dotsc,d-1\}) \xto[\bGamma,\chi_{d-3,d}]* (H,\emptyset)$.
  \item
    Suppose that $i\ge 3$.
    Then all cells with one of the colours $c_1,\dotsc,c_{i-2}$ in
    $\cG([d],[d])$ are blank in $\cG(H,[d])$.
    Hence, $(i,i-1)$ is an isolated blank cell of $\cG(H,[d])$
    and $c_{i-1}$ is absent from $\cG(H,[d]\setminus\{i-1\})$.
    Continuing in this fashion, we obtain $(H,[d])
    \xto[\bGamma,\chi_{d-3,d}]* (H,\{i,\dotsc,d\})$ and
    all cells of $\cG(H,\{i,\dotsc,d\})$ are blank.
    We may thus again proceed as in the proof of Lemma~\ref{lem:allblank_Gamma}.
\end{itemize}
\end{ex}

\begin{rem}
  For $\bSigma$, let $\chi'_{b,d}$ be the partial colouring of $\sS([d],[d])$
  (see Example~\ref{ex:Sigma})
  which assigns the same colours as $\chi_{b,d}$ to off-diagonal entries of
  the grid $[d]\times [d]$ and which is blank along the diagonal.
  Using Example~\ref{ex:allblank_adm} in place of
  Lemma~\ref{lem:allblank_Gamma}, a variation of our arguments from above
  shows that $\chi'_{b,d}$ is $\bSigma([d],[d])$-admissible whenever $b
  \le d-3$.
\end{rem}

\subsection{Proofs of Corollary~\ref{cor:asym}--\ref{cor:sym}}

\paragraph{Generalities.}
Our proofs of Corollary~\ref{cor:asym}--\ref{cor:sym} share several common
steps. We will therefore initially consider both cases at once.
Recall the definition of $\bGamma$ and $\bSigma$ from
Examples~\ref{ex:graph}--\ref{ex:Sigma}, including the definitions of the sets
$\sS(I,J)$ and $\sE(I,J)$.
Further recall the descriptions of the associated grids from
Example~\ref{ex:basic3} and of isolated cells from
Example~\ref{ex:isolated3}.

\begin{lemma}
  \label{lem:sym_vs_rec}
  Let $U,W \in \Powf(\NN)$ with $U\cap W = \emptyset$.
  Let $\hat\beta\colon \sS(U,W) \to \Coloursb$ (resp.~$\hat\beta\colon
  \sE(U,W)\to\Coloursb$) be a partial colouring.
  Let
  $\beta\colon U\times W \to \Coloursb$
  (resp.~$\beta\colon U \odot W \to \Coloursb$)
  be the induced partial colouring on
  the associated grid w.r.t.\ $\bSigma$ (resp.~$\bGamma$).
  Let $W'\subset W$.
  Then $(U,W) \xto[\bRho,\beta]* (U,W')$ if and only if
  $(U,W) \xto[\bSigma,\hat\beta]* (U,W')$
  (resp.~$(U,W) \xto[\bGamma,\hat\beta]* (U,W')$).
\end{lemma}
\begin{proof}
  As $U\cap W = \emptyset$, each cell of $U\times W$ (resp.\ $U\odot W$)
  is isolated w.r.t.\ $\bSigma$ (resp.~$\bGamma$).
  The claim follows since
  by Lemma~\ref{lem:gridcol}, $\xto[\bSigma,\hat\beta]{}$
  (resp.\ $\xto[\bGamma,\hat\beta]{}$) and $\xto[\bRho,\beta]{}$ coincide when
  restricted to pairs of disjoint sets.
\end{proof}

\begin{defn}
  \label{d:betahat}
Let $I,J\in \Powf(\NN)$ with $I\cap J = \emptyset$.
Write $V := I\cup J$.
Let $\beta\colon I\times J \to \Coloursb$ be a partial colouring.
Define $\hat\beta\colon \sS(V,V) \to \Coloursb$
(resp.~$\hat\beta\colon \sE(V,V) \to \Coloursb$)
as follows:
\[
  x \hat\beta = \begin{cases}
    (i,j)\beta, & \text{if $x= \{ i,j\}$ for $i\in I$ and $j\in J$},\\
    \blank, & \text{otherwise}.
  \end{cases}
\]
As before, we also let $\hat\beta$ denote the induced partial colouring of
the associated grid $\cG(V,V) = V\times V$ (resp.\ $\cG(V,V) = V\odot V$)
w.r.t.\ $\bSigma$ (resp.\ $\bGamma$).
That is, $(i,j) \hat\beta = (j,i)\hat\beta = (i,j) \beta$ for $i\in I$ and
$j\in J$ and $x\hat\beta = \blank$ for $x\not\in (I \times J) \cup (J\times
I)$.
\end{defn}

\begin{ex}
  Let $I = \{1,2\}$, $J = \{3,4,5\}$, and let $\beta$ be the
  $\bRho(I,J)$-admissible partial colouring of $I\times J$ given by
  \[
    \begin{tikzpicture}
      [box/.style={rectangle,draw=black,thick, minimum size=1cm},
      scale=0.4, every node/.style={scale=0.4}
      ]
      \foreach \x in {0,1,2}{
        \foreach \y in {0,1}
        \node[box,fill=white] at (\x,\y){};
      }
      \node[box,fill=ourdark] at (0,0){};  
      \node[box,fill=ourdark] at (2,1){};

      \node[box,fill=YellowOrange] at (0,1){};  
      \node[box,fill=YellowOrange] at (1,0){};  
      \node[box,fill=YellowOrange] at (1,1){};  

      \node at (-1,0) {2};
      \node at (-1,1) {1};

      \node at (0,2) {3};
      \node at (1,2) {4};
      \node at (2,2) {5};
    \end{tikzpicture}.
  \]
  Then the induced partial colouring $\hat\beta$ on the grid $(I\cup J)\times
  (I\cup J) = [5]\times [5]$ w.r.t.\ $\bSigma$ is
  \[
    \begin{tikzpicture}
      [box/.style={rectangle,draw=black,thick, minimum size=1cm},
      scale=0.4, every node/.style={scale=0.4}
      ]
      \foreach \x in {0,1,2,3,4}{
        \foreach \y in {0,1,2,3,4}
        \node[box,fill=white] at (\x,\y){};
      }
      \node[box,fill=ourdark] at (2,3){};  
      \node[box,fill=ourdark] at (4,4){};

      \node[box,fill=YellowOrange] at (2,4){};  
      \node[box,fill=YellowOrange] at (3,4){};  
      \node[box,fill=YellowOrange] at (3,3){};  

      \node[box,fill=ourdark] at (0,0){};  
      \node[box,fill=ourdark] at (1,2){};

      \node[box,fill=YellowOrange] at (0,2){};  
      \node[box,fill=YellowOrange] at (0,1){};  
      \node[box,fill=YellowOrange] at (1,1){};  

      \node at (-1,0) {5};
      \node at (-1,1) {4};
      \node at (-1,2) {3};
      \node at (-1,3) {2};
      \node at (-1,4) {1};

      \node at (0,5) {1};
      \node at (1,5) {2};
      \node at (2,5) {3};
      \node at (3,5) {4};
      \node at (4,5) {5};
    \end{tikzpicture}.
  \]
  By deleting the diagonal cells from this grid, we obtain the colouring
  associated with $\beta$ on the grid $[5]\odot [5]$ w.r.t\ $\bGamma$.
  It is of course no coincidence that the construction of $\hat\beta$ from
  $\beta$ is reminiscent of the definitions of $\SymBoard_{d\times e}$ and $\AltBoard_{d\times
    e}$ in terms of $\Board_{d\times e}$ in \S\ref{ss:intro_relations}.
\end{ex}

\paragraph{Towards Corollary~\ref{cor:sym}.}
Let the notation and $\hat\beta\colon \sS(V,V) \to \Coloursb$ be as in
Definition~\ref{d:betahat}.
The following is the last missing piece towards Corollary~\ref{cor:sym}.

\begin{prop}
  \label{prop:embedded_Sigma}
  $\hat\beta$ is $\bSigma(I\cup J,I\cup J)$-admissible if and only if
  $\beta$ is $\bRho(I,J)$-admissible.
\end{prop}
\begin{proof}
  Let $\emptyset \not= H\subset V$.
  Write $H = I' \cup J'$ for $I'\subset I$ and $J'\subset J$.
  For all $i'\in I'$ (resp.~$j'\in J'$), the cell $(i',i')$ (resp.~$(j,j')$)
  is isolated and blank within $\cG(H,V)$.
  We can thus delete all columns in $I'\cup J'$ to obtain
  $(H,V) \xto[\bSigma,\hat\beta]* (H,V\setminus H)$.
  Next, each cell in $I'\times (I\setminus I')$ or $J'\times(J\setminus J')$ is
  $(H,V \setminus H)$-isolated and blank.
  Thus, if $I'\not= \emptyset \not= J'$, then $(H,V\setminus
  H)\xto[\bSigma,\hat\beta]* (H,\emptyset)$.
  We are thus left to consider the cases $J' = \emptyset$ or $I' =
  \emptyset$.

  Suppose that $J' = \emptyset$ so that $H = I'$ and $(H,V\setminus H)
  = (I', (I\setminus I') \cup J)
  \xto[\bSigma,\hat\beta]* (I',J)$.
  Using Lemma~\ref{lem:gridcol},
  it is easy to see that $\hat\beta(I',J)$ agrees with $\beta(I',J)$
  on $I'\times J$.
  By Lemma~\ref{lem:sym_vs_rec},
  for all non-empty $I''\subset I$,
  $(I'',J) \xto[\bSigma,\hat\beta]* (I'',\emptyset)$  if and only if
  $(I'',J) \xto[\bRho,\beta]* (I'',\emptyset)$.

  Finally, suppose that $I' = \emptyset$ so that $H = J'$
  and $(H,V\setminus H) \xto[\bSigma,\hat\beta]* (J',I)$.
  In this case, $\hat\beta(J',I)$ agrees with $\beta^\top(J',I)$
  on $J'\times I$.
  Hence, for all non-empty $J''\subset I$,
  $(J'',I) \xto[\bSigma,\hat\beta]* (J'',\emptyset)$
  if and only if
  $(J'',I) \xto[\bRho,\beta^\top]* (J'',\emptyset)$.
  The claim thus follows from Lemma~\ref{lem:transpose_adm}.
\end{proof}

Corollary~\ref{cor:sym} follows by combining
Proposition~\ref{prop:embedded_Sigma} and Corollary~\ref{cor:Gamma_zeta}---the translation between matrices and 
families of module representations is similar to 
the proof of Corollary~\ref{cor:rec} in \S\ref{ss:proof_rec}.

\paragraph{Towards Corollary~\ref{cor:asym}.}
The main difference between Corollaries~\ref{cor:asym} and \ref{cor:sym} is that
our proof of the former will involve admissibility of level $1$ rather than
$0$.
Let the notation and $\hat\beta\colon \sE(V,V) \to \Coloursb$ be as in
Definition~\ref{d:betahat}.

\begin{prop}
  \label{prop:embedded_Gamma}
  If $\beta$ is $\bRho(I,J)$-admissible,
  then $\hat\beta$ is $\bGamma(I\cup J,I\cup J)$-admissible of level $1$.
\end{prop}
\begin{proof}
  Let $H = I' \cup J' \not= \emptyset$ for $I' \subset I$ and $J'\subset J$. 
  Suppose that $I'\not= \emptyset$; analogously to the proof of
  Proposition~\ref{prop:embedded_Sigma},
  using Lemma~\ref{lem:transpose_adm}, the case $J'\not= \emptyset$ 
  of the following is similar.
  Let $i \in I'$ and $D := D(H) := \{ i\}$.
  Within the $i$th row of $\cG(H, V \setminus D)$,
  all cells with columns in $I$ are blank and isolated.
  Thus, $(H,V\setminus D) \xto[\bGamma,\hat\beta]* (H,J)$.
  Similar to the proof of Proposition~\ref{prop:embedded_Gamma},
  within the grid $\cG(H,J)$,
  (a) the induced partial colouring on
  $I'\times J$ coincides with $\beta(I',J)$
  and (b) all cells in $J'\times J$ are blank.
  We consider two cases:
  \begin{enumerate}
  \item
    Suppose that $J'\not= \emptyset$.
    As all cells within the non-empty set $J'\times (J\setminus  J')\subset
    \cG(H,J)$ are blank and isolated, $(H,J) \xto[\bGamma,\hat\beta]* (H,J')$.
    Since $\beta$ is $\bRho(I,J)$-admissible,
    by Lemma~\ref{lem:ex_blank} and Proposition~\ref{prop:two_adm},
    $\cG(I',J')$ contains a blank cell, say $(i',j')$.
    Since $J'\times J'\subset \cG(H,J')$ is entirely blank,
    $(i',j')$ is also a blank and isolated cell of $\cG(H,J')$.
    In particular, $(H,J')\xto[\bGamma,\hat\beta]{} (H,J'\setminus\{j'\})$.
    Within the grid $\cG(H,J'\setminus\{j'\})$, all cells in the $j'$th row are
    blank and isolated whence $(H,J') \xto[\bGamma,\hat\beta]* (H,\emptyset)$.
  \item
    If $J' = \emptyset$, then $(H,J) = (I',J) \xto[\bGamma,\hat\beta]*
    (I',\emptyset)$ by Lemma~\ref{lem:sym_vs_rec}. \qedhere
  \end{enumerate}
\end{proof}

Corollary~\ref{cor:sym} follows by combining
Proposition~\ref{prop:embedded_Gamma}
and Corollary~\ref{cor:Sigma_zeta}
similarly to the proof of Corollary~\ref{cor:rec} in \S\ref{ss:proof_rec}.

\subsection{Proof of Theorem~\ref{thm:embedded}}

As the culmination of the techniques developed in the present article,
the following provides the (admittedly technical) template for all three
results in Theorem~\ref{thm:embedded}. 
Recall the matrix notation for maps from \S\ref{ss:pushouts}.

\begin{thm}
  \label{thm:template}
  Let $\bTheta$ be a combinatorial family of module representations
  over $R$ as in Definition~\ref{d:combfam}.
  Let $I,J\in \Powf(\NN)$ and suppose that $\bTheta$ is $(I,J)$-constant of
  rank $\ell \ge 0$ (Definition~\ref{d:IJconst}).
  Define $\UU$ and $\acute\dtimes$ as in \S\ref{ss:relmod}.
  Write $\sB_\infty := \bigcup_{I,J\in\Powf(\NN)}\sB(I,J)$
  and let $\beta\colon \sB_{\infty}\to\Coloursb$ be a partial colouring;
  we assume that $\sB_\infty\subset \UU$.
  Suppose that $\beta$ is $\bTheta(I,J)$-admissible of level $\ell$ (Definition~\ref{d:adm}).
  Let $\tilde I,\tilde J\in \Powf(\NN)$
  with $I \subset \tilde I$ and $J \subset \tilde J$.
  Let $N$ be a finitely generated $R$-module and let $\eta\colon N\to \Hom(R\tilde
  I,R\tilde J)$ be a module representation.
  Recall the definition of $\infl^{\tilde I,\tilde J}_{I,J}(\dtimes)$ from~\S\ref{s:ask}.
  Define
  \[
  \begin{aligned}
    \sigma & := 
    \begin{bmatrix}
      \inf^{\tilde I,\tilde J}_{I,J}\bigl(\theta(I,J)\sslash \beta\bigr) \\
      \acute\eta
    \end{bmatrix} \!\colon &
    \Board(\sB(I,J) \sslash \beta)\oplus \acute N&\to \Hom(\acute R\tilde I,\acute R\tilde J)
    \text{ and}\\
    \tilde\sigma &:=
    \begin{bmatrix}
      \inf^{\tilde I,\tilde J}_{I,J}\bigl(\acute\theta(I,J)\bigr) \\
      \acute\eta
    \end{bmatrix}\!\colon&
    \acute R \, \sB(I,J) \oplus \acute N &\to \Hom(\acute R\tilde I,\acute R\tilde J).
  \end{aligned}
  \]
  Let $\fO$ be an $\acute R$-algebra which is a compact \DVR{}.
  Then
  $\Zeta^{\ak}_{\sigma^\fO}(T) = \Zeta^{\ak}_{\tilde\sigma^\fO}(T)$.
\end{thm}
\begin{proof}
  By Corollary~\ref{cor:adm_blueprint}\ref{cor:adm_blueprint2},
  $\theta(I,J)\sslash\beta$ is an orbital subrepresentation of
  $\acute\theta(I,J)$.
  Lemmas~\ref{lem:orbital_enlarge}--\ref{lem:orbital_sum} thus show that
  $\sigma$ is an orbital subrepresentation of $\tilde\sigma$.
  The claim now follows from Lemma~\ref{lem:orbital_same_zeta}.
\end{proof}

Let $I',J'\in \Powf(\NN)$ and let $\beta\colon I'\times J'\to \Coloursb$ be
$\bRho(I',J')$-admissible.

\begin{enumerate}[label=(\alph{*})]
\item
  \label{pf_embedded1}
  Recall that $\bRho$ is $(I',J')$-constant of rank $0$
  (Example~\ref{ex:Rho_crk}).
  We may thus apply Theorem~\ref{thm:template} with $\bTheta = \bRho$,
  $I = I'$, $J=J'$, and $\ell = 0$.
\item
  \label{pf_embedded2}
  For $\bTheta = \bSigma$,
  suppose that $I'\cap J' = \emptyset$.
  Given $\beta$, define $\hat\beta$ as in Definition~\ref{d:betahat}.
  By Proposition~\ref{prop:embedded_Sigma},
  $\hat\beta$ is $\bSigma(I' \cup J',I' \cup J')$-admissible.
  By Example~\ref{ex:Sigma_crk}, $\bSigma$ is $(I'\cup J',I'\cup J')$-constant of
  rank $0$.
  We may thus apply Theorem~\ref{thm:template} with $\bTheta = \bSigma$,
  $I = J = I'\cup J'$, and $\ell = 0$.
\item
  \label{pf_embedded3}
  For $\bTheta = \bGamma$,
  again suppose that $I'\cap J' = \emptyset$ and
  define $\hat\beta$ as in Definition~\ref{d:betahat}.
  By Proposition~\ref{prop:embedded_Gamma},
  $\hat\beta$ is $\bGamma(I' \cup J',I' \cup J')$-admissible of level $1$.
  Suppose that $I' \cup J' \not= \emptyset$.
  By Corollary~\ref{cor:Gamma_crk}, $\bGamma$ is $(I'\cup J',I'\cup
  J')$-constant of rank $1$.
  We may thus apply Theorem~\ref{thm:template} with $\bTheta = \bGamma$,
  $I = J = I'\cup J'$, and $\ell = 1$.
\end{enumerate}

The preceding points
\ref{pf_embedded1}--\ref{pf_embedded3}
imply Theorem~\ref{thm:embedded}.
Indeed, by permuting rows and columns in Theorem~\ref{thm:embedded},
we may assume that $r_i = i$ and $c_j = j$ for all
$1\le i\le m$ and $1\le j\le n$.
By applying \ref{pf_embedded1} with $R = \fO$, $I = [d]$, $J = [e]$,
$\tilde I = [\tilde m]$, $\tilde J = [\tilde n]$,
and $\eta = \bigl(N\incl \Mat_{\tilde m\times \tilde n}(\fO)\bigr)$, we obtain
the case $M = \Mat_{d\times e}(\fO)$ and $M' = \Board_{d\times e}(\cA,u,\fO)$ of
Theorem~\ref{thm:embedded};
the final translation from module representations to matrices is again based
on Example~\ref{ex:abstract_intro_relations}.
The other two cases in Theorem~\ref{thm:embedded} follow very similarly using
\ref{pf_embedded2}--\ref{pf_embedded3}.

{
  \def\emph{\itemph}
  \bibliographystyle{abbrv}
  \tiny
  \phantomsection
  \addcontentsline{toc}{section}{References}
  \bibliography{board}

\begin{bibdiv}
\begin{biblist}

\bib{Bae38}{article}{
      author={Baer, Reinhold},
       title={Groups with abelian central quotient group},
        date={1938},
     journal={Trans. Amer. Math. Soc.},
      volume={44},
      number={3},
       pages={357\ndash 386},
}

\bib{BB03}{article}{
      author={Belkale, Prakash},
      author={Brosnan, Patrick},
       title={Matroids, motives, and a conjecture of {K}ontsevich},
        date={2003},
     journal={Duke Math. J.},
      volume={116},
      number={1},
       pages={147\ndash 188},
}

\bib{BDOP13}{article}{
      author={Berman, Mark~N.},
      author={Derakhshan, Jamshid},
      author={Onn, Uri},
      author={Paajanen, Pirita},
       title={Uniform cell decomposition with applications to {C}hevalley
  groups},
        date={2013},
        ISSN={0024-6107},
     journal={J. Lond. Math. Soc. (2)},
      volume={87},
      number={2},
       pages={586\ndash 606},
         url={http://dx.doi.org/10.1112/jlms/jds056},
}

\bib{Bou70}{book}{
      author={Bourbaki, Nicolas},
       title={Alg{\`e}bre: Chap. 1 {\`a} 3},
   publisher={Hermann},
        date={1970},
}

\bib{BC17}{article}{
      author={Brenti, Francesco},
      author={Carnevale, Angela},
       title={Proof of a conjecture of {K}lopsch-{V}oll on {W}eyl groups of
  type {$A$}},
        date={2017},
     journal={Trans. Amer. Math. Soc.},
      volume={369},
      number={10},
       pages={7531\ndash 7547},
}

\bib{CSV18}{article}{
      author={Carnevale, Angela},
      author={Shechter, Shai},
      author={Voll, Christopher},
       title={Enumerating traceless matrices over compact discrete valuation
  rings},
        date={2018},
     journal={Israel J. Math.},
      volume={227},
      number={2},
       pages={957\ndash 986},
}

\bib{ConHCFT}{misc}{
      author={Conrad, Keith},
       title={History of class field theory},
        note={Available at \url{https://kconrad.math.uconn.edu/blurbs/}},
}

\bib{Singular}{misc}{
      author={Decker, Wolfram},
      author={Greuel, Gert-Martin},
      author={Pfister, Gerhard},
      author={Sch\"onemann, Hans},
       title={{\textsc{Singular}} {4-3-0} --- {A} computer algebra system for
  polynomial computations},
         how={\url{http://www.singular.uni-kl.de}},
        date={2022},
}

\bib{dS05}{article}{
      author={du~Sautoy, Marcus P.~F.},
       title={Counting conjugacy classes},
        date={2005},
        ISSN={1469-2120},
     journal={Bull. London Math. Soc.},
      volume={37},
      number={1},
       pages={37\ndash 44},
         url={http://dx.doi.org/10.1112/S0024609304003637},
}

\bib{Eis95}{book}{
      author={Eisenbud, David},
       title={Commutative algebra},
      series={Graduate Texts in Mathematics},
   publisher={Springer-Verlag},
     address={New York},
        date={1995},
      volume={150},
         url={http://dx.doi.org/10.1007/978-1-4612-5350-1},
}

\bib{EH88}{article}{
      author={Eisenbud, David},
      author={Harris, Joe},
       title={Vector spaces of matrices of low rank},
        date={1988},
     journal={Adv. in Math.},
      volume={70},
      number={2},
       pages={135\ndash 155},
}

\bib{Fit36}{article}{
      author={{Fitting}, Hans},
       title={Die {D}eterminantenideale eines {M}oduls},
        date={1936},
     journal={{Jahresber. Dtsch. Math.-Ver.}},
      volume={46},
       pages={195\ndash 228},
}

\bib{FG15}{article}{
      author={Fulman, Jason},
      author={Goldstein, Larry},
       title={Stein's method and the rank distribution of random matrices over
  finite fields},
        date={2015},
     journal={Ann. Probab.},
      volume={43},
      number={3},
       pages={1274\ndash 1314},
}

\bib{Gar81}{article}{
      author={Garbanati, Dennis},
       title={Class field theory summarized},
        date={1981},
     journal={Rocky Mountain J. Math.},
      volume={11},
      number={2},
       pages={195\ndash 225},
}

\bib{M2}{misc}{
      author={Grayson, Daniel~R.},
      author={Stillman, Michael~E.},
       title={Macaulay2, a software system for research in algebraic geometry},
         how={Available at \url{http://www.math.uiuc.edu/Macaulay2/}},
}

\bib{GSS88}{article}{
      author={Grunewald, Fritz~J.},
      author={Segal, Dan},
      author={Smith, Geoff~C.},
       title={Subgroups of finite index in nilpotent groups},
        date={1988},
     journal={Invent. Math.},
      volume={93},
      number={1},
       pages={185\ndash 223},
}

\bib{LLMPSZ11}{article}{
      author={Lewis, Joel~Brewster},
      author={Liu, Ricky~Ini},
      author={Morales, Alejandro~H.},
      author={Panova, Greta},
      author={Sam, Steven~V.},
      author={Zhang, Yan~X.},
       title={Matrices with restricted entries and {$q$}-analogues of
  permutations},
        date={2011},
     journal={J. Comb.},
      volume={2},
      number={3},
       pages={355\ndash 395},
}

\bib{LW00}{article}{
      author={Linial, Nathan},
      author={Weitz, Dror},
       title={Random vectors of bounded weight and their linear dependencies
  (unpublished manuscript)},
        date={2000},
        note={\url{http://www.drorweitz.com/ac/pubs/rand_mat.pdf}},
}

\bib{Lins1/19}{article}{
      author={Lins~{d}e Araujo, Paula~Macedo},
       title={Bivariate representation and conjugacy class zeta functions
  associated to unipotent group schemes, {I}: Arithmetic properties},
        date={2019},
     journal={J. Group Theory},
      volume={22},
      number={4},
       pages={741\ndash 774},
}

\bib{Lins2/20}{article}{
      author={Lins~{d}e Araujo, Paula~Macedo},
       title={Bivariate representation and conjugacy class zeta functions
  associated to unipotent group schemes,~{II}: groups of type {$F$}, {$G$}, and
  {$H$}},
        date={2020},
     journal={Internat. J. Algebra Comput.},
      volume={30},
      number={5},
       pages={931\ndash 975},
}

\bib{Nor76}{book}{
      author={Northcott, D.~G.},
       title={Finite free resolutions},
   publisher={Cambridge University Press, Cambridge-New York-Melbourne},
        date={1976},
        note={Cambridge Tracts in Mathematics, No. 71},
}

\bib{O'BV15}{article}{
      author={O'Brien, E.~A.},
      author={Voll, C.},
       title={Enumerating classes and characters of {$p$}-groups},
        date={2015},
     journal={Trans. Amer. Math. Soc.},
      volume={367},
      number={11},
}

\bib{OSS11}{book}{
      author={Okonek, Christian},
      author={Schneider, Michael},
      author={Spindler, Heinz},
       title={Vector bundles on complex projective spaces},
      series={Modern Birkh\"{a}user Classics},
   publisher={Birkh\"{a}user/Springer Basel AG, Basel},
        date={2011},
        note={With an appendix by S. I. Gelfand},
}

\bib{Reu93}{book}{
      author={Reutenauer, Christophe},
       title={Free {L}ie algebras},
      series={London Mathematical Society Monographs. New Series},
   publisher={The Clarendon Press, Oxford University Press, New York},
        date={1993},
      volume={7},
        ISBN={0-19-853679-8},
        note={Oxford Science Publications},
}

\bib{1489}{incollection}{
      author={Rossmann, Tobias},
       title={A framework for computing zeta functions of groups, algebras, and
  modules},
        date={2017},
   booktitle={Algorithmic and experimental methods in algebra, geometry, and
  number theory},
   publisher={Springer, Cham},
       pages={561\ndash 586},
}

\bib{ask}{article}{
      author={Rossmann, Tobias},
       title={The average size of the kernel of a matrix and orbits of linear
  groups},
        date={2018},
     journal={Proc. Lond. Math. Soc. (3)},
      volume={117},
      number={3},
       pages={574\ndash 616},
}

\bib{padzeta}{article}{
      author={Rossmann, Tobias},
       title={Computing local zeta functions of groups, algebras, and modules},
        date={2018},
     journal={Trans. Amer. Math. Soc.},
      volume={370},
      number={7},
       pages={4841\ndash 4879},
}

\bib{ask2}{article}{
      author={Rossmann, Tobias},
       title={The average size of the kernel of a matrix and orbits of linear
  groups, {II}: duality},
        date={2020},
     journal={J. Pure Appl. Algebra},
      volume={224},
      number={4},
       pages={28 pages},
}

\bib{Zeta}{manual}{
      author={Rossmann, Tobias},
       title={{\textsf{\textup{Zeta}}, version 0.4.2}},
        date={2022},
        note={See \url{https://github.com/torossmann/Zeta}.},
}

\bib{cico}{article}{
      author={Rossmann, Tobias},
      author={Voll, Christopher},
       title={Groups, graphs, and hypergraphs: average sizes of kernels of
  generic matrices with support constraints},
        date={2022$^+$},
     journal={To appear in Mem. Amer. Math. Soc.},
        note={arXiv:\href{https://arxiv.org/abs/1908.09589}{1908.09589}},
}

\bib{stacks-project}{misc}{
      author={{Stacks Project Authors}, {The}},
       title={\textit{Stacks Project}},
        date={2022},
        note={\url{http://stacks.math.columbia.edu}},
}

\bib{SV14}{article}{
      author={Stasinski, A.},
      author={Voll, C.},
       title={Representation zeta functions of nilpotent groups and generating
  functions for {W}eyl groups of type {$B$}},
        date={2014},
     journal={Amer. J. Math.},
      volume={136},
      number={2},
       pages={501\ndash 550},
}

\bib{SageMath}{manual}{
      author={{The Sage Developers}},
       title={{S}age{M}ath, the {S}age {M}athematics {S}oftware {S}ystem
  ({V}ersion 9.5)},
        date={2022},
        note={See {\url{https://www.sagemath.org}}.},
}

\bib{Vol11}{incollection}{
      author={Voll, Christopher},
       title={A newcomer's guide to zeta functions of groups and rings},
        date={2011},
   booktitle={Lectures on profinite topics in group theory},
      series={London Math. Soc. Stud. Texts},
      volume={77},
   publisher={Cambridge Univ. Press, Cambridge},
       pages={99\ndash 144},
}

\end{biblist}
\end{bibdiv}
}

\vspace*{.5em}

{\small
\noindent
{\footnotesize
\begin{minipage}[t]{0.5\textwidth}
  Angela Carnevale\\
  E-mail: \href{mailto:angela.carnevale@nuigalway.ie}{angela.carnevale@nuigalway.ie}
\end{minipage}
\hfill
\begin{minipage}[t]{0.45\textwidth}
  Tobias Rossmann\\
  E-mail: \href{mailto:tobias.rossmann@nuigalway.ie}{tobias.rossmann@nuigalway.ie}
\end{minipage}

\vspace*{1.1em}

\noindent
School of Mathematical and Statistical Sciences\\
National University of Ireland, Galway \\
Ireland
}
}

\end{document}